\documentclass{article}
\usepackage{amsmath,amsthm,amssymb,amsfonts}
\usepackage{hyperref,lineno}
\usepackage{xcolor}
\usepackage{geometry}
\usepackage{natbib}
\usepackage{verbatim}

\DeclareSymbolFont{newfont}{OML}{cmm}{m}{it}

\hypersetup{
	colorlinks,
	linkcolor={red!50!black},
	citecolor={blue!50!black},
	urlcolor={blue!80!black}
}

\newtheorem{assumption}{Assumption}
\newtheorem{proposition}{Proposition}
\newtheorem{lemma}{Lemma}[section]

\newtheorem{corollary}{Corollary}
\newtheorem{theorem}{Theorem}

\newtheorem{definition}{Definition}

\newcommand{\indep}{\perp\!\!\!\perp}

\newcounter{subsubsubsection}[subsubsection]
\renewcommand{\thesubsubsubsection}
{\thesubsubsection.\arabic{subsubsubsection}}

\newcommand{\subsubsubsection}[1]{%
	\refstepcounter{subsubsubsection}%
	\par\medskip
	\noindent\textbf{\thesubsubsubsection\quad #1}%
	\par\smallskip
}

\title{Coupling and Maximal Inequalities for Graph-Dependent Empirical Processes\footnote{We thank Michael Leung and Yuya Sasaki for useful comments. All potential errors are our own.}}

\author{Mengsi Gao\thanks{USC, Dept. of Economics. email: mengsi.gao@usc.edu}~~and Demian Pouzo\thanks{UC Berkeley, Dept. of Economics. email: dpouzo@berkeley.edu} }

\date{\today}

\begin{document}

	\maketitle
	\thispagestyle{empty}
	
	\begin{abstract}
    We develop maximal inequalities for empirical processes indexed by graph-dependent observations. Our bounds separate the complexity of the indexing class from two features specific to graph dependence: the geometry of the underlying graph and the cost of coupling graph-separated blocks to independent copies. The coupling construction combines a novel graph-adapted dependence coefficient with a coloring of a block partition. As an application, we derive Glivenko--Cantelli results and characterize the associated effective sample size. A central implication is that graph-dependent empirical processes need not exhibit a generic root-$n$ rate: convergence is jointly determined by function-class complexity, graph geometry, and the decay of dependence with graph distance. We specialize the results to graphs with polynomial and exponential growth and to directed dyadic graphs. Finally, we apply the results to network autoregressive models, nonlinear local-propagation models, and treatment-interference settings.
    
	\end{abstract}

	\setcounter{page}{1}

\section{Introduction}

Maximal inequalities are a cornerstone of empirical process theory. They underpin Glivenko--Cantelli and Donsker results (see, e.g., \cite{VdV-W1996}), strong approximations (see, e.g., \cite{DudleyPhilipp1983}), and stochastic equicontinuity, and therefore play a central role in the asymptotic analysis of estimators such as M- and GMM-estimators (see, e.g., \cite{NeweyMcFadden1994}). While this theory is well developed for IID data, extensions to dependent data structures remain more limited. Results are available for time series under appropriate mixing conditions; see, e.g., \cite{ChenShen1998,DMR1995,Pouzo2026,yu1994rates}. For graph-dependent data, however, a comparable general empirical-process theory is largely absent. This gap matters because graph dependence arises naturally in social interactions, peer effects, spillovers, worker--firm mobility, trade, financial and production networks, dyadic data, and modern data-science applications.

The main objective of this paper is to provide tools for developing such a theory. We construct couplings for graph-dependent empirical processes and use them to establish a maximal inequality for processes indexed by a possibly infinite class of functions. Specifically, for a finite graph \(\mathfrak G_n=(N_n,V_n)\), with \(\lvert N_n\rvert=n\), and a class \(\mathcal F\) of measurable functions, we derive a bound of the form
\begin{align}
\left\|
\sup_{f\in\mathcal F}
\frac{1}{n}
\sum_{i\in N_n}
\left\{
f(Z_i)-\mathbb E[f(Z_i)]
\right\}
\right\|_{L^1(P)}
\leq
\operatorname{rate}_n
\times
\mathsf C(\mathcal F).
\label{eqn:intro.maximal}
\end{align}

The complexity term \(\mathsf C(\mathcal F)\) is based on Talagrand's generic chaining functionals under semimetrics induced by the empirical process. This separates the size of the indexing class from the graph-dependent stochastic component, captured by $\operatorname{rate}_{n}$, and permits the use of existing bounds based on entropy, bracketing, VC-type, and smoothness conditions; see, e.g., \cite{talagrand2005,talagrand2014,VdV-W1996,VanHandel2018,VanHandel2018b}.

Under IID sampling, \(\operatorname{rate}_n\) is of order \(n^{-1/2}\). For graph-dependent data, there is no analogous universal rate: it depends on graph geometry and on the decay of dependence. To explain how these features enter our bound, it is useful to begin with the standard time-series argument.

For time series, one partitions the observations into consecutive blocks and divides them into interlaced subcollections, such as odd and even blocks. Within each subcollection, the original blocks are replaced by independent copies, allowing exponential inequalities and chaining arguments to be applied. The approximation error is controlled by mixing coefficients such as the $\beta$-mixing coefficient (see \cite{DMR1995}) or the weaker $\tau$-mixing coefficient (\cite{dedecker2002maximal,dedecker2004coupling,dedecker2005new,doukhan2008weakly,Pouzo2026}).


This time-series construction relies on the temporal ordering, which supplies a prespecified notion of past and future. Networks have no comparable ordering: graph distance identifies well-separated blocks but does not determine the predecessor sigma-field relative to which an independent copy should be constructed. We address this problem by combining block partitions with color classes and introducing a, to our knowledge, novel, graph-adapted notion of \(\tau\)-mixing. Blocks in the same color class are ``well-separated'', while an ordering within each class defines their predecessors. Our coupling theorem then replaces the original blocks by copies that preserve their marginal distributions and are mutually independent within each color class. Moreover, the expected discrepancy between each block and its copy equals the corresponding $\tau$-coefficient and is therefore minimal among all marginal-preserving copies that are independent of the block's predecessors. This optimal-coupling characterization makes the resulting dependence condition practical to verify:  any admissible coupling provides an upper bound. 


This construction determines the graph-dependent component \(\operatorname{rate}_n\) in \eqref{eqn:intro.maximal}. It reflects three features: the coloring cost required to separate the blocks; the within-block geometry and dependence, captured by the distribution of pairwise graph distances and the accumulation of local dependence; and the coupling cost across separated blocks.

We use the maximal inequality to derive, to our knowledge, novel
Glivenko--Cantelli results and to characterize the associated effective
sample size. These results show that graph-dependent empirical processes
need not exhibit a universal root-\(n\) rate: the effective sample size
is determined jointly by graph growth, local dependence, and the decay
of the coupling error. We specialize our bounds to graphs with
polynomial or exponential growth, directed dyadic graphs, and several
other settings. We also verify the coupling conditions for network
autoregressive models, nonlinear local-propagation models, and models
of treatment interference on networks.

\paragraph{Related literature}

Our paper relates first to the literature on spatial and network-dependent asymptotics. Classical results for spatial random fields typically index observations by subsets of a fixed-dimensional Euclidean space; see \cite{guyon1995,jenishprucha2009} for reviews and representative results. This framework does not encompass general graph-dependent data, whose shortest-path geometry need not embed isometrically into a fixed-dimensional Euclidean space and may exhibit very different neighborhood growth.

To bridge this gap, a growing literature develops pointwise laws of large numbers, central limit theorems, and bootstrap procedures for graph- or network-dependent data. \cite{Kuersteiner2019} establishes laws of large numbers and stable central limit theorems under a conditional spatial mixingale condition based on a model-dependent random distance. \cite{ChandrasekharJacksonMcCormickThiyageswaran2024} provide covariance-based sufficient conditions for central limit theorems for general dependent triangular arrays, while \cite{LeungMoon2026} obtain normal approximations in generalized random-geometric graphs under stabilization. \cite{KOJEVNIKOV2021882} adapt the \(\psi\)-dependence framework of \cite{DOUKHAN1999} to networks, controlling covariances between nonlinear functions of graph-separated collections of observations, and \cite{kojevnikov2021} establishes bootstrap validity.

Our objective differs from these pointwise results: we control the empirical process uniformly over a possibly infinite class \(\mathcal F\). This requires separate control of the complexity of \(\mathcal F\), which pointwise covariance bounds do not provide. Methodologically, we control a \(\tau\)-coefficient rather than conditional covariances. This yields a coupling result that permits the use of exponential inequalities and generic chaining while separating function-class complexity, within-block geometry, coloring cost, and dependence across separated blocks.

The literature on uniform results under spatial or network dependence is more limited. Relative to this literature, our main contribution is a maximal inequality for graph-dependent empirical processes, thereby providing a general framework for establishing uniform convergence
for a broad range of function classes and network structures.

\cite{jenishprucha2009} derive a generic uniform law of large numbers
and sufficient conditions for stochastic equicontinuity under
\(\alpha\)- and \(\phi\)-mixing for observations indexed by irregular
subsets of a fixed-dimensional Euclidean space. Our results extend this
type of uniform analysis to general graph-indexed observations, for
which shortest-path geometry need not admit a fixed-dimensional
Euclidean representation. For graph-dependent data, \cite{sasaki2026gmmmestimationnetwork} establishes a uniform law for finite-dimensional parameter-indexed classes under conditional \(\psi\)-dependence by combining a pointwise law with compactness, a high-level equicontinuity condition, and a finite-net approximation. Its condition excludes classes whose modulus of continuity depends on realized data, including quadratic-loss and other unbounded criterion functions. Our maximal inequality can accommodate such classes through direct empirical-process control, as illustrated in Section~\ref{sec:examples}.

Finally, \cite{CaoLeung2025} establish stochastic equicontinuity for
double/debiased machine learning under metric-space dependence. Their
analysis combines \(\beta\)-mixing with neighborhood stability of the
estimated nuisance function to avoid cross-fitting and is tailored to
that particular estimation problem. By contrast, our maximal inequality
is not specific to double/debiased machine learning or to
neighborhood-stable nuisance estimators. It applies to a broader range
of function classes and estimation problems, uses the weaker
Lipschitz-based notion of \(\tau\)-dependence, and explicitly quantifies
how graph geometry affects uniform convergence.

\section{Setup}
\label{sec:setup}

\paragraph{Graph-dependent data} Let \(\mathfrak{G}_n=(N_n,V_n)\) be a finite graph with vertex set \(N_n \) and edge set \(V_n\) such that $|N_{n}| = n$ for $n \in \mathbb{N}$. Let $(Z_i)_{i\in N_n}$ be a collection of random elements each taking values in a Polish space \((\mathsf{Z}, \mathcal Z)\). The data $(Z_i)_{i\in N_n}$ is drawn from $\mathbb P$, a probability measure on   $\mathsf{Z}^{\infty}$. We use $\mathbb E$ to denote the expectation and $P_{Z}$ to denote the marginal distribution of $Z$.

%

%
%
%
%
%


For \(i,j\in N_n\), define the graph distance as the shortest path between them: 
\[
d_n(i,j) : =
\inf\{ k\ge 0:\ \exists (i_0,\dots,i_k),\ i_0=i,\ i_k=j,\ (i_{\ell-1},i_\ell)\in V_n\}.
\]
If no path exists, set \(d_n(i,j)=+\infty\). In addition, by convention, we set $d_{n}(i,i) = 0$. For any $A,B \subseteq N_n$, define the distance between them by $d_n(A,B):=\inf_{i\in A,j\in B} d_n(i,j)$. 

We assume that for any disjoint subsets \(A,B\subseteq N_n\), if $d_n(A,B) = + \infty$ then   $\sigma(Z_i:i\in A) \ \text{is independent of}\ \sigma(Z_j:j\in B)$.

\paragraph{Function class and empirical process} The empirical process is defined as 
\begin{align*}
	f \mapsto  \mathbb G_{n} (f) : = n^{-1/2} \sum_{i \in N_{n}} \{ f(Z_{i})  - \mathbb E[f(Z_{i})] \}
\end{align*}
over a class $\mathcal F$ of real-valued measurable functions on $\mathsf Z$ such that $\mathcal F \ni 0$.\footnote{The requirement \(0\in\mathcal F\) is a normalization used to anchor the chaining argument; it is typically without loss and more generally, the results can be stated for \(\{\mathbb G_n(f)-\mathbb G_n(f_0):f\in\mathcal F\}\), for any fixed \(f_0\in\mathcal F\).}


\subsection{\(\tau\)-mixing on graph-dependent data}
\label{sec:mixing}

We now introduce a notion of $\tau$-mixing for graph-dependent data.
The construction requires two ingredients. The first is a measure of
dependence between a random element and a sigma-field. The second is a
geometric device that identifies collections of blocks that are mutually
well separated in the graph. We use a proper coloring of a block
partition for the latter purpose.

\paragraph{$\tau$-mixing coefficient}
We first recall the definition of the $\tau$-mixing coefficient on a
generic metric space given by \cite{dedecker2006inequalities}. Let
$(\mathsf X,\mathcal X)$ be a Polish space equipped with a
pseudometric $d_{\mathsf X}$, let $X$ be an
$\mathsf X$-valued random element, and let $\mathcal M$ be a
sub-sigma-field. Suppose that, for some $x_0\in\mathsf X$, $\mathbb E\left[
	d_{\mathsf X}(X,x_0) 	\right] 	< 	\infty$.  The $\tau$-coefficient between $\mathcal M$ and $X$ is\footnote{$\operatorname{Lip}_1(\mathsf X,d_{\mathsf X})$ is the class
		of measurable functions $g:\mathsf X\to\mathbb R$ satisfying $ |g(x)-g(x')|
		\le
		d_{\mathsf X}(x,x')$ for all $x,x'\in\mathsf X$. }
\begin{align}
	\tau_{d_{\mathsf X}}(\mathcal M,X)
	:=
	\left\|
	\sup_{g\in\operatorname{Lip}_1(\mathsf X,d_{\mathsf X})}
	\left|
	\mathbb E[g(X)\mid\mathcal M]
	-
	\mathbb E[g(X)]
	\right|
	\right\|_{L^1(\mathbb P)}.
	\label{eq:tau-generic}
\end{align}


This notion is weaker than measuring discrepancy
against all bounded measurable functions, as in the usual
$\beta$-mixing coefficient. 
The weaker Lipschitz requirement is useful
for processes for which total-variation-based dependence does not decay,
but Wasserstein-type dependence does decay with separation; for instance, for processes with discrete outcomes, the $\beta$-mixing coefficient remains bounded away from zero, whereas the $\tau$-mixing coefficient can be made to be arbitrarily close to zero provided $\mathcal{M}$ and $X$ are ``sufficiently separated'' (see \cite{dedecker2002maximal,dedecker2004coupling,dedecker2005new,dedecker2006inequalities} for formal results and discussion in the time series context).

\paragraph{Block partitions and graph colorings} In this paper, we apply the notion of $\tau$-mixing to blocks of graph-dependent random variables of the form $Z_{B} := (Z_{i_1},\dots,Z_{i_m}) \in \mathsf{Z}^{m}$ for some blocks  \(B=\{i_1,\dots,i_m\}\subseteq N_n\). Unlike time
series, where there is a clear direction in terms of past and future, graphs do not supply a canonical temporal ordering. To tackle this issue, we employ the notion of coloring classes.\footnote{Definition \ref{def:color} is a small refinement of the standard definition of proper coloring and chromatic number; e.g. see \cite{bondy2008graph}.}


\begin{definition}
	\label{def:color}
	For any $M \in \mathbb{N}$, let $\mathcal P_{M}=\{B_1,\ldots,B_M\}$ denote an $M$-partition of $N_n$. 
	For any $r \in \mathbb{N}$, a proper $r$-coloring associated to an $M$-partition $\mathcal P_{M}$ is a function $c: [M] \to \mathbb{N}$ such that
		\begin{align}
		\forall m,m' \colon c(m)=c(m')~and~m\neq m'
		\Longrightarrow
		d_n(B_m,B_{m'})
		>
		r.
		\label{eq:proper-coloring}
	\end{align}
The $r$-chromatic number is the smallest possible number of distinct colors, i.e.,
\begin{align}	\label{eq:r-chromatic-number}
 K_{r} : = 	K_{r}(\mathcal P_M) : = \min_{c} |c([M])|
\end{align}
where the min is taken over all possible $r$-colorings of $\mathcal P_{M}$.

The sets $C_{k} : = c^{-1}(k) = \{ m \in [M] \colon c(m) = k \}$ for all $k \in [K_{r}]$ are the color classes --- the collection of these sets is denoted as $\mathcal C_{r,M} : = \mathcal C_{r}[\mathcal P_{M}]$.
%
%
%
\end{definition}

Each color class \(C_k\) indexes blocks  that are pairwise separated by graph
distance strictly greater than $r$. Elements of $C_{k}$ are in $\mathbb{N}$ and thus completely ordered:\footnote{The labeling of the blocks induces the within-color order used in \eqref{eq:ordered-color-class}; it does not represent a causal or geometric direction in the graph. Although other orderings could produce smaller coupling costs, we fix this ordering throughout.} 
\begin{align}
C_k
	=
	\left\{
	m_{k,1},
	\ldots,
	m_{k,L_k}
	\right\},~with~
	m_{k,1}
	<
	\cdots
	<
	m_{k,L_k}.
	\label{eq:ordered-color-class}
\end{align}

For every $k\in[K_{r}]$ and $\ell\in[L_k]$, define the within-class
predecessor set of the $\ell$-th block by
\begin{align}
	\operatorname{Pred}_{\ell}( C_k)
	:=
	\bigcup_{1\leq u<\ell}
	B_{m_{k,u}},
	\label{eq:within-class-predecessor-set}
\end{align}
with the convention that $\operatorname{Pred}_{1}( C_k)
	=
	\varnothing$. 
Thus,
$\operatorname{Pred}_{\ell}( C_k)$ is the union of the blocks
in $ C_k$ that precede $B_{m_{k,\ell}}$ within coloring class in
\eqref{eq:ordered-color-class} and are ``well-separated'' in the sense that, if $\ell\ge2$, $d_n\left(
	\operatorname{Pred}_{\ell}( C_k),
	B_{m_{k,\ell}}
	\right)
	>
	r$. 

\paragraph{Coloring-adapted $\tau$-coefficients}
Let $\mathcal B\subseteq \operatorname{span}\mathcal F$. For every $m\in\mathbb N$,
equip $\mathsf Z^m$ with the pseudometric
\begin{align}
	d_{\mathcal B,m}(z,z')
	:= 	\sup_{f\in\mathcal B}
	\sum_{s=1}^{m}
	\left|
	f(z_s)-f(z'_s)
	\right|,
	\qquad
	z,z'\in\mathsf Z^m.
	\label{eq:block-pseudometric}
\end{align}

\begin{definition}[Coloring-adapted average $\tau$-coefficient]
	Given a set $\mathcal B \subseteq \operatorname{span}\mathcal F$ and a partition-coloring pair $\left\{
		\mathcal P_{M},
		\mathcal C_{r}[\mathcal P_{M}]
		\right\}$, 
	define the coloring-adapted average $\tau$-coefficient by
	\begin{align}
	\bar{\tau}_{\mathcal B}
	\left(
	\mathcal P_{M},
	\mathcal C_{r}[\mathcal P_{M}]
	\right)
	:=
	\frac{1}{n}
	\sum_{k=1}^{K_{r}}
	\sum_{\ell=1}^{L_k}
	\tau_{d_{\mathcal B,|B_{m_{k,\ell}}|}}
	\left(
	\sigma\left(
	Z_j:
	j\in\operatorname{Pred}_{\ell}( C_k)
	\right),
	Z_{B_{m_{k,\ell}}}
	\right). 
	\label{eq:average-coloring-adapted-tau}
\end{align}
\end{definition}

The coefficient inside the sums in
\eqref{eq:average-coloring-adapted-tau} is the block-level
$\tau$-dependence between $Z_{B_{m_{k,\ell}}}$ and the sigma-field
generated by its within-class predecessors. 

The coefficient $\bar{\tau}_{\mathcal B}$ is therefore an average of
the actual coupling costs generated by the given partition and
coloring. This average is the natural aggregation for the goal of the present paper that is to provide a coupling
argument used for empirical processes --- this claim becomes apparent in the proof of Theorem \ref{thm:maximal} where the
block-specific coupling errors are summed over the partition.

A perhaps more widely applicable measure, not tailored to empirical processes, is one that does not rely on the average and provides a bound uniform over color classes:
\begin{align}
	\tau^{\mathrm{max}}_{\mathcal B}
	\left(
	\mathcal P_{M},
	\mathcal C_{r}[\mathcal P_{M}]
	\right)
	:=
	\max_{k\in[K_{r}]}
	\max_{\ell\in[L_k]}
	\frac{1}{|B_{m_{k,\ell}}|}
	\tau_{d_{\mathcal B,|B_{m_{k,\ell}}|}}
	\left(
	\sigma\left(
	Z_j:
	j\in\operatorname{Pred}_{\ell}( C_k)
	\right),
	Z_{B_{m_{k,\ell}}}
	\right).
	\label{eq:max-coloring-adapted-tau}
\end{align}

Finally, yet another alternative measure is given by a version that is uniform over every $r$-coloring with maximal size block given by $q$,
\begin{align}
	\tau_{\mathcal B}(r,q)
	:=
	\sup_{\substack{
			A,B\subseteq N_n:
			A\neq\varnothing,\,
			B\neq\varnothing,\,
			d_n(A,B)\ge r,\,
			|B|\le q
	}}
	\frac{1}{|B|}
	\tau_{d_{\mathcal B,|B|}}
	\left(
	\sigma(Z_A),
	Z_B
	\right).
	\label{eq:global-graph-tau}
\end{align}

For every $q\ge\max_{m\in[M]}|B_m|$, the preceding coefficients
satisfy
\begin{align}
	\bar{\tau}_{\mathcal B}
	\left(
	\mathcal P_{M},
	\mathcal C_{r}[\mathcal P_{M}]
	\right)
	\le
	\tau^{\mathrm{max}}_{\mathcal B}
	\left(
	\mathcal P_{M},
	\mathcal C_{r}[\mathcal P_{M}]
	\right)
	\le
	\tau_{\mathcal B}(r,q).
	\label{eq:tau-coefficient-comparison}
\end{align}

The first inequality follows because
$\bar{\tau}_{\mathcal B}$ is a weighted average of the normalized
block-level coefficients appearing in
$\tau^{\mathrm{max}}_{\mathcal B}$. The second follows because the
predecessor set
$\operatorname{Pred}_{\ell}(C_k)$ and the block
$B_{m_{k,\ell}}$ are separated by graph distance strictly greater
than $r$.\footnote{The term corresponding to $\ell=1$ is zero because its
predecessor sigma-field is trivial.}

The three coefficients differ in the extent to which they are tailored to the coupling construction. The coefficient $\bar{\tau}_{\mathcal B}$ is the average of the block-specific coupling costs induced by the chosen partition, coloring, and within-class ordering; it is therefore the sharpest quantity for our maximal inequality. The coefficient $\tau^{\max}_{\mathcal B}$ replaces this average by a uniform bound over the blocks of the same construction. Finally, $\tau_{\mathcal B}(r,q)$ is uniform over all graph-separated pairs of sets with second component of cardinality at most $q$. It is consequently the most convenient primitive condition, but may be conservative because it ignores the particular predecessor sets generated by the coloring.


\paragraph{Coupling}  By \cite{dedecker2006inequalities}, the $\tau$-dependence coefficient admits an optimal-transport interpretation through minimal couplings, which is useful for the implementation of our results. Formally,
\[
\tau_{d_{\mathsf{X}}}(\mathcal M,X)
=
\inf \Big\{   \mathbb E [ d_{\mathsf{X}}(X',X) ] \colon 
X' \stackrel{d}{=} X,\;
X' \indep \mathcal M
\Big\}.
\]
That is, the $\tau$-coefficient in (\ref{eq:tau-generic}) is the minimal transportation cost needed to replace $X$ by an independent copy $X'$ while preserving its marginal distribution. So, letting $$\mathcal M^{m_{k,\ell} - 1}_{C_{k}} : = 	\sigma\left(
Z_j:
j\in\operatorname{Pred}_{\ell}( C_k)
\right),$$ 
{\small{\begin{align}\label{eqn:tau.coupling.main.1}
		\bar{\tau}_{\mathcal B}
	\left(
	\mathcal P_{M},
	\mathcal C_{r}[\mathcal P_{M}]
	\right)
	=
	\frac{1}{n}
	\sum_{k=1}^{K_{r}}
	\sum_{\ell=1}^{L_k} \inf \Big\{   \mathbb E [ d_{\mathcal B,|B_{m_{k,\ell}}| } (X',	Z_{B_{m_{k,\ell}}} ) ] \colon 
	X' \stackrel{d}{=} 	Z_{B_{m_{k,\ell}}} ,\;
	X' \indep 	\mathcal M^{m_{k,\ell}-1}_{C_{k}} 
	\Big\}.
\end{align}}}
A similar result holds for $\tau^{\mathrm{max}}_{\mathcal B}$ and also 
\begin{align}\label{eqn:tau.coupling.main.2}
	\tau_{\mathcal{B}}(r,q) =	\sup_{\substack{A,B\subseteq N_n\\ d_n(A,B)\ge r,\; |B| \le q}} \frac{1}{|B|} 	\inf \Big\{
 \mathbb E\big[	d_{\mathcal{B},|B|} (Z_B,Z'_B)  \big] 
	\colon 
	Z'_B \stackrel{d}{=} Z_B,\;
	Z'_{B} \indep \sigma(Z_A)
	\Big\}.
\end{align}

\section{Coupling Result}
\label{sec:coupling}

This section presents a coupling result that replaces $(Z_i)_{i\in N_n}$ by a
coupled process $(Z_i^\ast)_{i\in N_n}$ whose blocks are mutually independent
within each color class, while providing explicit and uniform control of the
resulting coupling error. Henceforth, let $\Pr$ denote the joint law of the original and coupled processes, and let $\mathbb E_{\Pr}$ denote the expectation under
this law.

\begin{theorem}\label{thm:coupling}
	Let $\mathcal{B} \subseteq \operatorname{span} \mathcal{F}$ and $\mathcal P_{M} : = (B_m)_{m=1}^M$ be an $M$-partition of $N_n$, and let $\mathcal C_{r} : = (C_{1},...,C_{K_r})$  be a proper $r$-coloring of $(B_m)_{m=1}^M$.	Then there exists an extension of the probability space and, on that extension,
	a collection $(Z_i^\ast)_{i\in N_n}$ with the following properties:
	
	\begin{enumerate}
		\item For every $m\in [M] $, $Z_{B_m}^\ast \stackrel{d}{=} Z_{B_m}$ with  $Z_{B_m}^\ast := (Z_i^\ast)_{i\in B_m}$ and  $Z_{B_m} := (Z_i)_{i\in B_m}$.
		
		\item For every $k\in [K_r]$, the family $\bigl(Z_{B_m}^\ast\bigr)_{m\in C_k}$	is mutually independent.
		
		\item For every $k\in [K_r]$ and every $m\in C_k$,
		\begin{align}
			\mathbb E_{\Pr}\!\left[
			d_{\mathcal{B},|B_m|}\!\left(Z_{B_m}^\ast,Z_{B_m}\right)
			\right]	 =		&	\inf \Big\{
			\mathbb E\big[	d_{\mathcal{B},|B_m|}(Z_{B_{m}},X')\big]
			:\;
			X' \stackrel{d}{=} Z_{B_{m}},\;
			X' \indep	\mathcal{M}^{m-1}_{C_{k}} \label{eqn:min.coupling.0} 
			\Big\} \\
			= & 	\tau_{d_{\mathcal B,|B_{m}|}}
			\left(
			\mathcal{M}^{m-1}_{C_{k}}, 
			Z_{B_{m}}
			\right), \label{eqn:min.coupling}
		\end{align}
				\begin{align}\label{eqn:tau.bound.max}
			\mathbb E_{\Pr}\!\left[
			d_{\mathcal{B},|B_m|}\!\left(Z_{B_m}^\ast,Z_{B_m}\right)
			\right]
			\le |B_{m}| \tau_{\mathcal{B}}(r,|B_{m}|),
		\end{align}
		and 
		\begin{align}\label{eqn:average.min.coupling}
		n^{-1} \sum_{k=1}^{K_{r}} \sum_{m \in C_{k}} \mathbb E_{\Pr} \left[ d_{ \mathcal{B},|B_{m}|}(Z^{\ast}_{B_{m}},Z_{B_{m}}) \right] =  \bar{\tau}_{ \mathcal B}(\mathcal P_{M}, \mathcal{C}_{r}). 
		\end{align}
	\end{enumerate}
\end{theorem}

\begin{proof}
	See Appendix \ref{sec:coupling.proof}.
\end{proof}

Theorem \ref{thm:coupling} provides a blockwise decoupling device for graph-dependent data. Each coupled block \(Z_{B_m}^{\ast}\) preserves the distribution of \(Z_{B_m}\), while the coupled blocks within each color class are mutually independent. Moreover, the expected blockwise discrepancy is the corresponding \(\tau\)-coefficient, and its average across the partition equals
\(\bar\tau_{\mathcal B}(\mathcal P_M,\mathcal C_r[\mathcal P_M])\), as shown in \eqref{eqn:average.min.coupling}.

This expression in part 3 is the implication we used to obtain maximal inequalities. The LHS presents the relevant notion of distance between the original and coupled processes implied by the empirical processes. The RHS is our $\tau$-dependence coefficient defined in \eqref{eq:average-coloring-adapted-tau}. Expression \eqref{eqn:min.coupling} is the basis of \eqref{eqn:average.min.coupling} and is a direct consequence of the minimal coupling representation of the $\tau$-coefficient (see \eqref{eqn:tau.coupling.main.1}-\eqref{eqn:tau.coupling.main.2} above). It implies that, for each color class $C_k$ and each $m\in C_k$, the variable $Z_{B_m}^\ast$ is constructed as a minimizer of this problem with $\mathcal M = \mathcal M^{m-1}_{C_k}$ and $X = Z_{B_m}$.

The minimal-coupling representation also provides a practical route for bounding the coupling error. 
Although the infimum defining the $\tau$-coefficient may be difficult to evaluate directly, any admissible coupling yields an upper bound. Thus, in applications, it is enough to construct a copy of the current block that preserves its marginal distribution, is independent of the relevant predecessor $\sigma$-field, and remains close to the original block. The remaining question is how such couplings can be constructed in concrete models. The standard approach for time series provides the relevant template and, as we explain next, the coloring and predecessor construction allows that approach to be adapted to graph-dependent data.

\paragraph{From time series to graph-dependent data}
The canonical approach to bounding \(\tau\)- and \(\beta\)-mixing coefficients in causal time-series models uses a model-specific coupling. Let \(\mathcal M_t\) denote the sigma-field generated by observations through time \(t\). To control the dependence between \(\mathcal M_t\) and \(Z_{t+h}\), one constructs \(Z_{t+h}'\) satisfying $Z_{t+h}'\overset{d}{=}Z_{t+h}$, and $Z_{t+h}'\perp\!\!\!\perp\mathcal M_t$, so that
\[
\tau_d(\mathcal M_t,Z_{t+h})
\leq
\mathbb E\left[
d\left(Z_{t+h},Z_{t+h}'\right)
\right].
\]
Typically, \(Z_{t+h}'\) is obtained by replacing the primitive innovations or initial state carrying information from the past with independent copies. Bounding \(\tau\) then reduces to controlling how rapidly the effect of this replacement dissipates with \(h\); see, e.g., \cite{dedecker2004coupling,dedecker2002maximal} for examples of this approach.

Graphs lack an ordering that determines the relevant predecessor sigma-field. Our coloring construction supplies this missing structure: blocks within each color class are well separated, while their ordering as elements of \([M]\) defines the predecessor sets in \eqref{eq:within-class-predecessor-set}. For each block \(B_{m_{k,\ell}}\), one may thus construct a copy satisfying $Z_{B_{m_{k,\ell}}}'
\overset{d}{=}
Z_{B_{m_{k,\ell}}}$, and $Z_{B_{m_{k,\ell}}}'
\perp\!\!\!\perp
\sigma\left(
Z_j:
j\in\operatorname{Pred}_{\ell}(C_k)
\right)$, and bound the corresponding \(\tau\)-coefficient by the expected coupling discrepancy.

Thus, model-specific coupling arguments developed for time series extend to graphs once color classes identify well-separated blocks and predecessor sets provide an analogue of the temporal past. Section \ref{sec:examples} illustrates this construction by resampling the primitive shocks through which predecessor blocks affect the current block.\footnote{For comparison, \(\psi\)-dependence, introduced by \cite{DOUKHAN1999} and adapted to networks by \cite{KOJEVNIKOV2021882}, bounds covariances using localized versions \(Z_A'\) and \(Z_B'\) independent of \(Z_A\) and \(Z_B\). This is weaker than our requirement that
\(Z_{B_{m_{k,\ell}}}'\perp\!\!\!\perp\sigma\left(
Z_j:
j\in \operatorname{Pred}_{\ell}(C_k)\right)\). However, as opposed to our coefficient \(\bar{\tau}^{\max}\),
independence from localized predecessors does not provide the sequential coupling needed for our maximal inequality. 
}

\section{Maximal Inequality}
\label{sec:maximal}

In this section we derive a maximal inequality for the empirical process $\mathbb{G}_n(f)
	= \frac{1}{\sqrt{n}}
	\sum_{i\in N_n}
	\Big(
	f(Z_i) - \mathbb{E}[f(Z_i)]
	\Big)$ 
over a  class of measurable functions \(\mathcal F\). 

To state the result, we define the metric over $\mathcal{F}$ and the corresponding complexity measure.

\paragraph{Talagrand's measure of complexity} 	For any $\mathcal B \subseteq \mathcal F$ and $r>0$, a sequence of partitions $T^\infty := (T_l)_{l\in \mathbb{N}_0}$ is called an admissible partition sequence of order $r$ for $\mathcal B$ if:
\begin{enumerate}
	\item $T^\infty$ is increasing (i.e., $T_{l+1} \supseteq T_l$ for all $l$),
	\item $\operatorname{card}(T_0)=1$, $\operatorname{card}(T_l) \leq 2^{2^{l/r}}$ for all $l \geq 1$.
\end{enumerate}
Let $\mathcal T_r(\mathcal B)$ denote the collection of all such admissible partition sequences. For any $f \in \mathcal B$ and $l\in\mathbb{N}_0$, let $T(f,T_l)$ be the unique element of $T_l$ containing $f$, and let
\[
z \mapsto D(f,T_l)(z)
:=
\sup_{f_1,f_2 \in T(f,T_l)} |f_1(z) - f_2(z)|
\]
be the diameter function associated to it.

\begin{definition}[Talagrand complexity measure]\label{def:talagrand}
	Let $r>0$, $p>0$ and let $d$ be a (pseudo) metric over $\mathcal B$. The Talagrand complexity measure for $\mathcal B \subseteq \mathcal F$ is defined as
	\[
	\gamma_{p,r}(\mathcal B,d)
	:=
	\inf_{T^\infty \in \mathcal T_r(\mathcal B)}
	\sup_{f \in \mathcal B}
	\sum_{l=0}^{\infty}
	2^{l/p}\, d\big( D(f,T_l) \big).
	\]
\end{definition}

\paragraph{Graph-induced semi-norms over $\mathcal F$} Dating back at least to \cite{Dudley-1967}, the natural metric for measuring the complexity of a function class is the one dictated by the concentration behavior of the associated empirical process. In the present setting, the empirical process is approximated by the color-wise block independent process constructed in Theorem \ref{thm:coupling} and  the relevant concentration inequality is the Bernstein inequality (presented in Lemma \ref{lem:bern} below). Consequently, the appropriate distances over $\mathcal F$ are those that separately control the two quantities entering that inequality: a variance component of the block independent process and a uniform boundedness component.

For time series, the variance component is typically measured through the standard deviation of partial sums over blocks of consecutive observations. For a block $B = \{ 1,..., L  \}$ of length $L$, this quantity takes the form $\sqrt{
		\operatorname{Var}
		\left( L^{-1/2}
		\sum_{l=1}^{L} g(Z_{l})
		\right)
	}$. 
Under stationarity and suitable mixing conditions, one can use Theorem 1.1 in \cite{rio2017asymptotic} to obtain, for $p\geq 1$,
\begin{align}
	\sqrt{
		\operatorname{Var}
		\left( L^{-1/2}
		\sum_{l=1}^{L} g(Z_{l})
		\right)
	}
	\lesssim 
		\left(
	\int_0^1
	\left[
	\sum_{j=0}^{L}
	1\{u\leq \alpha_{\mathcal F ,B}(j)\}
	\right]^{\frac{p}{p-1}}
	du
	\right)^{\frac{p-1}{2p}} \, \|g\|_{L^{2p}(P)} ,
	\label{eqn:TimeSeries.norm}
\end{align}
where
\begin{align*}
	\alpha_{\mathcal{F}, B}(r) : =  \sup_{f \in \mathcal{F} - \mathcal{F} } \sup_{i,j \in B \colon d_{n}(i,j) \geq r }   \alpha(f(Z_i),f(Z_j)),~\forall  r \in \mathbb N_{0}  
\end{align*}
is the $\alpha$-mixing coefficient as in \cite{rio2017asymptotic} expression 1.8a.\footnote{This formulation of the $\alpha$-coefficient is no larger than the standard one and in some cases is non-trivially smaller. To see this note that the standard one can be taken as $\sup_{f,g} |\operatorname{Cov}(f(Z_{i}),g(Z_{j}))|$ where the supremum is over the class of bounded (by one) measurable functions. On the other hand, in the present formulation of $\alpha$ the supremum is taken over the class given by $BV_{1} \circ (\mathcal F - \mathcal F)$, which is no larger, and in some cases strictly smaller than the former class.} The right-hand side of \eqref{eqn:TimeSeries.norm} therefore induces the relevant metric for the chaining argument.

We follow the same principle in this paper, but the graph structure changes the geometry of the variance term. Blocks are no longer intervals of consecutive observations. Instead, they are elements $B$ of a partition of the node set $N_n$. Thus the object to be controlled is  $\sqrt{
		\operatorname{Var}
		\left( |B|^{-1/2}
		\sum_{i\in B} g(Z_i)
		\right)
	}$,  and, in contrast to the time-series case, there is generally no translation-invariant structure that reduces this expression to a function of the block length alone.

A second difference is that graph distance is no longer one-dimensional. In a time series, the distance between observations $i$ and $j$ is simply $|i-j|$, and the number of observations at any fixed distance from a given point is uniformly bounded by two. In general graphs, distance is given by $d_n$, and, as we shall see in detail in Section \ref{sec:geometry} below,  the number of node pairs inside a block at each distance can vary substantially. Since dependence is typically stronger at shorter graph distances, blocks with many pairs at small distances have larger variance than blocks whose pairs are mostly far apart. The graph-induced norm must therefore account not only for the size of the block, but also for how its pairwise distances are distributed. 

As will become apparent from the proof of Lemma \ref{lem:VarBlock.upper} below, the appropriate quantity for capturing this feature is the \emph{within-block distance neighbor pair profile}: For any block $B$, the within-block distance neighbor pair of $B$ profile is 
\begin{align}
	S_{B}(r) := \{ (i,j) \in B \times B : d_{n}(i,j)=r \},~\forall r  \geq 0,
\end{align}
and $|S_{B}(r)|$ is the count of neighbor pairs at distance $r$.

%
%
%
%

The following result formalizes these observations and provides a natural (pseudo) metric for graph-dependent data.\footnote{Here and throughout, $\lesssim$ denotes ``less than or equal to \emph{up to universal constants}''.}	 

\begin{lemma}\label{lem:VarBlock.upper}
For any $g \in \mathcal F - \mathcal F$, any $n \in \mathbb{N}$, any $p \in [1,\infty]$, and any $B \subseteq N_{n}$, let 
	\begin{align}\label{eqn:var.bound.text}
	 \sqrt{	\operatorname{Var} \left( |B|^{-1/2} \sum_{i \in B} g(Z_{i})   \right) } \lesssim    \min \left\{  	\Gamma^{\mathrm{sum}}_{p}(B) ,  	\Gamma_p^{\sharp}(B) |B|^{1/2p}  \right\} \max_{i \in B}  ||g||_{L^{2p}(P_{Z_{i}})}, 
	\end{align}
where 
\begin{align*}
	\Gamma^{\mathrm{sum}}_{p}(B) : = \left\{  
	\begin{array}{ll}
		\sqrt{ \sum_{\rho \in R(B)}  	\frac{|S_{B}(\rho)|}{|B|}  \left( \alpha_{ \mathcal{F} , B }(\rho)  \right)^{1-\frac{1}{p} } }  & if~p \in (1,\infty] \\
		\sqrt{ \sum_{\rho \in R(B)}  	\frac{|S_{B}(\rho)|}{|B|} }  & if~p = 1
	\end{array}
	\right.
\end{align*}
and \begin{align*}
	\Gamma_p^{\sharp}(B)
	:=
	\left\{  
		\begin{array}{ll}
	\sqrt{ \left(
	\int_0^1
	\left[
	\sum_{\rho\in R(B)}
	\frac{|S_B(\rho)|}{|B|}
	1\{\alpha_{\mathcal F,B}(\rho)\geq u\}
	\right]^{\frac{p}{p-1}}
	du
	\right)^{1-\frac{1}{p} } } & if~p \in (1,\infty] \\
	\sqrt{ \sum_{\rho \in R(B)}  	\frac{|S_{B}(\rho)|}{|B|} }  & if~p = 1	
	 	\end{array}
 		 \right.
\end{align*}
with $R(B) : = \{ r \in \{ 0,1, \ldots, n-1\} \cup \{\infty\}    \colon 	S_{B}(r) \ne \emptyset  \}$.\footnote{We set $d_{n}(i,i) = 0$ and $\alpha_{\mathcal{F} , B }(0) = 1$. The case $r=\infty$ captures pairs that are not connected, which by our conditions are independent and thus $\alpha_{\mathcal{F} , B }(\infty) : =  0$.}

Moreover, if $\max_{i \in B} ||dP_{Z_{i}}/dP_0||_{\infty} \leq C$, for some measure $P_0$, then the norm can be taken to be $||.||_{L^{2p}(P_0)}$ and the factor $|B|^{1/2p}$ in  \eqref{eqn:var.bound.text} can be dispensed with.
\end{lemma}

\begin{proof}
	See Appendix \ref{app:maximal}.
\end{proof}

Henceforth, for any $n \in \mathbb{N}$, any $p \in [1,\infty]$, and any $B \subseteq N_{n}$, let
\begin{align*}
	\Gamma_{p}(B) : = \min\left\{  \Gamma_{p}^{\mathrm{sum}}(B) , \Gamma_{p}^{\sharp}(B) |B|^{1/2p}  \right\}.
\end{align*}
Thus, \(\Gamma_p(B)\) controls the variance inflation generated by local network dependence inside \(B\). 
It summarizes the combined effect of two features of the block \(B\): the number of within-block $r$-distance neighbor pairs, measured by \(|S_B(r)| \) and the strength of dependence at distance \(r\), measured by \(\alpha_{\mathcal F,B}(r)\), for all possible distances  $r$ in $R(B)$. 

\paragraph{Maximal inequality} We are now in a position to state the maximal inequality.\footnote{Here and throughout, to avoid measurability issues, for any nonempty
	index set $T$ and process $(X_t)_{t\in T}$ consisting of integrable
	random variables, we follow
	\cite[Section~2.2, p.~21, equation~(2.2)]{talagrand2014} and define $\mathbb E\sup_{t\in T}X_t
	:=
	\sup_{\substack{F\subseteq T\\0<|F|<\infty}}
	\mathbb E\max_{t\in F}X_t$. 
	Each finite maximum is measurable, so the right-hand side is well-defined
	whenever each $X_t$ is integrable.}

\begin{theorem}\label{thm:maximal}
	For any $n \in \mathbb{N}$, $a,b \geq 1$, $p \in [1,\infty]$, and for any partition $	\mathcal P_{M} : = \{B_{1},...,B_{M}\}$ and any $r$-coloring of this partition,  $\mathcal C_{r}[\mathcal P_{M}]$,
	\begin{align*}
		\left \| \sup_{f\in\mathcal{F}}  \mathbb G_{n}(f)   \right\|_{L^{1}(\mathbb P)} \lesssim &   \sqrt{K_{r}}  \max_{m \in [M]} \Gamma_{p}(B_{m})  \gamma_{2,a}(\mathcal F ,   \max_{i \in N_{n}}   || \cdot ||_{L^{2p}(P_{Z_{i}}) } )   \\
		& +  \frac{K_{r}  \max_{m \in [M]} |B_{m}|}{\sqrt{n}} \gamma_{1,b}(\mathcal F ,   ||\cdot||_{\infty}) + \sqrt{n} 	\bar{\tau}_{  \mathcal F}
		\left(
		\mathcal P_{M},
		\mathcal C_{r}[\mathcal P_{M}]
		\right) . 
	\end{align*}
\end{theorem}

\begin{proof}
	See Appendix \ref{sec:ProofMaximal}.
\end{proof}

This theorem is a refinement and generalization of Theorem 1 in \cite{Pouzo2026} for time series to graph-dependent data. In time series, the maximal inequality is controlled by the generic chaining complexity of the class and the relevant marginal norms. The present result preserves that structure but extends it to general graph-dependent data by introducing two additional objects: a graph-coloring device, which separates approximately independent blocks, and dependence coefficients, which control the error from replacing the original process by blockwise independent approximations.

Since the partition and the coloring are arbitrary and the left-hand side does not depend on either of them, the bound may be optimized over all admissible choices by minimizing the right-hand side over all partitions and over all admissible $r$-colorings. In practice, carrying out this minimization may be impractical. But, it is enough to find \emph{one} partition and coloring for which the resulting geometric and dependence terms are controlled.

Another relevant feature of the bound is that its components separate three distinct sources. The first source is the size of the function class, measured by the generic chaining functionals $\gamma_{2,a}\left(\mathcal F,\max_{i\in N_n} ||\cdot||_{L^{2p}(P_{Z_i})}\right)$ and $\gamma_{1,b}(\mathcal F,||\cdot||_\infty)$. These are standard complexity measures, introduced by Talagrand (cf. \cite{talagrand2005,talagrand2014}), further developed and used in several papers (e.g., \cite{VanHandel2018,VanHandel2018b}), and bounded above by classical entropy quantities such as Dudley’s entropy integral (see \cite{talagrand2014}). The second source is the geometry of the graph and of the chosen block partition. This is summarized by the $r$-coloring and its chromatic number $K_r$, and the within-block distance neighbor pair profile, $S_{B_m}$. The third source is the mixing structure. This enters in two distinct ways: one captured by the $\tau$-coefficient, 
and another, local within-block dependence, captured by the within-block dependence factor, $\Gamma_p(B_m)$. The latter combines the within-block distance neighbor-pair profile  with the decay of the $\alpha$-mixing coefficients.

Consequently, the theorem reduces the problem of proving a maximal inequality to three tasks: bounding the chaining complexity of $\mathcal F$, constructing a block partition with a manageable $r_n$-chromatic number and distribution of within-block distance neighbor pairs, and verifying sufficient decay of the $\tau$-mixing and within-block $\alpha$-mixing coefficients.

The graph-geometric tasks described above appear to be novel and intrinsic to graph-dependent data. We study them in detail below, while bounds on the graph-based $\tau$-mixing coefficient are developed for specific models in Section~\ref{sec:examples}.\footnote{For the remaining task of bounding the complexity functionals, one can draw on the existing generic-chaining literature. In particular, \cite{talagrand2005,talagrand2014} show that these functionals can be controlled by familiar quantities such as Dudley's entropy integral and Ossiander's bracketing integral (e.g., \cite{VdV-W1996}). Talagrand's results further relate the complexity functionals, under suitable conditions on $\mathcal F$ and the underlying metric, to the expected supremum of an associated Gaussian process. See \cite{Pouzo2026} for a discussion of how these tools can be applied to empirical processes for time-series data. In addition, \cite{VanHandel2018,VanHandel2018b} provide bounds based on the entropy numbers of suitably chosen ``simple'' subsets of the indexing class.}

Before examining these graph-geometric components under specific network structures, however, we place Theorem~\ref{thm:maximal} within two classical benchmarks: the exact dependency-graph and time series settings. These benchmarks provide useful sanity checks by showing how the theorem relates to known results.

\paragraph{Benchmark I: Finite Classes under a Dependency Graph} We illustrate Theorem~\ref{thm:maximal} by specializing it to the classical dependency-graph setting. The theorem recovers the standard Bernstein-type maximal inequality rate based on a proper coloring of the \emph{dependency} graph with a finite function class (see \cite{janson2004largedeviations}).

Let \(\mathfrak D_n=(N_n,E_n)\) be a dependency graph for
\((Z_i)_{i\in N_n}\), in the sense that, for any two disjoint sets
\(A,B\subseteq N_n\) without edges between them, $\sigma(Z_i:i\in A)$  and $\sigma(Z_i:i\in B)$ are independent.

\begin{corollary}
\label{cor:finite-dependency-graph}
 If \(\mathcal F\) is such that \(\lvert\mathcal F\rvert=J\)
then
\begin{align}
    \mathbb E\left[
        \sup_{f\in\mathcal F}
        \left|
            \frac{1}{n}
            \sum_{i\in N_n}
            \left(
                f(Z_i)-\mathbb E[f(Z_i)]
            \right)
        \right|
    \right]
    \lesssim
    \sigma
    \sqrt{
        \frac{\chi( \mathfrak D_n )\log(2J)}{n}
    }
    +
    b\,
    \frac{\chi( \mathfrak D_n  )\log(2J)}{n},
    \label{eq:finite-dependency-graph}
\end{align}
with $\sigma : = \max_{f\in\mathcal F}\max_{i\in N_n}
    \bigl\|f(Z_i)-\mathbb E[f(Z_i)]\bigr\|_{L^2}$, $
    b : = \sup_{f\in\mathcal F}\|f\|_{\infty}$, 
and  \(\chi(\mathfrak D_n)\) denotes the chromatic number of $\mathfrak D_{n}$.

\end{corollary}

\begin{proof}
   See Appendix  \ref{sec:dep-graph.proof}.
\end{proof}

Classical dependency-graph inequalities, such as \cite{janson2004largedeviations}, yield Bernstein bounds for individual sums and, by a union bound, finite function classes. Corollary~\ref{cor:finite-dependency-graph} recovers this structure using the chromatic number. 

Theorem~\ref{thm:maximal} extends this benchmark to possibly infinite classes through generic chaining and to more general graph dependence through the coupling and within-block dependence terms. Under exact dependency-graph dependence, however, sharper inequalities replace \(\chi(\mathfrak D_n)\) with the fractional chromatic number \(\chi^*(\mathfrak D_n)\); see \cite{ScheinermanUllman1997}. Thus, our theorem is broader but may not recover the sharpest graph factor, except when \(\chi(\mathfrak D_n)\) and \(\chi^*(\mathfrak D_n)\) have the same order. Inspection of the proofs leads us to conjecture that Theorem \ref{thm:maximal} extends to fractional colorings, but a formal development is outside the scope of this paper and is left to future work.


The preceding benchmark is obtained by taking singleton blocks, in which case
the coupling error vanishes and the only relevant geometric quantity is the
chromatic number of the dependency graph. We now turn to the more general
setting in which graph-separated observations may remain dependent. In this
case, the choice of block partition, the coloring cost, and the within-block
distance profile become substantive. 

\paragraph{Benchmark II: Time series} In the time series setting  $N_{n} \subseteq \mathbb{Z}$ and $d_n(i,j)=|i-j|$, giving the graph a clear geometric structure. This structure suggests a partition of consecutive blocks, $B_m=\{(m-1)q+1,\ldots,mq\}$ for $m \in [n/q]$.\footnote{We assume that $M = n/q$ is an integer for simplicity.}

The within-block $B_{m}$ distance neighbor profile is trivial, uniformly bounded. To see this, observe that since $B_{m}$ is ordered and has length $q$, $|S_{B_{m}}(\rho)|=0$ for any $\rho \geq q$. At distance $\rho=q-1$ there are only two pairs ($((m-1)q+1,mq)$ and $(mq,(m-1)q+1)$); at distance $\rho = q-2$ there are four (ordered) pairs (first and second-to-last and second and last). Thus, iterating in this manner, it follows that $|S_{B_{m}}(\rho)| = 2(q-\rho)$ for all $\rho \in \{0,\ldots,q-1\}$. 
This implies that $\Gamma_{p}(B_m)$ (or at last a bound of it) only depends on the within-block dependence, not on the neighbor profile.

Given an $r \geq 1$, blocks that conflict at radius $r$ are precisely those with sufficiently nearby indices. Indeed, any blocks $B_{m},B_{\ell}$ have distance  $d_n(B_m,B_{\ell}) 	=	\max\{0,\; (|m-\ell|-1)q + 1\}$. Hence, if they are within $r$ of each other then $|m-\ell| \leq 1+\frac{r}{q}$. This readily implies that the degree of any node $m$ is bounded (up to universal constants) by $1+r/q$. Hence, by the well-known greedy bound, the $r$-chromatic number is also bounded by $1+r/q$.\footnote{See Appendix \ref{app:graph} for definitions of these concepts.}

By inputting these bounds into Theorem \ref{thm:maximal} and setting $q_{n}=r_{n}$, we obtain an improvement (up to constants) of Theorem 1 in \cite{Pouzo2026}.\footnote{The improvement lies in the term $\sqrt{n} \tau_{  \mathcal F}$; the last part of the proof of Theorem \ref{thm:maximal} explicitly discusses this improvement. }


\subsection{Geometric Structure of Graphs}
\label{sec:geometry}

The biggest difference between the bounds obtained for IID and time series data and the one in Theorem \ref{thm:maximal} is due to the intrinsic ``geometric structure'' of the time series graph. We now provide bounds on the key quantities associated to this structure in canonical examples of graphs. 


\paragraph{ Uniform doubling graphs with polynomial growth} The next result isolates the key geometric features of a time series graphs to obtain an analogous result but for a larger class of graphs which we call \emph{uniform doubling with polynomial growth of dimension $d$}.

A sequence of metric spaces $(X_n,d_n)_{n \in \mathbb N}$ is uniformly doubling metric space with dimension $d$ if there exists a constant $C_{\mathrm{dbl}} < \infty$, independent of $n$, such that for every 
	$x \in X_n$ and every $R>0$,
	\begin{align*}
		B(x,R)
		\subseteq
		\bigcup_{k=1}^{C_{\mathrm{dbl}}}
		B(x_k, R/2),
	\end{align*}
	for some points $x_1,\dots,x_{C_{\mathrm{dbl}}} \in X_n$. The associated dimension is given by $d : = \log_{2} C_{\mathrm{dbl}}$. 


   The doubling condition distills a fundamental feature of finite-dimensional Euclidean spaces: their geometric complexity grows at a controlled, polynomial rate with distance and scale. This is useful because it provides Euclidean-type covering and packing bounds, with the doubling dimension playing the role of the Euclidean dimension, without requiring a linear structure.
	 
	The sequence   $(X_n,d_n)_{n \in \mathbb{N}}$ is  said to have  polynomial (shell) growth if there exists a constant $C  < \infty$, independent of $n$, such that  \begin{align*}
		\max_{x \in X_{n}} |\partial B(x,s) | \leq C s^{d-1} ~\forall s \geq 1.  
	\end{align*}

        These conditions are natural for networks whose geometry is induced by an underlying fixed-dimensional metric structure, such as geographic proximity networks, spatial lattices and geographic or low-dimensional latent-space networks, but may fail in networks with hubs or rapidly expanding neighborhoods.

 \begin{lemma}
 	\label{lem:PolyGraph.color}
 	Let \(\mathfrak{G}_n=(N_n,V_n)\) be  a uniform doubling with polynomial growth graph of dimension $d$. Then, for every \(q\ge 1\), there exists
 	a Voronoi partition \(\mathcal P_{M} : =\{B_1,\ldots,B_M\}\) of \(N_n\) such that
 	\begin{align}\label{eqn:Double.Poly.1}
 		\max_{1\le m\le M}|B_m|
 		\lesssim q.
 	\end{align}
 	And, for every \(r\ge 1\),
 	\begin{align}\label{eqn:Double.Poly.chrom}
 		K_{r}(\mathcal P_{M} )
 		\lesssim
 		1+\frac{r^d}{q},
 	\end{align}
and, for any $m \in [M]$,
 	\begin{align}\label{eqn:Double.Poly.shells}
 	\frac{|S_{B_m}(r)|}{|B_m|}
 	\lesssim
 	\min\{r^{d-1},q\}.
 \end{align}
 
 \end{lemma}

\begin{proof}
	See Appendix \ref{app:geometry}. 
\end{proof}


By taking $q_n\asymp r_n^d$ the coloring number remains bounded and Theorem~\ref{thm:maximal} yields:

\begin{proposition}
\label{pro:maximal.DoublePoly}
	Let \(\mathfrak{G}_n=(N_n,V_n)\) be uniform doubling with polynomial growth graph of dimension $d$. 
    Then, for any $n \in \mathbb{N}$, any $a,b \ge 1$, any $p \in [1,\infty]$, and any $d \geq 1$, there exists a partition \(\mathcal P_{M_{n}}=\{B_1,\ldots,B_{M_n}\}\)  satisfying $\max_{m\in[M_n]}|B_m|\lesssim q_n $ with $q_{n} \in [n] $, such that
{\small{	\begin{align*}
		\left \| \sup_{f\in\mathcal{F}}  \mathbb G_{n}(f)   \right\|_{L^{1}(\mathbb P)} \lesssim 	\Lambda_{p,d}(q_{n})  \gamma_{2,a}(\mathcal F ,   \max_{i \in N_{n}}   || \cdot ||_{L^{2p}(P_{Z_{i}}) } )   +   \frac{q_{n}}{\sqrt{n}}   \gamma_{1,b}(\mathcal F ,   ||\cdot||_{\infty}) +  \sqrt{n} 	\bar{\tau}_{ \mathcal F}
		\left( 	\mathcal P_{M_{n}}, 	\mathcal C_{q_{n}^{1/d}}[\mathcal P_{M_{n}} ] \right)
	\end{align*}}}
where 
\begin{align*}
	\Lambda_{p,d}(q_{n}) : = 	\left\{  
	\begin{array}{ll}
	  q_{n}^{\frac{1-\frac{1}{d}}{2}} \sqrt{ 
\min \left\{  \Lambda^{\sharp}_{p,d}(q_{n}),  \Lambda^{\operatorname{sum}}_{p,d}(q_{n})  \right\}  }  & if~p \in (1,\infty] \\
		\sqrt{ 	q_{n}  } & if~p = 1
	\end{array}
\right.,
\end{align*}
with $\Lambda^{\sharp}_{p,q}(q_{n}) : = q_{n}^{\frac{1+p}{p}} \left(
	\int_0^1
	\left[
	\sum_{\rho = 0}^{ \lfloor 2 q^{1/d}_{n} \rfloor+1 }
	1\{\alpha_{\mathcal F,B_m}(\rho)\geq u\}
	\right]^{\frac{p}{p-1}}
	du
	\right)^{1-\frac{1}{p} } $ and $\Lambda^{\operatorname{sum}}_{p,q}(q_{n}) : =  \sum_{\rho=0}^{	\lfloor 2 q^{1/d}_{n} \rfloor+1}   \left(  \alpha_{ \mathcal{F} , B_m }(\rho)   \right)^{1-\frac{1}{p} }$.

\end{proposition}

\begin{proof}
	See Appendix \ref{app:geometry}.
\end{proof}

\paragraph{Exponential Growth Graphs} We now analyze the maximal inequality for graphs that have more rapidly expanding neighborhoods than the class of polynomial growth graphs. We say  \(\mathfrak{G}_n=(N_n,V_n)\) has \emph{exponential (shell) growth with constant $c$}  if there exist constants \(C,c>0\), independent of \(n\), such that
\begin{align}
	\sup_{i\in N_n} |\partial B(i,s)|
	\le C e^{c s}
	\qquad
	\text{for all } s\ge 1.
\end{align}

\begin{lemma}
	\label{lem:coloring.exponential}
	Let \(\mathfrak{G}_n=(N_n,V_n)\) be a graph with an exponential (shell) growth with constant $c$. 	Then, for every \(q\ge 1\), there exists a partition
	\(\mathcal P_M=\{B_1,\ldots,B_M\}\) of \(N_n\) such that $\max_{1\le m\le M}|B_m|
		\lesssim q$,  and
for every \(r\ge 1\) and $m \in [M]$, $K_{r}(\mathcal P_M) 
		\lesssim
		q^{2} e^{c r}$, and $
		\frac{|S_{B_m}(r)|}{|B_m|}
		\lesssim
		\min\{e^{c r},q\}$. 
\end{lemma}

\begin{proof}
	See Appendix \ref{app:geometry}. 
\end{proof}

The next result specializes the maximal inequality in Theorem \ref{thm:maximal} to this class of graphs. 

\begin{proposition}
	\label{pro:maximal.exponential.growth}
	Let \(\mathfrak{G}_n=(N_n,V_n)\) be an exponential (shell) growth with constant $c$ graph.  Then, for any $n \in \mathbb{N}$, any $a,b \geq 1$, and any $p \in [1,\infty]$, there exists a  partition \(\mathcal P_{M_n}=\{B_1,\ldots,B_{M_n}\}\) satisfying $\max_{m\in[M_n]}|B_m|\lesssim q_n $ with  \(q_n\asymp 1\) such that
	\begin{align*}
		\left\|
		\sup_{f\in\mathcal F}\mathbb G_n(f)
		\right\|_{L^1(\mathbb P)}
		\lesssim\;&
		e^{c \frac{r_n}{2}}
		\gamma_{2,a}
		\left(
		\mathcal F,
		\max_{i\in N_n}\|\cdot\|_{L^{2p}(P_{Z_i})}
		\right)
		+
		\left(
		\frac{e^{c r_n}}{\sqrt n}
		\right)
		\gamma_{1,b}
		\left(
		\mathcal F,\|\cdot\|_\infty
		\right) +	\sqrt n\,
		\bar{\tau}_{ \mathcal F}(r_n).
	\end{align*}
    where 
	\(	\bar{\tau}_{  \mathcal F}
	\left( 	\mathcal P_{M_{n}}, 	\mathcal C_{r_{n}}[\mathcal P_{M_{n}} ] \right)    = : 
	\bar{\tau}_{\mathcal F}(r_n)\)

\end{proposition}

\begin{proof}
	See Appendix \ref{app:geometry}.
\end{proof}

As opposed to the polynomial growth case, $q_{n}$ is a nuisance as it increases the chromatic number. Thus, it is set to a uniform constant. In addition, from the first term of the RHS one can see that $r_{n}$ can grow at most at logarithmic, and at a specific, rate to at least have $e^{c r_{n}/2} = o(\sqrt{n})$.

\paragraph{Directed dyadic graphs}
Consider directed dyadic data, such as friendship nominations, trade or migration flows, input--output linkages, and interbank exposures. Let
\begin{align}\label{eqn:Z.Fform}
	N_n
	:=
	\left\{
	(i,j):
	i,j\in[n],\ i\neq j
	\right\},
	\qquad
	Z_{ij}
	=
	F(\mu_i,\mu_j,\varepsilon_{ij}),
\end{align}
where $(\mu_i)_{i\in[n]}$ and $(\varepsilon_{ij})_{(i,j)\in N_n}$ are mutually independent collections of IID random variables. The natural dependency graph $\mathfrak G_n$ connects two dyads whenever they share an endpoint:
\[
	(i,j)\sim(k,\ell)
	\quad\Longleftrightarrow\quad
	\{i,j\}\cap\{k,\ell\}\neq\varnothing.
\]
Thus, distinct dyads are at distance one if they share an endpoint and at distance two otherwise.

The next lemma summarizes the relevant geometry and dependence.

\begin{lemma}
\label{lem:directed-dyads-coloring-tau}
The following claims hold:
\begin{enumerate}

	\item
	For every $(i,j)\in N_n$, $|S((i,j),0)|
		=
		1$, $|S((i,j),1)|
		\asymp
		n$, and 
		$|S((i,j),2)|
		\asymp
		n^2$. 

	\item
	$\tau_{\mathcal F}(r,q)=0$ for every $r\geq2$.
\end{enumerate}
\end{lemma}

\begin{proof}
See Appendix~\ref{app:geometry}.
\end{proof}

Unlike time series, dyadic graphs have a trivial long-range dependence coefficient but dense local geometry, reflected in their shell profile and coloring number. By Theorem~\ref{thm:maximal},

\begin{proposition}
\label{prop:maximal-directed-dyads}
Let $(Z_{ij})_{(i,j)\in N_n}$ be defined as above. For every $p>1$ and $a,b\geq1$,
\begin{align}
	\left\|
		\sup_{f\in\mathcal F}
		\mathbb G_n(f)
	\right\|_{L^1(P)}
	\lesssim
	\sqrt n\,
	\gamma_{2,a}
	\left(
		\mathcal F,
		\|\cdot\|_{L^{2p}(P)}
	\right)
	+
	\gamma_{1,b}
	\left(
		\mathcal F,
		\|\cdot\|_\infty
	\right).
	\label{eq:maximal-directed-dyads}
\end{align}
\end{proposition}

\begin{proof}
See Appendix~\ref{app:geometry}.
\end{proof}

Thus, despite having $n(n-1)$ observations, directed dyadic data have an analogous bound to an IID setting with sample size of order $n$. 

\section{Glivenko--Cantelli Results and Effective Sample Size}
\label{sec:gc-effective-sample-size}

The purpose of this section is twofold. First, we employ
Theorem~\ref{thm:maximal} to obtain Glivenko--Cantelli results over graph-dependent data. Second,
we provide convergence rates and interpret them in terms of effective sample size. The main
message is that, contrary to IID data, there is no universal Glivenko--Cantelli rate. The
rate is determined jointly by graph geometry, local and long-range dependence.

\subsection{A Glivenko--Cantelli result for  graph-dependent data}

For any finite partition \(T_{L}\) of \(\mathcal F\), with $|T_{L}| \leq 2^{2^{L}}$, let \(\Pi[T_{L}] : = (\pi_C)_{C \in T_{L}}  \) be a representative class: For each $C \in T_{L}$, the element is such that $\pi_{C}$ for all $f \in C$, $\pi_{C} : = \pi_{L} f$. Also, let $	\mathcal D[T_{L}] := \{D_C \colon C \in T_{L}\}$ be the class of diameter functions, $z \mapsto D_C(z)
	:=
	\sup_{f_1,f_2\in C}
	|f_1(z)-f_2(z)|,~\forall C\in T_{L}$.

We approximate $\mathcal F$ by a finite representative class and control the residual variation within each cell. The maximal inequality is applied to the representatives, while the average cell diameter controls the approximation error.

\begin{theorem}	\label{thm:single-scale-network-gc}
	Suppose $\sup_{ f \in \mathcal F } ||f||_{\infty} < \infty$ and, for each $n$, there exists:
	\begin{itemize}
		\item A partition
		\(\mathcal P_n=\{B_1,\ldots,B_{M_n}\}\) of \(N_n\), with
		\(\max_m |B_m|\le q_n\), and an associated \(r_n\)-coloring,
		\item A  finite partition of \(\mathcal F\), \(T_{L_{n}}\)
	\end{itemize}
	such that for some $a$ and $b$,
	{\small{\begin{align}
		&\left( \frac{
			\sqrt{ K_{r_{n}}}	}{\sqrt n}  \max_{m\in[M_n]}\Gamma_p(B_m)    \right)   2^{\frac{\lceil a L_{n} \rceil}{2}}  \to 0,~  2^{\lceil b L_{n} \rceil}   \frac{K_{r_{n}}  q_n}{n}
		+	\bar{\tau}_{ \Pi[T_{L_{n}}]  }
		\left( 	\mathcal P_{n}, 	\mathcal C_{r_{n}}[\mathcal P_{n} ] \right)    \to 0, 	\label{eq:ssgc-net-gamma1} \\
		&~and \quad \frac{1}{n}
		\sum_{i\in N_n}
		\mathbb E 	\left[  \sup_{C\in T_{L_n}} D_C(Z_i) \right] \to 0,
		\label{eq:ssgc-pop-diameter}
	\end{align}}}
	as $n \to \infty$. Then \(  (\mathcal F,\mathfrak{G}_n ) \) is graph Glivenko--Cantelli, i.e.,
	\begin{align}
		\sup_{f\in\mathcal F}
		\left| \frac{1}{n} \sum_{i \in N_{n}} \{ f(Z_{i}) - \mathbb E[f(Z_{i})]   \}  \right|
		=
		o_{\mathbb P}(1).
	\end{align}
\end{theorem}

\begin{proof}
	See Appendix \ref{sec:proofGC}.
\end{proof}

Theorem \ref{thm:single-scale-network-gc} provides sufficient conditions over the class of functions $\mathcal{F}$, as well as over the graph $\mathfrak G_{n}$, to support a uniform law of large numbers. The conditions require that the partition \( T_{L_n}\) balances two requirements: The class
\(\Pi[ T_{L_n}]\) must be sufficiently small for the stochastic term to vanish, while its
cells must be sufficiently fine for the mean cell oscillation to vanish.

One way to construct such a partition is to take $T_{L}$ as $(B_{||.||_{\infty}}(f_{j},\epsilon))_{j=[M]}$ where $\{f_{1},...,f_{M}\}$ is the $\epsilon$-packing of $\mathcal{F}$ under $||.||_{\infty}$ and $M : = M(\epsilon,\mathcal F, ||.||_{\infty})$ is the corresponding packing number. Under this choice, $|T_{L}| = M$ and $||D_{C}||_{\infty} \leq 2 \epsilon$; therefore the conditions hold provided that there exists a vanishing sequence $(\epsilon_{n})_{n}$ such that 
\begin{align*}
		& \sqrt{ K_{r_{n}}}	  \max_{m\in[M_n]}\Gamma_p(B_m)  \sqrt{ \frac{ \log  M(\epsilon_{n},\mathcal F, ||.||_{\infty})}{n} } \to 0\\
		~and~	& \left(  K_{r_{n}} q_n   \right)  \frac{\log  M(\epsilon_{n},\mathcal F, ||.||_{\infty})}{n} +n 	\bar{\tau}_{ \Pi[T_{L_{n}}] }
		\left( 	\mathcal P_{n}, 	\mathcal C_{r_{n}}[\mathcal P_{n} ] \right)  \to 0.
\end{align*}
These conditions are extensions of the ``classical'' Glivenko-Cantelli restriction, $ \frac{\log  M(\epsilon_{n},\mathcal F, ||.||_{\infty})}{n}  \to 0$ (cf. \cite{VdV-W1996}) for IID data, to graph-dependent data. 

On the other hand, if the Talagrand's measures of complexity are finite, one can verify the conditions in the theorem by simply choosing $T_{L}$ as member of the (approximate) minimizing admissible sequence. Under this choice $\sup_{C\in T_{L} } ||D_{C}||_{\infty} \lesssim \gamma_{1,b}(\mathcal F , || . ||_{\infty})/2^{L}$ and the condition in \eqref{eq:ssgc-pop-diameter} readily follows for any diverging $L_{n}$, provided that the graph $\mathfrak G_{n}$ admits a partition and coloring such that $\frac{\sqrt{ K_{r_{n}} }	}{\sqrt n}  \max_{m\in[M_n]}\Gamma_p(B_m)   \to 0$ and $\frac{K_{r_{n}} q_n}{n} + 	\bar{\tau}_{\Pi[T_{L_{n}}] }
\left( 	\mathcal P_{n}, 	\mathcal C_{r_{n}}[\mathcal P_{n} ] \right)  \to 0$.

\subsection{Effective sample size} Theorems \ref{thm:maximal} and \ref{thm:single-scale-network-gc}  provide more than a maximal inequality and a Glivenko--Cantelli result. They also provide a measure of the effective amount of independent information contained in a graph-dependent sample. In IID settings, the effective sample size is the cardinality of the
sample. For graph-dependent data, the rate need not be \(n^{-1/2}\). The
graph may contain highly dependent local neighborhoods, and the number of
approximately independent pieces of information may be much smaller than
\(n\).

In this section we assume that there exist $a,b \ge 1$ and $p \in [1,\infty]$ such that  $\gamma_{2,a}(\mathcal F , \max_{i \in N_{n}} || . ||_{L^{2p}(P_{Z_{i}})})$ and $ \gamma_{1,b}(\mathcal F , || . ||_{\infty})$ are both finite. Theorem \ref{thm:maximal} suggests an \emph{effective sample size} given by 
\begin{align}
	\mathfrak{n}_{\mathrm{eff}}(\mathcal{F},\mathfrak{G}_n) : = \left( 	\frac{
		\sqrt{K_{r_{n}}}\max_{m \in [M_{n}]}\Gamma_p(B_m) }{\sqrt n}  + \frac{K_{r_{n}}  q_{n}}{n}  + 	\bar{\tau}_{  \mathcal F }
		\left( 	\mathcal P_{n}, 	\mathcal C_{r_{n}}[\mathcal P_{n} ] \right)   \right)^{-2}
\end{align}
for the partition and coloring ($\mathcal P_{n}  : = \{B_{1},...,B_{M_{n}}\}, 	\mathcal C_{r_{n}}[\mathcal P_{n} ]  $) that minimize the quantity in parenthesis.

In the IID case, this expression is of order $n$. For time series data, if the mixing coefficient decays fast enough with $r=q$, then the effective sample size coincides with that in the IID case (see, for example, \cite{DMR1995,yu1994rates} for $\beta$-mixing), however, without this restriction, \cite{Pouzo2026} shows that the sample size can diverge much slower than the actual one.\footnote{The word ``coincides'' in the last sentence should be qualified. It is true that \emph{asymptotically} the implied effective sample size and the actual sample size coincide, but in finite sample the former can be much smaller --- \cite{Pouzo2026} for an example and a more thorough discussion.}

For graph-dependent data this quantity is more nuanced: it depends on the graph structure and dependence structure through $\{ K_{r_{n}},\Gamma_p,q_{n}, \bar{\tau}_{ \mathcal F } \}$. Thus, contrary to IID or time series data, large graphs (i.e., large number of observations) need not imply a large effective sample size. A graph
with many nodes can still have a small effective sample size if neighborhoods expand rapidly, if many colors are required to separate blocks, or if dependence decays slowly.

We illustrate how graph geometry and mixing determine the effective sample size.

\paragraph{Polynomial-growth graphs}
For uniformly doubling graphs with polynomial growth of dimension $d$, Lemma~\ref{lem:PolyGraph.color} provides, for each $q_n$, a partition satisfying $\max_m|B_m|
	\lesssim
	q_n$, $\operatorname{diam}(B_m)
	\lesssim
	q_n^{1/d}$, and $K_{q_n^{1/d}}
	=
	O(1)$. Hence, Proposition~\ref{pro:maximal.DoublePoly} implies
\begin{align}
	\mathfrak n_{\mathrm{eff}}(\mathcal F,\mathfrak G_n)
	\asymp
	\left\{
	\frac{\Lambda_{p,d}(q_n)}{\sqrt n}
	+
	\frac{q_n}{n}
	+
	\bar\tau_{ \mathcal F}
	\left(
	\mathcal P_n,
	\mathcal C_{q_n^{1/d}}[\mathcal P_n]
	\right)
	\right\}^{-2}.
	\label{eq:polynomial-growth-rate}
\end{align}
Polynomial growth prevents the number of nearby blocks from exploding
too quickly, so the graph can be decomposed into a bounded number of
approximately independent color classes. However, the slower the coupling term decays, the smaller the effective sample size is. For instance, if
$\bar\tau_{ \mathcal F}
(\mathcal P_n,\mathcal C_r[\mathcal P_n])
\lesssim r^{-\beta}$, then
\[
	\mathfrak n_{\mathrm{eff}}(\mathcal F,\mathfrak G_n)
	\asymp
	\left\{
	\frac{\Lambda_{p,d}(q_n)}{\sqrt n}
	+
	\frac{q_n}{n}
	+
	q_n^{-\beta/d}
	\right\}^{-2}.
\]
Thus, if $\Lambda_{p,d}(q_n)=O(1)$ and $\beta/d\geq1$, a suitable choice of $q_n$ yields the IID effective sample size. More generally, the effective sample size diverges whenever there exists $q_n\to\infty$ such that
$q_n=o(n)$ and $\Lambda_{p,d}(q_n)/\sqrt n=o(1)$, yielding the graph Glivenko--Cantelli property.

\paragraph{Exponential-growth graphs}
Proposition~\ref{pro:maximal.exponential.growth} gives
	$\mathfrak n_{\mathrm{eff}}(\mathcal F,\mathfrak G_n)
	\asymp
	\left\{
	\sqrt{\frac{e^{cr_n}}{n}}
	+
	\bar\tau_{ \mathcal F}(r_n)
	\right\}^{-2}$.
If $\bar\tau_{\mathcal F}(r)\lesssim e^{-br}$, choosing
$r_n\asymp\log(n)/(2b+c)$ yields $\mathfrak n_{\mathrm{eff}}(\mathcal F,\mathfrak G_n)
	\asymp
	n^{2b/(2b+c)}$. 
If instead
$\bar\tau_{\mathcal F}(r)\lesssim r^{-b}$, an asymptotically balancing choice gives $\mathfrak n_{\mathrm{eff}}(\mathcal F,\mathfrak G_n)
	\asymp
	(\log n)^{2b}$. Thus, even though exponential graph growth can sustain a Glivenko--Cantelli property, it substantially reduces the effective sample size.

\paragraph{Directed dyadic graphs}
By Lemma~\ref{lem:directed-dyads-coloring-tau}, these graphs have diameter two and
$\tau_{\mathcal F}(r,q)=0$ for $r\geq2$, but their local geometry is dense. Proposition~\ref{prop:maximal-directed-dyads} therefore gives $\mathfrak n_{\mathrm{eff}}(\mathcal F,\mathfrak G_n)
	\asymp
	\frac{|N_n|}{n}
	\asymp
	n$, 
because $|N_n|=n(n-1)$. Thus, although the number of dyads is of order $n^2$, their effective sample size is only of order $n$, consistent with dependence generated by $n$ latent unit-level effects.

%
%
%

\section{Examples}
\label{sec:examples}

We illustrate the results using two models that serve different purposes. The first is a network autoregression in which, under a layered structure, shocks propagate across network links and the coupling cost decays geometrically with layer separation. This example demonstrates the scope of the new technique, yielding maximal inequalities that are, to our knowledge, beyond the reach of existing results for this model. The second is a nonlinear local-propagation model, encompassing network-treatment interference, in which dependence vanishes beyond a finite graph distance. Here, our coupling construction recovers the finite range of dependence, confirming that it captures the model’s dependence structure. We also derive Glivenko–Cantelli results that, to our knowledge, are new for this model.

\subsection{Network autoregressive model}

\paragraph{Setup} Consider a set of firms indexed by $N_n$. Let $Z_i$ denote the log output of firm $i$, and write
\begin{align}
	Z
	=
	\rho W Z+\xi,
	\qquad
	\xi
	:=
	(\xi_i)_{i\in N_n},
	\qquad
	Z
	:=
	(Z_i)_{i\in N_n}.
	\label{eqn:production.network.ar}
\end{align}
The vector $\xi$ collects firm-specific primitive shocks. In the production-network interpretation, one may write $\xi_i
	= X_i^{\top}\beta+\eta_i$,  where $X_i\in\mathbb R^p$ contains observable firm-level production shifters, such as factor prices, input availability, or demand conditions, and $\eta_i$ is an unobserved productivity or technology shock. We assume that $(\xi_i)_{i\in N_n}$ are IID across firms with $\sup_{i\in N_n}
		\mathbb E \left[
		|\xi_i|
		\right]
		\leq
		M_{\xi}
		<
		\infty$.

Although the production-network interpretation is our leading example, the same specification also covers standard spatial autoregressive and social-interaction models; see \cite{BramoulleEtAl2009,LeSagePace2009,LiuEtAl2014}. For production networks, we based our model on \cite{AcemogluCarvalhoOzdaglarTahbazSalehi2012}, who study the propagation of idiosyncratic shocks through input--output linkages.

In the production-network interpretation, $w_{ij}\neq 0$ means that firm $j$ affects firm $i$; e.g., $j$ may be an upstream supplier of intermediate inputs used by $i$, or a firm whose output enters the production technology of $i$. The graph $\mathfrak G_n = (N_{n},V_{n})$ is the undirected skeleton of this directed network:
\[
	(i,j)\in V_n
	\quad\Longleftrightarrow\quad
	w_{ij}\neq 0
	\ \text{or}\
	w_{ji}\neq 0.
\]
Thus, $\mathfrak G_n$ captures production linkages, while $W$ retains their direction and strength. Under
\begin{align}
	w_{ij}
	&\geq
	0,
	\qquad
	\sup_{i\in N_n}
	\sum_{j\in N_n}
	w_{ij}
	\leq
	\kappa
	<
	\infty,
	\qquad
	\theta
	:=
	|\rho|\kappa
	<
	1,
	\label{eqn:production.network.stability}
\end{align}
it follows that
\begin{align}
	Z
	&=
	(I-\rho W)^{-1}\xi
	=
	\sum_{s=0}^{\infty}
	\rho^sW^s\xi,
	\label{eqn:production.network.resolvent}
\end{align}
so shocks propagate along network paths, with their effects decaying geometrically. 

\paragraph{Production layers}
We impose a hierarchical structure on the production network, partitioning \(N_n\) into \(L_n\) production layers: $N_n=\bigsqcup_{\ell=1}^{L_n}\mathcal L_\ell$ with $q_{n} : = \max_{1\leq \ell \leq L_{n}} |L_{\ell}|$. 
For instance, the first layer may contain firms producing basic inputs, such as raw materials or energy; the second, firms transforming them into components; and the last, firms producing or distributing final goods.

\begin{assumption}
For every \(i\in\mathcal L_\ell\) and \(j\in\mathcal L_u\),
\begin{align}
w_{ij}\neq0
\quad\Longrightarrow\quad
u\in\{\ell-D,\ldots,\ell\}.
\label{eq:layered-local-W-D}
\end{align}
\end{assumption}

Thus, \(D\) bounds the number of layers spanned by a direct input--output linkage: a firm may use inputs from its own layer or the preceding \(D\) layers. Dependence across more distant layers can therefore arise only through sequences of intermediate supplier--customer links.

This structure provides the two ingredients needed for our maximal inequality. The layers form a natural partition requiring at most \(Dr+1\) colors at separation radius \(r\), while \eqref{eqn:production.network.resolvent} implies that the coupling error across separated layers decays geometrically at rate \(\theta^r\). Increasing \(r\) therefore raises the coloring cost at most linearly while reducing the coupling error geometrically. As shown below, choosing \(r\) of logarithmic order yields a Glivenko--Cantelli result and an explicit lower bound on the effective sample size.

\paragraph{Object of interest}

Our goal is to bound
\begin{align}
\label{eqn:NAR.object}
\mathbb E\left[
\sup_{f\in\mathcal F}
\frac{1}{n}\sum_{i\in N_n}
\left\{
f(Z_i)-\mathbb E[f(Z_i)]
\right\}
\right].
\end{align}
When \(\mathcal F\) is the normalized unit Lipschitz class, this quantity is the Wasserstein-\(1\) distance between the empirical distribution of firm outcomes and the average marginal distribution \(n^{-1}\sum_{i\in N_n}P_{Z_i}\). Alternatively, \(\mathcal F\) may represent a family of welfare functions indexed by a policy or structural parameter, in which case \eqref{eqn:NAR.object} provides uniform control of sample average welfare.

We assume that \(\mathsf Z\subset\mathbb R^{d_z}\) is compact and, for some \(s>d_z\) and \(L<\infty\), $\mathcal F\subseteq\mathcal H^s(L)
:=
\left\{
f:\mathsf Z\to\mathbb R:
\|f\|_{\mathbb C^s(\mathsf Z)}\leq L
\right\}$ . 
Then every \(f\in\mathcal F\) is \(L\)-Lipschitz and, by standard complexity bounds, for every \(p\in[1,\infty]\),
\begin{align}
\label{eqn:NAR.complexity}
\gamma_{2,a}
\left(
\mathcal F,
\max_{i\in N_n}\|\cdot\|_{L^{2p}(P_{Z_i})}
\right)
\lesssim L,
\qquad
\gamma_{1,b}
\left(
\mathcal F,\|\cdot\|_\infty
\right)
\lesssim L.
\end{align}

\paragraph{Natural partition and coloring} The production layers define the natural partition,
\begin{align}
\mathcal P_{L_n}
=
\left\{
\mathcal L_1,\ldots,\mathcal L_{L_n}
\right\}.
\label{eq.partition.blocks}
\end{align}
For $r\geq1$, let $h_r:=Dr+1$ and define the color classes $\mathcal C_r
\left[
\mathcal P_{L_n}
\right]
:=
\left\{
C_k:
k\in[h_r],~
C_k\neq\varnothing
\right\}$, 
where $C_k
:=
\left\{
m\in[L_n]:
m
\equiv
k
\pmod{h_r}
\right\} = \{ j h_{r} + k \colon j=0,1,...   \} \cap [L_{n}]$. 

Any two distinct production layers in the same class have indices differing by at least $Dr+1$. Since a single link in $W$ can span at most $D$ production layers, any path connecting such layers must contain at least $r+1$ links. Hence, $\mathcal C_r[\mathcal P_{L_n}]$ is a proper $r$-coloring of the layer partition with 
\begin{align}
K_r
\leq
\min
\left\{
Dr+1,
L_n
\right\}.
\label{eqn:auto.chormatic}
\end{align}

\paragraph{Coupling bound} We next show that the coupling error associated with the preceding partition and coloring decays geometrically:
\begin{align}
\tau^{\mathrm{max}}_{\mathcal F}
\left(
\mathcal P_{L_n},
\mathcal C_r
\left[
\mathcal P_{L_n}
\right]
\right)
\leq
\frac{
2LM_{\xi}
}{
1-\theta
}
\theta^{r+1}.
\label{eqn:auto.tau}
\end{align}

Fix a color class $C_k$ with $m_{k,1}
<
\cdots
<
m_{k,|C_k|}$, $\ell\in[|C_k|]$, and set $m:=m_{k,\ell}$. Every predecessor of $\mathcal L_m$ in the same color class belongs to a production layer with index at most $m-h_r$. So,\footnote{With the convention that $A_{m,r}=\varnothing$ whenever $m\leq h_r$.} 
\begin{align*}
\operatorname{Pred}_{\ell}(C_k)
\subseteq
A_{m,r},~where~A_{m,r} :=
\bigcup_{u=1}^{m-h_r}
\mathcal L_u,
\end{align*}


The layered restriction on $W$ implies that, if $j\in\mathcal L_u$, then $Z_j$ is measurable with respect to primitive shocks in layers no later than $u$. Consequently, $\sigma
\left(
Z_j:
j\in
\operatorname{Pred}_{\ell}(C_k)
\right)
\subseteq
\sigma
\left(
\xi_j:
j\in
A_{m,r}
\right)$. 

We construct an independent version of $Z_{\mathcal L_m}$ by resampling the primitive shocks in $A_{m,r}$. Let $(\xi_i')_{i\in N_n}$ be an independent copy of $(\xi_i)_{i\in N_n}$ and define
\begin{align}
\xi_i^{k,\ell}
:=
\begin{cases}
\xi_i',
&
i\in A_{m,r},
\\
\xi_i,
&
i\notin A_{m,r},
\end{cases}
\qquad
Z^{k,\ell}
:=
(I-\rho W)^{-1}
\xi^{k,\ell}.
\label{eq}
\end{align}
Because the primitive shocks are IID, $Z_{\mathcal L_m}^{k,\ell}
\overset{d}{=}
Z_{\mathcal L_m}$ and $Z_{\mathcal L_m}^{k,\ell}
\indep
\sigma
\left(
Z_j:
j\in
\operatorname{Pred}_{\ell}(C_k)
\right)$.


Every shock in \(A_{m,r}\) lies at least \(Dr+1\) layers upstream of \(\mathcal L_m\), while each application of \(W\) spans at most \(D\) layers. Hence, no propagation chain of length \(s\leq r\) can connect \(A_{m,r}\) to \(\mathcal L_m\), so
\[
[W^s]_{ij}=0
\qquad
\text{for all }
i\in\mathcal L_m,\ j\in A_{m,r},\ \text{and }s\leq r.
\]
Thus, the resampled shocks affect \(Z_{\mathcal L_m}\) only through terms involving \(W^s\) with \(s\geq r+1\). 

Hence, using
\(\sum_{j\in N_n}[W^s]_{ij}\leq\kappa^s\),
\(\theta=|\rho|\kappa\), and the bound on $|\zeta_i|$,
\begin{align*}
\mathbb E\left[
\left|Z_i-Z_i^{k,\ell}\right|
\right] \leq \sum_{s=r+1}^{\infty}
\rho^s
\sum_{j\in A_{m,r}}
[W^s]_{ij} \mathbb E \left[ 
\left(\xi_j-\xi_j'\right)   \right]
&\leq
2M_\xi
\sum_{s=r+1}^{\infty}
|\rho|^s
\sum_{j\in N_n}[W^s]_{ij}
\leq
2 M_\xi
\sum_{s=r+1}^{\infty}\theta^s
\end{align*}

Because every $f\in\mathcal F$ is $L$-Lipschitz, $d_{\mathcal F,|\mathcal L_m|}
\left(
Z_{\mathcal L_m},
Z_{\mathcal L_m}^{k,\ell}
\right)
\leq
L
\sum_{i\in\mathcal L_m}
\left|
Z_i-Z_i^{k,\ell}
\right|$. The coupling characterization of the $\tau$-coefficient therefore gives
\begin{align*}
&
\frac{1}{|\mathcal L_m|}\tau_{d_{\mathcal F,|\mathcal L_m|}}
\left(
\sigma
\left(
Z_j:
j\in
\operatorname{Pred}_{\ell}(C_k)
\right),
Z_{\mathcal L_m}
\right) \leq 2 L M_\xi
\sum_{s=r+1}^{\infty}\theta^s
=
\frac{
2LM_{\xi}
}{
1-\theta
}
\theta^{r+1}.
\end{align*}
After taking the maximum over the blocks and color classes, we obtain \eqref{eqn:auto.tau}.

\paragraph{Maximal inequality} Given our choice of partition and coloring and \eqref{eqn:auto.chormatic} and \eqref{eqn:auto.tau}, it follows from Theorem \ref{thm:maximal}, for any $a,b \ge 1$ any $p\in[1,\infty]$, and any $r_{n}$,
\begin{align}
	\left\|
	\sup_{f\in\mathcal F}
	\left|
	\frac{1}{n}
	\sum_{i\in N_n}
	\left\{
	f(Z_i)
	-
	\mathbb E[f(Z_i)]
	\right\}
	\right|
	\right\|_{L^1(\mathbb P)} \lesssim &
	\sqrt{
		\frac{
			\min\{Dr_n+1,L_n\}
		}{
			n
		}
	}
	\max_{m\in[L_n]}
	\Gamma_p(\mathcal L_m)
	\gamma_{2,a}
	\left(
	\mathcal F,
	\max_{i\in N_n}
	\|\cdot\|_{L^{2p}(P_{Z_i})}
	\right) \nonumber
	\\
	&+
	\left\{
	\frac{
		q_n\min\{Dr_n+1,L_n\}
	}{
		n
	}
	\right\}
	\gamma_{1,b}
	\left(
	\mathcal F,
	\|\cdot\|_{\infty}
	\right) + 
	\theta^{r_n+1}	\label{eq:production-maximal-inequality}
\end{align}

By using the (admittedly crude) bound of $\Gamma_p(\mathcal L_{m}) \leq \sqrt{q_{n}}$ for $p=1$; by expression \eqref{eqn:NAR.complexity}; and by choosing $r_n
=
\left\lceil
\frac{
	\log n
}{
	2|\log (\rho \kappa)|
}
\right\rceil$ to balance the terms, one obtains 
\begin{align*}
	\left\|
	\sup_{f\in\mathcal F}
	\left|
	\frac{1}{n}
	\sum_{i\in N_n}
	\left\{
	f(Z_i)
	-
	\mathbb E[f(Z_i)]
	\right\}
	\right|
	\right\|_{L^1(\mathbb P)} \lesssim
	\sqrt{ 
		\frac{ q_{n} 
			\min\{D \log n,L_n\}
		}{
			n
		}
	}
 +
	\left\{
	\frac{
		q_n\min\{D \log n,L_n\}
	}{
		n
	}
	+
\frac{1}{\sqrt{n}} \right\}. 
\end{align*}

Therefore, by Theorem \ref{thm:single-scale-network-gc} the autoregressive network with a H\"older ball is graph Glivenko-Cantelli provided the maximal production layer size does not grow too fast, i.e., $	\frac{q_n\min\{D \log n,L_n\}}{n} = o(1)$, and it has an effective sample size satisfying the following bound $\mathfrak{n}_{\mathrm{eff}} \geq 	\frac{n}{q_n\min\{D \log n,L_n\}}$. 

Thus, the loss relative to the IID benchmark is governed by two features of the layered production structure: the maximal number of firms within a production layer and the log coloring cost required to separate blocks sufficiently far for the autoregressive dependence to become negligible.

\subsection{Nonlinear smooth local propagation model}


Let \((\varepsilon_j)_{j\in N_n}\) be IID shocks, and suppose
\begin{align}\label{eqn:Nonlinear.model}
Y_i = H_i\big((\varepsilon_j)_{j\in \mathcal N (i,R)}\big),
\end{align} 
for some $R > 0$ and $\mathcal N(i,R) :  = \{ j \in N_{n} \colon d_{n}(i,j) \leq R \}$. Assume that the map \(H_i\) satisfies the coordinatewise Lipschitz bound
\begin{align} \label{eqn:Nonlinear.lip}
|H_i(\varepsilon)-H_i(\varepsilon')|
\le
\sum_{j\in \mathcal N(i,R)} c_{ij}|\varepsilon_j-\varepsilon_j'|.
\end{align} 

This model allows nonlinear propagation of shocks within a local neighborhood. The condition \eqref{eqn:Nonlinear.lip} is a coordinatewise Lipschitz condition on the structural map \(H_i\). It requires the effect of perturbing the primitive shock vector to be bounded by the sum of the coordinatewise perturbations, with weights \(c_{ij}\). 
An important case is the network-interference model in estimation of treatment effects which we now discuss.

\paragraph{Network interference in treatment-effect models}

Under network interference, one unit's treatment may affect the outcomes of network-connected units; see \cite{hudgens2008toward,aronow2017estimating,leung2020treatment}. Following the exposure-mapping approach of \cite{aronow2017estimating}, consider
\begin{align}
	Y_i
	&=
	\alpha
	+
	\tau_0D_i
	+
	\sum_{s=1}^{R}
	\rho_s\bar D_i(s)
	+
	U_i,
	\qquad
	\bar D_i(s)
	=
	\frac{1}{|S_n(i,s)|}
	\sum_{j\in S_n(i,s)}
	D_j,
	\label{eqn:network.interference.model}
\end{align}
where $D_i\sim\operatorname{Bern}(\pi)$ is the treatment status,
$S_n(i,s):=\{j\in N_n:d_n(i,j)=s\}$, and $U_i$ is an idiosyncratic shock. We assume that $(D_i,U_i)_{i\in N_n}$ is IID across nodes and that $D_i$ and $U_i$ are independent. Here, $\tau_0$ is the direct effect of own treatment and $\rho_s$ is the spillover effect of average treatment exposure at distance $s$. This is a finite-range version of the network-spillover models in \cite{leung2020treatment}: imposing $\rho_s=0$ for $s>R$ makes $Z_i=(Y_i,D_i,U_i)$ measurable with respect to $\mathcal N(i,R)$.\footnote{We set $\bar D_i(s)=0$ whenever $S_n(i,s)=\varnothing$.}

This model is a particular instance of the nonlinear smooth propagation model, where the primitive shock is given by 
\(\varepsilon_j=(D_j,U_j)\), and\footnote{One can add exogenous covariates, $X_{j}$; we omit these for the sake of presentation.}
\[
c_{ij}
=
\max\!\left\{|\tau_0|,1\right\}
\mathbf 1\{j=i\}
+
\mathbf 1\!\left\{1\leq d_n(i,j)\leq R\right\}
\frac{
	|\rho_{d_n(i,j)}|
}{
	|S_n(i,d_n(i,j))|
}.
\]

\subsubsection{$\tau$-mixing coefficient}

For model \eqref{eqn:Nonlinear.model}--\eqref{eqn:Nonlinear.lip}, finite-range propagation implies that
\begin{align}
\label{eqn:nonliner-bar-tau}
\tau_{\mathcal B}(r,q)=0
\qquad \forall r\geq 2R+1,~\forall q \geq 1, \forall \mathcal B\subseteq \mathcal F.
\end{align}
Indeed, if \(d_n(A,B)>2R\), the \(R\)-neighborhoods of \(A\) and \(B\) are disjoint by the triangle inequality. Hence, \(Z_A\) and \(Z_B\) are measurable with respect to disjoint collections of IID primitive shocks and are therefore independent. 
The same conclusion holds for \(\bar\tau_{\mathcal B}\) and \(\tau_{\mathcal B}^{\max}\) under any partition-coloring pair.

\subsubsection{Maximal inequality and Glivenko--Cantelli results: Application to uniform convergence of the least-squares criterion}

Let $X_i
	:=
	\left(
	1,
	D_i,
	\bar D_i(1),
	\ldots,
	\bar D_i(R)
	\right)^{\top}$ and $\theta_0
	:=
	\left(
	\alpha_0,
	\tau_0,
	\rho_{0,1},
	\ldots,
	\rho_{0,R}
	\right)^{\top}$. Suppose $\theta_{0}$ belongs to a compact subset $\Theta \subset\mathbb R^{R+2}$.
	
	The network interference model in \eqref{eqn:network.interference.model} can be cast as  $Y_i
	=
	X_i^{\top}\theta_0
	+
	U_i$.  Since $U_{i}$ is IID, for estimation of $\theta_0$ one can consider the OLS estimator which minimizes 
\begin{align*}
	Q_n(\theta)
	:= \mathbb P_{n} \ell_{\theta} : = 
	\frac{1}{n}
	\sum_{i\in N_n}
	\ell_\theta(Y_i,X_i),~with~	\ell_\theta(Y_i,X_i)
	:=
	\left(
	Y_i-X_i^{\top}\theta
	\right)^2.
\end{align*}
In order to establish the asymptotic properties of the estimator, it is useful to control 
\begin{align*}
	\mathbb E\left[
\sup_{\theta\in\Theta}
\left|
Q_n(\theta)-Q_n^0(\theta)
\right|
\right],~with~
	Q_n^0(\theta) : = \mathbb P^{0}_{n} \ell_{\theta}
	:=
	\frac{1}{n}
	\sum_{i\in N_n}
	\mathbb E\left[
	\ell_\theta(Y_i,X_i)
	\right].
\end{align*}
	
To do this, we employ Theorem \ref{thm:maximal}. A technical hurdle is that $\ell_{\theta}$ is Lipschitz, but not with uniformly bounded constant (as it depends on the data). To circumvent this issue, for any $M>0$, we split 
$	\ell_\theta
=
\ell_{\theta,M}
+
r_{\theta,M}$ with 
\begin{align*}
	\ell_{\theta,M}(Y_i,X_i)
	:=
	\ell_\theta(Y_i,X_i)
	\mathbf 1\left\{
	|Y_i|\le M
	\right\},~and~
	r_{\theta,M}(Y_i,X_i)
	:=
	\ell_\theta(Y_i,X_i)
	\mathbf 1\left\{
	|Y_i|>M
	\right\}.
\end{align*}
We also assume that, for some $\delta>0$,
\begin{align}\label{eqn:uniform.outcome.moment}
	\max_{i\in N_n}
	\mathbb E\left[
	|Y_i|^{2+\delta}
	\right]
	\le
	C_Y
	<
	\infty.
\end{align}

Since $\|X_i\|_2\le \sqrt{R+2}$, $\sup_{\theta\in\Theta}
	\ell_\theta(Y_i,X_i)
	\le
	2Y_i^2
	+
	2C_\Theta^2(R+2)$, where   $C_\Theta
	:=
	\sup_{\theta\in\Theta}
	\|\theta\|_2
	<
	\infty$.  This result  and  \eqref{eqn:uniform.outcome.moment} imply $\max_{i\in N_n}
	\mathbb E\left[
	\sup_{\theta\in\Theta}
	r_{\theta,M}(Y_i,X_i)
	\right]
	\lesssim
	M^{-\delta}$. 
Therefore,
\begin{align}\label{eqn:loss.tail.bound}
	\mathbb E\left[
	\sup_{\theta\in\Theta}
	\left|
	(P_n-P_n^0)
	r_{\theta,M}
	\right|
	\right]
	\lesssim
	M^{-\delta}.
\end{align}

For any $\theta,\theta'\in\Theta$, $\left|
	\ell_{\theta,M}(Y_i,X_i)
	-
	\ell_{\theta',M}(Y_i,X_i)
	\right|
	\le
	2\sqrt{R+2}
	\left\{
	M+
	C_\Theta\sqrt{R+2}
	\right\}
	\|\theta-\theta'\|_2$. 
Consequently, for $\mathcal L_M
	:=
	\left\{
	\ell_{\theta,M}:
	\theta\in\Theta
	\right\}$,
\begin{align*}
	\gamma_{2,a}
	\left(
	\mathcal L_M,
	\max_{i\in N_n}
	\|\cdot\|_{L^{2p}(P_{(Y_i,X_i)})}
	\right)
	&\lesssim M 
	\gamma_{2,a}
	\left(
	\Theta,
	\|\cdot\|_2
\right),~and~
	\gamma_{1,b}
	\left(
	\mathcal L_M,
	\|\cdot\|_\infty
	\right)
	&\lesssim
	M
	\gamma_{1,b}
	\left(
	\Theta,
	\|\cdot\|_2
	\right).
\end{align*}
Moreover, truncation does not alter the finite-range dependence structure. Thus, by \eqref{eqn:nonliner-bar-tau}, $\tau_{\mathcal L_M}(r_n,1)
	=
	0$  with  $r_n=2R+1$. Let $K_{r_{n}}$ denote the associated chromatic number. By the
	greedy-coloring bound, $K_{r_n} = \chi(\mathfrak G_n^{\,r_n})\le 1+\Delta(\mathfrak G_n^{\,r_n})$, 
	where the graph $\mathfrak G_n^{\,r_n}$ has the same vertex set as
	$\mathfrak G_n$, but two distinct vertices $i$ and $j$ are connected
	if and only if their distance in the original graph $\mathfrak G_n$
	is at most $r_n$, that is, $d_n(i,j)\le r_n$.
	
	Since $\Theta$ is a bounded subset of
$\mathbb R^{R+2}$, the preceding complexity bounds and Theorem \ref{thm:maximal} imply
\begin{align}\label{eqn:truncated.loss.maximal.bound}
	\mathbb E\left[
	\sup_{\theta\in\Theta}
	\left|
	(P_n-P_n^0)
	\ell_{\theta,M}
	\right|
	\right]
	\lesssim
	(M+1)
	\left\{
	\sqrt{\frac{ \chi (\mathfrak G_n^{2R+1}) }{n}}
	+
	\frac{ \chi (\mathfrak G_n^{2R+1})  }{n}
	\right\}.
\end{align}
Combining \eqref{eqn:loss.tail.bound} and \eqref{eqn:truncated.loss.maximal.bound} gives
\begin{align*}
	\mathbb E\left[
\sup_{\theta\in\Theta}
\left|
Q_n(\theta)-Q_n^0(\theta)
\right|
\right]
	\lesssim
	(M+1)
	\left\{
	\sqrt{\frac{ \chi (\mathfrak G_n^{2R+1})  }{n}}
	+
	\frac{ \chi (\mathfrak G_n^{2R+1}) }{n}
	\right\}
	+
	M^{-\delta}.
\end{align*}
Therefore, provided the chromatic number does not grow too fast --- formally $\frac{  \chi (\mathfrak G_n^{2R+1})  }{n} \to 0 $ ( a sufficient condition is for the degree $\Delta (\mathfrak G_n^{2R+1})$ not to grow too fast ) ---  then, for any sequence $M_n\to\infty$ satisfying $M_n
	\sqrt{\frac{  \chi (\mathfrak G_n^{2R+1})  }{n}} \to
	0$, 
it follows that
\begin{align*}
	\sup_{\theta\in\Theta}
	\left|
	\frac{1}{n}
	\sum_{i\in N_n}
	\ell_\theta(Y_i,X_i)
	-
	\frac{1}{n}
	\sum_{i\in N_n}
	\mathbb E\left[
	\ell_\theta(Y_i,X_i)
	\right]
	\right|
	=
	o_\mathbb{P}(1).
\end{align*}

\bibliographystyle{plainnat}

\bibliography{CoupleNetworks.bib}

\appendix

\section{Appendix: Proofs of Theorems \ref{thm:coupling}, \ref{thm:maximal}, \ref{thm:single-scale-network-gc}, and Corollary \ref{cor:finite-dependency-graph}}

\subsection{Proof of Theorem \ref{thm:coupling}}
\label{sec:coupling.proof}

The following lemma is a re-statement of Lemma 1 in \cite{dedecker2006inequalities} and is here merely for completeness; we refer the reader to that  paper for a proof.


\begin{lemma}\label{lem:couple.strong}
	Let $\mathsf X$ be a Polish equipped with  pseudo-metric space $d$, let $X$ be an
	$\mathsf X$-valued random element, and let $\mathcal M$ be a
	sub-$\sigma$-field. Suppose that, on a suitable extension of the
	underlying probability space, there exists a random variable $U\sim \mathrm{Unif}[0,1]$ independent of $\sigma(X)\vee\mathcal M$.
	
	Then there exists an $\mathsf X$-valued random element $X^\ast$,
	measurable with respect to $\sigma(U)\vee\sigma(X)\vee\mathcal M$, 
	such that
	\begin{enumerate}
		\item
		$X^\ast \stackrel{d}{=} X$;
		
		\item $X^\ast \perp\!\!\!\perp \mathcal M$;
		
		\item
		$\mathbb E\!\left[d(X,X^\ast)\right]
		=
		\tau_d(\mathcal M,X)$. 
	\end{enumerate}
\end{lemma}


	Fix $k\in [K_r]$. Let $C_k$ be denoted as 
	\[
	 C_k=\{m_{k,1},\ldots,m_{k,L_k}\},~
	m_{k,1}<\cdots<m_{k,L_k}~and~m_{k,i} \in [M].
	\]
	
	For $\ell\in [L_{k}]$, define	$X_{k,\ell}:=Z_{B_{m_{k,\ell}}}$, 
	viewed as a random element of $\mathsf Z^{|B_{m_{k,\ell}}|}$ equipped with the  metric $	d_{\mathcal{B},|B_{m_{k,\ell}}|}$. Also define the past $\sigma$-field within color $k$ by
	\[
	\mathcal M_{k,\ell-1}
	:=
	\sigma\!\left(
	Z_j: j\in \bigcup_{1 \leq u<\ell} B_{m_{k,u}}
	\right),
	\qquad \ell\ge 1,
	\]
	with the convention that $\mathcal M_{k,0}$ is the trivial $\sigma$-field.

	On a suitable extension of the original probability space, let $\left(U_{k,\ell}\right)_{k\in[K_r],\,\ell\in[L_k]}$ 
	be IID \(\operatorname{Unif}[0,1]\) random variables, independent of
	\(\sigma(Z_i:i\in N_n)\). For every \(k\in[K_r]\) and
	\(\ell\in[L_k]\), define
	\[
	\widetilde{\mathcal M}_{k,\ell-1}
	:=
	\mathcal M_{k,\ell-1}
	\vee
	\sigma(U_{k,1},\ldots,U_{k,\ell-1}),
	\]
	where \(\widetilde{\mathcal M}_{k,0}\) is trivial.
	
	Since
	$\sigma(U_{k,1},\ldots,U_{k,\ell-1})
	\perp\!\!\!\perp
	\sigma(X_{k,\ell})\vee\mathcal M_{k,\ell-1}$, for every bounded measurable function \(h\), $\mathbb E\left[
	h(X_{k,\ell})
	\mid
	\widetilde{\mathcal M}_{k,\ell-1}
	\right]
	=
	\mathbb E\left[
	h(X_{k,\ell})
	\mid
	\mathcal M_{k,\ell-1}
	\right]$ a.s. Hence,
	\begin{align}
		\tau_{d_{\mathcal B,|B_{m_{k,\ell}}|}}
		\left(
		\widetilde{\mathcal M}_{k,\ell-1},
		X_{k,\ell}
		\right) =
		\tau_{d_{\mathcal B,|B_{m_{k,\ell}}|}}
		\left(
		\mathcal M_{k,\ell-1},
		X_{k,\ell}
		\right).
		\label{eq:tau-enlarged-filtration}
	\end{align}
	
	For every \(k\in[K_r]\) and \(\ell\in[L_k]\), apply Lemma
	\ref{lem:couple.strong} with $X=X_{k,\ell}$, 
	$\mathcal M=\widetilde{\mathcal M}_{k,\ell-1}$, and  $U=U_{k,\ell}$. 
	This yields a random element \(X_{k,\ell}^{\ast}\), measurable with
	respect to $\sigma(U_{k,\ell})
	\vee
	\sigma(X_{k,\ell})
	\vee
	\widetilde{\mathcal M}_{k,\ell-1}$, 
	such that
	\begin{align}
		&X_{k,\ell}^{\ast}
		\stackrel{d}{=}
		X_{k,\ell},
		\label{eq:coupling-marginal}
		\\
		&X_{k,\ell}^{\ast}
		\perp\!\!\!\perp
		\widetilde{\mathcal M}_{k,\ell-1},
		\label{eq:coupling-independence}
		\\
		& \mathbb E\left[
		d_{\mathcal B,|B_{m_{k,\ell}}|}
		\left(
		X_{k,\ell}^{\ast},
		X_{k,\ell}
		\right)
		\right]
		=
		\tau_{d_{\mathcal B,|B_{m_{k,\ell}}|}}
		\left(
		\widetilde{\mathcal M}_{k,\ell-1},
		X_{k,\ell}
		\right).
		\label{eq:coupling-cost-enlarged}
	\end{align}
	Combining \eqref{eq:tau-enlarged-filtration} and
	\eqref{eq:coupling-cost-enlarged}, we obtain
	\begin{align}
		&
		\mathbb E\left[
		d_{\mathcal B,|B_{m_{k,\ell}}|}
		\left(
		X_{k,\ell}^{\ast},
		X_{k,\ell}
		\right)
		\right] =
		\tau_{d_{\mathcal B,|B_{m_{k,\ell}}|}}
		\left(
		\mathcal M_{k,\ell-1},
		X_{k,\ell}
		\right).
		\label{eq:coupling-cost}
	\end{align}

    For every \(k\in[K_r]\) and \(\ell\in[L_k]\), define $Z_{B_{m_{k,\ell}}}^{\ast}:=
	X_{k,\ell}^{\ast}$. 	Since \(\{C_1,\ldots,C_{K_r}\}\) is a partition of \([M]\), this defines
	\(Z_{B_m}^{\ast}\) for every \(m\in[M]\). By
	\eqref{eq:coupling-marginal}, $Z_{B_m}^{\ast}\stackrel{d}{=}Z_{B_m},~m\in[M]$. 

	We next verify mutual independence. For \(u<\ell\), $X_{k,u}^{\ast}$ is measurable with respect to $
	\sigma(U_{k,u})
	\vee
	\sigma(X_{k,u})
	\vee
	\widetilde{\mathcal M}_{k,u-1}$. 
	Since \(u<\ell\), $\sigma(U_{k,u})
	\vee
	\sigma(X_{k,u})
	\vee
	\widetilde{\mathcal M}_{k,u-1}
	\subseteq
	\widetilde{\mathcal M}_{k,\ell-1}$. 
	Therefore, $\sigma\left(
	X_{k,1}^{\ast},\ldots,X_{k,\ell-1}^{\ast}
	\right)
	\subseteq
	\widetilde{\mathcal M}_{k,\ell-1}$. 
	Thus, by \eqref{eq:coupling-independence}, $X_{k,\ell}^{\ast}
	\perp\!\!\!\perp
	\sigma\left(
	X_{k,1}^{\ast},\ldots,X_{k,\ell-1}^{\ast}
	\right)$.  An induction on \(\ell\) therefore shows that $\left(
	X_{k,1}^{\ast},\ldots,X_{k,L_k}^{\ast}
	\right) $ is mutually independent.

	It remains to establish the coupling bound. Fix \(k\in[K_r]\) and
	\(\ell\in[L_k]\), and recall that $\operatorname{Pred}_{\ell}( C_k)
	:=
	\bigcup_{1\leq u<\ell}B_{m_{k,u}}$ and observe that $\mathcal M_{k,\ell-1} = \sigma\!\left(
		Z_j: j\in \operatorname{Pred}_{\ell}( C_k)
		\right)$. 
	Therefore,	
	
	\begin{align}
		\tau_{d_{\mathcal{B},|B_{m_{k,\ell}}|}}(\mathcal M_{k,\ell-1},X_{k,\ell})
		&=
		\tau_{d_{\mathcal{B},|B_{m_{k,\ell}}|}} \!\left(
		\sigma\!\left(
		Z_j: j\in \operatorname{Pred}_{\ell}( C_k)
		\right),
		\, Z_{B_{m_{k,\ell}}}
		\right).
		\label{eq:tauGn-bound-proof}
	\end{align}
	Combining \eqref{eq:coupling-cost} and
\eqref{eq:tauGn-bound-proof}, we obtain
\[
\mathbb E\left[
d_{\mathcal{B},|B_{m_{k,\ell}}|}
\left(
Z_{B_{m_{k,\ell}}}^{\ast},
Z_{B_{m_{k,\ell}}}
\right)
\right]
=	\tau_{d_{\mathcal{B},|B_{m_{k,\ell}}|}} \!\left(
\sigma\!\left(
Z_j: j\in \operatorname{Pred}_{\ell}( C_k)
\right),
\, Z_{B_{m_{k,\ell}}}
\right). 
\]

By summing over $\ell \in [L_{k}]$ and over $k \in [K_r]$, it follows that 
{\small{\begin{align*}
	\sum_{k=1}^{K_r} \sum_{\ell = 1}^{L_{k}} \mathbb E  \left[ d_{ \mathcal{B},|B_{m}|}(Z^{\ast}_{B_{m}},Z_{B_{m}}) \right] =  \sum_{k=1}^{K_r} \sum_{\ell = 1}^{L_{k}} \tau_{	d_{\mathcal{B},|B_m|}} ( \sigma\!\left(
	Z_j: j\in \operatorname{Pred}_{\ell}( C_k)
	\right),
	\, Z_{B_{m_{k,\ell}}} ) =   n \bar{\tau}_{\mathcal B } (\{B_{1},...,B_{M}\} , \{C_{1} ,..., C_{K_r}\}),
\end{align*}}}
thereby proving the desired result. 	
	
Finally, since the coloring is proper, $d_n(B_{m_{k,u}},B_{m_{k,\ell}})>r$	for every $u<\ell$. Hence, $d_n(\operatorname{Pred}_{\ell}( C_k) ,B_{m_{k,\ell}})>r$. This implies that
	\begin{align}
		\tau_{d_{\mathcal{B},|B_{m_{k,\ell}}|}}(\mathcal M_{k,\ell-1},X_{k,\ell})
		&=
		\tau_{d_{\mathcal{B},|B_{m_{k,\ell}}|}} \!\left(
		\sigma\!\left(
		Z_j: j\in \operatorname{Pred}_{\ell}( C_k)
		\right),
		\, Z_{B_{m_{k,\ell}}}
		\right)
		\notag\\
		&  \le |B_{m_{k,\ell}}| \sup_{\substack{A,B\subseteq N_n\\ d_n(A,B)\ge r,\; |B| \le |B_{m_{k,\ell}}| }}   \frac{1}{|B|}	\tau_{d_{\mathcal{B},|B_{m_{k,\ell}}|}} \!\left(
		\sigma\!\left(
		Z_j: j\in A
		\right),
		\, Z_{B}
		\right)  \notag   \\
		&= |B_{m_{k,\ell}}| \tau_{\mathcal{B} }(r,|B_{m_{k,\ell}}| ).
	\end{align}
	
	Combining \eqref{eq:coupling-cost} and
	this expression, we obtain
	\[
	\mathbb E\left[
	d_{\mathcal{B},|B_{m_{k,\ell}}|}
	\left(
	Z_{B_{m_{k,\ell}}}^{\ast},
	Z_{B_{m_{k,\ell}}}
	\right)
	\right]
	\le
	|B_{m_{k,\ell}}|
	\tau_{\mathcal B }
	\left(
	r,|B_{m_{k,\ell}}|
	\right).
	\]
	This proves \eqref{eqn:tau.bound.max}.

%

\subsection{Proof of Theorem \ref{thm:maximal}} 
\label{sec:ProofMaximal} 

To prove the theorem we need to introduce some additional concepts. First,  a slightly more general concept of admissible sequence partition: $T^{\infty}(c,d) : = (T_{l}(c,d))_{l \in \mathbb N_0}$ is a sequence of partitions of $\mathcal F$ with $|T_{0}(c,d)| = 1$ and $|T_{l}(c,d)| \leq 2^{d2^{l/c}}$ for some $c,d \geq 1$. 

Second, it is useful to provide an upper bound on $\operatorname{Var}\left(  |B|^{-1/2} \sum_{i \in B} g(Z_{i})     \right)  $ for any block $B \subseteq N_{n}$ and any $g \in \mathcal{F}$. This bound will be key to defining the relevant topology for measuring complexity of the space $\mathcal F$.

\begin{lemma}\label{lem:VarBlock.upper.0}
	For any $\mathcal B \subseteq \mathcal F - \mathcal F$, any $g \in \mathcal B$, and any block $B \subseteq N_{n}$, 
	\begin{align*}
		\sqrt{	\operatorname{Var}\left(  |B|^{-1/2} \sum_{i \in B} g(Z_{i})     \right)  } \leq  \sqrt{2} ||g||_{B},
	\end{align*}
	where 
	\begin{align*}
		||g||_{B} : =  \sqrt{ \sum_{\rho \in R(B)}  	\frac{|S_{B}(\rho)|}{|B|}    \max_{i \in B} \int_{0}^{1}  1\left\{\alpha_{ \mathcal{F} , B }(\rho) \geq u	\right\} |Q_{|g(Z_{i})|}(u) |^{2} du }. 
	\end{align*}
\end{lemma}

\begin{proof}
	See Appendix \ref{app:maximal}.
\end{proof}

For a given admissible partition sequence and for each $k \in \mathbb{N}_{0}$, and for each $f \in \mathcal F$, let $\pi_{k} f $ be an element of $T(f,T_{k})$, where, recall, $T(f,T_{k})$ is the unique set in $T_{k}$ containing $f$. Let $\Pi_{k} \mathcal F$ be the set of all such elements and let $\Delta_{k} \mathcal F : = \Pi_{k} \mathcal F - \Pi_{k-1} \mathcal F$.

Henceforth, we use $\mathbb P^{\ast}$ and $\mathbb E^{\ast}$ to denote the probability measure and expectation corresponding to the coupled process $(Z_i^{\ast})_{i\in N_n}$.

For any partition $\mathcal{P}_{M}  : = \{ B_{1},\ldots,B_{M}   \}$, the associated coloring of this partition is denoted $\mathcal C_{r}[B_{1},\ldots,B_{M} ]$. For any partition $\mathcal{P}_{M}$ of $N_{n}$ we can split the sum in the previous display into a sum within blocks and between blocks, i.e.,  
\begin{align*}
	\mathbb{G}_{n}( f) = \frac{1}{\sqrt{n}} \sum_{m=1}^{M}   U_{m}( f),
\end{align*}
where the within block $B_{m}$ sum is $f \mapsto 	U_{m}(f) : = \sum_{i \in B_{m}} \{ f(Z_{i}) - \mathbb E [f(Z_{i})]\}$. 

For any $r$-coloring, $\mathcal{C}_{r}[\mathcal{P}_{M}] = \{ C_{1} , \ldots, C_{K_{r}}  \}$, it follows that 
\begin{align*}
	\mathbb{G}_n(f) = \sum_{k=1}^{K_{r}}	\mathbb{G}^{(k)}_n( f),~\text{where}~	\mathbb{G}^{(k)}_n( f) : =  \frac{1}{\sqrt{ | C_{k} |  } }  \sum_{m \in C_{k}}   \sqrt{\frac{  | C_{k} |  }{n} }  U_{m}( f).
\end{align*}


Let $(Z^{\ast}_{i})_{i \in N_{n}}$ be the stochastic process from Theorem \ref{thm:coupling} associated to the $r$-coloring, $\mathcal{C}_{r}[\mathcal{P}_{M}]$. Let $\mathbb{G}^{\ast}_{n} :  \mathcal F \rightarrow \mathbb{R}$ be given by 
\begin{align*}
	\mathbb{G}^{\ast}_n( f) = \sum_{k=1}^{K_{r}} 	\mathbb{G}^{(k),\ast}_n( f),~\text{where}~	\mathbb{G}^{(k),\ast}_n( f) : =  \frac{1}{\sqrt{ | C_{k} |  } }  \sum_{m \in C_{k}}   \sqrt{\frac{  | C_{k} |  }{n} }  U^{\ast}_{m}(   f),
\end{align*}
with $f \mapsto U^{\ast}_{m}(f) : =  \sum_{i \in B_{m}}  \{ f(Z^{\ast}_{i}) - \mathbb E [f(Z^{\ast}_{i})]\}$ --- observe that for any $i \in B_{m}$, $Z^{\ast}_{i}$ has the same distribution as $Z_{i}$, so $\mathbb E^{\ast} [f(Z^{\ast}_{i})] = \mathbb E [f(Z_{i})]$.

By these calculations and the triangle inequality, 
\begin{align}\label{eqn:maximal-1}
	\left \| \sup_{f\in\mathcal{F}}  \mathbb G_{n}( f )   \right\|_{L^{1}(\mathbb P)} \leq 	\left \| \sup_{f\in\mathcal{F}}  \mathbb G^{\ast}_{n}( f )   \right\|_{L^{1}(\mathbb P^{\ast})}+ 	\left \| \sup_{f\in\mathcal{F}}  \{ \mathbb G_{n}( f )  - 	\mathbb{G}^{\ast}_n( f) \}   \right\|_{L^{1}(\Pr)}.
\end{align}

We now proceed to bound the two terms in the RHS of this expression individually. 

\subsubsection{ Bound for $ 	\left \| \sup_{f\in\mathcal{F}}   \{ \mathbb G_{n}( f )  - 	\mathbb{G}^{\ast}_n( f) \}   \right\|_{L^{1}(\Pr)} $ } Observe that 
\begin{align*}
	\left \|  \sup_{f\in\mathcal{F}} \{	\mathbb{G}_n( f)  - 	\mathbb{G}^{\ast}_n(  f) \}  \right \|_{L^{1}(\Pr)} = & \left \|  \sup_{f\in\mathcal{F}} \sum_{k=1}^{K_{r}} \{	\mathbb{G}^{(k)}_n( f)  - 	\mathbb{G}^{(k),\ast}_n( f) \} \right \|_{L^{1}(\Pr)} \\
	= &  \left \|  \sup_{f\in\mathcal{F}} n^{-1/2} \sum_{k=1}^{K_{r}} \sum_{m \in C_{k}} \{	U_{m}( f)  -  U^{\ast}_{m}( f) \}  \right \|_{L^{1}(\Pr)} \\
	\leq & n^{-1/2} \sum_{k=1}^{K_{r}} \sum_{m \in C_{k}}  \mathbb E_{\Pr} \left[    \sup_{h \in  \mathcal{F}} 	\sum_{i \in B_{m}} |h(Z_{i})  - h(Z^{\ast}_{i}) | \right] \\
	= & n^{-1/2} \sum_{k=1}^{K_{r}} \sum_{m \in C_{k}} \mathbb E_{\Pr} \left[ d_{ \mathcal{F},|B_{m}|}(Z_{B_{m}},Z^{\ast}_{B_{m}}) \right],
\end{align*}
because $d_{ \mathcal{F}, |B_{m}|}(Z_{B_{m}},Z^{\ast}_{B_{m}}) =   \sup_{h \in  \mathcal{F}} 	\sum_{i \in B_{m}} |h(Z_{i})  - h(Z^{\ast}_{i}) |$.
%

By Theorem \ref{thm:coupling}, 
\begin{align*}
	n^{-1/2} \sum_{k=1}^{K_{r}} \sum_{m \in C_{k}} \mathbb E_{\Pr} \left[ d_{ \mathcal{F},|B_{m}|}(Z_{B_{m}},Z^{\ast}_{B_{m}}) \right] \leq \sqrt{n} \bar{\tau}_{ \mathcal F}(\mathcal P_{M},\mathcal{C}_{r}[\mathcal{P}_{M}] ),
\end{align*}
and thus 
\begin{align}\label{eqn:coupling.empirical.bound}
	\left \|  \sup_{f\in\mathcal{F}}  \{	\mathbb{G}_n( f)  - 	\mathbb{G}^{\ast}_n( f) \}  \right \|_{L^{1}(\Pr)} \leq  \sqrt{n} \bar{\tau}_{  \mathcal F }(\mathcal P_{M},\mathcal{C}_{r}[\mathcal{P}_{M}] )  
\end{align}

\subsubsection{ Bound for $\left \|  \sup_{f\in\mathcal{F}}	 \mathbb{G}^{\ast}_n( f)    \right \|_{L^{1}(\mathbb P^{\ast})}$  } Observe that  $\left \| \sup_{f \in\mathcal{F}} \mathbb G^{\ast}_{n}(f)   \right\|_{L^{1}(\mathbb P^{\ast})}   \leq 	\left \| \sup_{f,f_0 \in\mathcal{F}} \mathbb G^{\ast}_{n}(f -f_0)   \right\|_{L^{1}(\mathbb P^{\ast})}  $ because, since $0 \in \mathcal F$, one can set $f_0=0$. 

By applying the telescoping sum, $f-f_0 = \sum_{l=1}^{\infty} \Delta_l f$ where $\Delta_l f : = \pi_l f - \pi_{l-1} f$ and $\pi_0 f = f_0$.

\begin{align*}
	\left \|  \sup_{f\in\mathcal{F}} \sum_{l=1}^{\infty}	\mathbb{G}^{\ast}_n(\Delta_l f)     \right \|_{L^{1}(\mathbb P^{\ast})} \leq \left\|  \sup_{f\in\mathcal{F}} \sum_{l=1}^{\infty}	\sum_{k=1}^{K_{r}}  \mathbb{G}^{(k),\ast}_n(\Delta_l f)     \right \|_{L^{1}(\mathbb P^{\ast})} = \left\|  \sup_{f\in\mathcal{F} } \sum_{l=1}^{\infty}	\sum_{k=1}^{K_{r}}    \sum_{m \in C_{k}}   \sqrt{  \frac{ 1  }{n} }  U^{\ast}_{m}(\Delta_{l}  f)     \right \|_{L^{1}(\mathbb P^{\ast})} 
\end{align*}

To bound this term we follow the steps in the proof of Lemma 3.1 in \cite{Pouzo2026} which rest on the results by \cite{talagrand2014}. For that we need to establish a Bernstein type inequality for the coupled color-class empirical process $g \mapsto  \frac{1}{\sqrt{ | C_{k} |  }}    \sum_{m \in C_{k}}   \sqrt{  \frac{  | C_{k} |  }{n} }  U^{\ast}_{m}(g)   $.

\begin{lemma}\label{lem:bern}
	Let $q : = \max_{m \in [M]} |B_{m}|$. For any fixed measurable function \(g \in \mathcal F - \mathcal F\) and any $n \in \mathbb{N}$ and  \(t>0\),
	\begin{align*}
		\mathbb P^{\ast} \!\left( 	\left|\mathbb G_{n}^{(k),\ast}(g)\right|  \ge t
		\right)
		\le  2\exp\left(
		-\frac{t^2}{
			4 \sum_{m\in\mathcal C_k} \frac{|B_{m}|}{n}  ||g||^{2}_{B_{m}}
			+\frac{4}{3}  \frac{q}{\sqrt{n}} ||g||_{\infty}   t   
		}
		\right).
	\end{align*}

Moreover, for every \(x>0\),
	\begin{align}
		\mathbb P^{\ast} \!\left(
		\left|\mathbb G_{n}^{(k),\ast}(g)\right|
		\ge
		2 \sqrt{x}  	\sqrt{  \sum_{m\in\mathcal C_k} \frac{|B_{m}|}{n}  ||g||^{2}_{B_{m}}   }
		+\frac{2}{3}  \frac{q}{\sqrt{n}}    ||g||_{\infty} x
		\right)
		\le 2e^{-x}.
	\end{align}
\end{lemma}

\begin{proof}
	See Appendix \ref{app:maximal}. 
\end{proof}

We can now establish the following bound for the empirical process $\mathbb{G}^{(k),\ast}_{n}$.
\begin{lemma}\label{lem:bound.pr.talagrand}
	Let $T^{\infty}(c,d) : = (T_{l}(c,d))_{l \in \mathbb N_0}$ be a sequence of admissible partitions of $\mathcal F$ such that $|T_{0}(c,d)| = 1$ and $|T_{l}(c,d)| \leq 2^{d2^{l/c}}$ for some $c,d \geq 1$. Suppose there exists functions $u : \mathbb{N}_{0} \to \mathbb{R}_{+}$ and $\ell : \mathbb{N}_{0} \to \mathbb{R}_{+}$ such that $\sup_{f \in \mathcal{F}}  ||\Delta_{l} f||_{\infty} \leq u(l)$ and $\ell(l)^{2} \geq 4 d 2^{l/c}$ for all $l \in \mathbb{N}_{0}$. Then,\footnote{The implicit constant in the $\lesssim$ can depend on $(c,d)$ but otherwise is universal.} 
	{\small{	\begin{align*}
				&\mathbb P^{\ast} \left(  \sup_{f \in \mathcal{F}}  \sum_{l=1}^{\infty} \sum_{k=1}^{K_{r}} |\mathbb G^{(k),\ast}_{n}(\Delta_{l} f) | \leq v   \sup_{f \in \mathcal{F}}  \sum_{l=1}^{\infty}  \sum_{k=1}^{K_{r}}  \left( 2  	\sqrt{  \sum_{m\in\mathcal C_k} \frac{|B_{m}|}{n}  ||\Delta_l f||^{2}_{B_{m}}   } \ell(l)   + \frac{2}{3} \frac{ \max_{m \in [M]} |B_{m}| }{\sqrt{ n }} ||\Delta_l f||_{\infty} (\ell(l)  )^{2}    \right)   \right) \\
				& \geq 1 - C \left(  \sum_{l=0}^{\infty} 2^{  2d 2^{l/c}  }  e^{- \ell(l)^{2} } \right) e^{-0.25 v}
			\end{align*}
	}}
	for any $v \geq 4$ and some universal finite constant $C$.\footnote{Observe that under the choice of $(c,d,\ell)$, the quantity $\sum_{l=0}^{\infty} 2^{  2d 2^{l/c}  }  e^{- 2\ell(l)^{2} }$ is finite.}  
\end{lemma}

\begin{proof}
	See Appendix \ref{app:maximal}.
\end{proof}

An immediate implication of this result is that 
\begin{align*}
	& \left \|  \sup_{f \in \mathcal{F}}  \sum_{l=1}^{\infty} \sum_{k=1}^{K_{r}} |\mathbb G^{(k),\ast}_{n}[\Delta_{l} f] |    \right \|_{L^{1}(\mathbb P^{\ast})} \\
	\lesssim & \sup_{f \in \mathcal{F}}  \sum_{l=1}^{\infty}  \sum_{k=1}^{K_{r}}  \left(  	\sqrt{  \sum_{m\in\mathcal C_k} \frac{|B_{m}|}{n}  ||\Delta_l f ||^{2}_{B_{m}}   } \ell(l)   + \frac{ \max_{m \in [M]} |B_{m}|  }{\sqrt{ n }}    ||\Delta_l f||_{\infty}  (\ell(l)  )^{2}    \right) \\
	\simeq & \sup_{f \in \mathcal{F}}  \sum_{l=1}^{\infty}   \left( \ell(l) \sum_{k=1}^{K_{r}}   	\sqrt{  \sum_{m\in\mathcal C_k} \frac{|B_{m}|}{n}  ||\Delta_l f ||^{2}_{B_{m}}   }   +  K_{r} \frac{ \max_{m \in [M]} |B_{m}|   }{\sqrt{ n }}    ||\Delta_l f||_{\infty} (\ell(l)  )^{2}    \right).
\end{align*}

Observe that
\begin{align*}
	\sum_{k=1}^{K_{r}}  \sqrt{  \sum_{m\in\mathcal C_k} \frac{|B_{m}|}{n}  ||\Delta_l f||^{2}_{B_{m}}   } \leq K_{r} \sqrt{ K_{r}^{-1} \sum_{k=1}^{K_{r}}  \sum_{m\in\mathcal C_k} \frac{|B_{m}|}{n}  ||\Delta_l f||^{2}_{B_{m}}    } \leq \sup_{m \in [M]} ||\Delta_l f ||_{B_{m}}   \sqrt{K_{r}} 
\end{align*}
where the first inequality follows by Jensen's inequality and the second follows because $\sum_{k=1}^{K_{r}} \sum_{m \in \mathcal C_{k}} |B_{m}| = n$.

Therefore, 
\begin{align}\label{eqn:bound.tala}
	\left \|  \sup_{f \in \mathcal{F}}  \sum_{l=1}^{\infty} \sum_{k=1}^{K_{r}} |\mathbb G^{(k),\ast}_{n}[\Delta_{l} f] |    \right \|_{L^{1}(\mathbb P^{\ast})} \lesssim  \sup_{f \in \mathcal{F}}  \sum_{l=1}^{\infty}   \left( \ell(l) \sup_{m \in [M]} || \Delta_{l} f ||_{B_{m}} \sqrt{K_{r}}  +  K_{r} \frac{\max_{m \in [M]} |B_{m}|    }{\sqrt{ n }}    ||\Delta_l f||_{\infty}   (\ell(l)  )^{2}    \right).
\end{align}

\subsubsubsection{Bound on $\left \| \sup_{f\in\mathcal{F}}  \mathbb G_{n}( f )   \right\|_{L^{1}(\mathbb P)} $}

By combining \eqref{eqn:coupling.empirical.bound} and \eqref{eqn:bound.tala} with \eqref{eqn:maximal-1}, it follows that
\begin{align*}
	& \left \| \sup_{f\in\mathcal{F}}  \mathbb G_{n}( f )   \right\|_{L^{1}(\mathbb P)} \\
	\lesssim & \sup_{f \in \mathcal{F}}  \sum_{l=1}^{\infty}  \left(   \sup_{m \in [M]} || \Delta_{l} f ||_{B_{m}} \sqrt{K_{r}}     \ell(l)   +  ||\Delta_l f||_{\infty}   \frac{  \max_{m \in [M]} |B_{m}|  }{\sqrt{ n }}  K_{r}  (\ell(l)  )^{2}  \right)  + \sqrt{n} \bar{\tau}_{  \mathcal F } (\mathcal P_{M} , \mathcal C_{r}[\mathcal P_{M}] )   
\end{align*}

Take $c = \min\{a,b\}$ and $\ell(l)^{2} = 4 d 2^{l/c}$. Since $a,b \ge 1$, $c \ge 1$ and thus it follows that $\ell(l) \lesssim 2^{l/2}$ for all $l \in \mathbb{N}_{0}$. Then, with $q = \max_{m \in [M]} |B_{m}|$, the previous display implies
\begin{align*}
	\left \| \sup_{f\in\mathcal{F}} \mathbb G_{n}(f )   \right\|_{L^{1}(\mathbb P)} 	\lesssim &  \sqrt{K_{r}} \sup_{f \in \mathcal{F}}  \sum_{l=1}^{\infty}   2^{l/2} \max_{m \in [M]} ||\Delta_{l} f ||_{B_m}    + K_{r}  \frac{ q }{\sqrt{ n }} \sup_{f \in \mathcal{F}}  \sum_{l=1}^{\infty}  2^{l}  ||\Delta_{l} f ||_{\infty}  \\
	&   + \sqrt{n}  \bar{\tau}_{ \mathcal F } \left(  \mathcal P_{M} ,  \mathcal  C_{r}[ \mathcal P_{M} ]  \right) .
\end{align*}

For any $B \subseteq N_{n}$, by \eqref{eqn:bound.Bnorm} in  the proof of Lemma \ref{lem:VarBlock.upper},
\begin{align*}
	||g||_{B}   \leq \Gamma_{p}(B)  \max_{i \in B}   || g ||_{L^{2p}(P_{Z_{i}}) }.
\end{align*}

Therefore, 
\begin{align*}
	\left \| \sup_{f\in\mathcal{F}} \mathbb G_{n}(f )   \right\|_{L^{1}(\mathbb P)} 	\lesssim &  \sqrt{K_{r}} \max_{m \in [M]} \Gamma_{p}(B_{m})  \sup_{f \in \mathcal{F}}  \sum_{l=1}^{\infty}   2^{l/2}  \max_{i \in N_{n}}  ||\Delta_{l} f ||_{L^{2p}(P_{Z_{i}})}    + K_{r}  \frac{  q  }{\sqrt{ n }} \sup_{f \in \mathcal{F}}  \sum_{l=1}^{\infty}  2^{l}  ||\Delta_{l} f ||_{\infty}  \\
	&   + \sqrt{n} \bar{\tau}_{ \mathcal F } \left(  \mathcal P_{M} ,  \mathcal C_{r}[ \mathcal P_{M} ]  \right), 
\end{align*}

We now choose the partition $T^{\infty}(c,d)  : =  ( T_{l}(c,d) )_{l \in \mathbb{N}_{0}}$. Let $T^{\infty}(a) \in \mathcal T_{a}(\mathcal F)$ and $T^{\infty}(b) \in \mathcal T_{b}(\mathcal F)$ be some arbitrary partition sequences. For any $l \geq 1$, let $T_{l}(c,d)$ be a partition comprised of sets of the form $A \cap B$ for $A \in T_{l-1}(a)$ and $B \in T_{l-1}(b)$.  Under this choice, for any $l \in \mathbb{N}$, $|T_{l}(c,d)| \leq 2^{2^{(l-1)/a} + 2^{(l-1)/b}}$. For $c = \min \{a,b\}$, $2^{(l-1)/a} + 2^{(l-1)/b} = 2^{l/c} 2^{-1/c} ( 2^{(l-1)(1/a-1/c)}  + 2^{(l-1)(1/b-1/c)}   ) \leq 2^{l/c} 2^{1-1/c} = d 2^{l/c}$. Hence, $|T_{l}(c,d)| \leq 2^{d2^{l/c}}$ for all $l \in \mathbb{N}_{0}$. Hence, $T^{\infty}(c,d)$ is indeed an admissible sequence partition. This construction holds for arbitrary $T^{\infty}(a) \in \mathcal T_{a}(\mathcal F)$ and $T^{\infty}(b) \in \mathcal T_{b}(\mathcal F)$, we now choose them so as to be (approximate) minimizers of the Talagrand's Complexity Measure in Definition \ref{def:talagrand}, i.e., 
\begin{align*}
	\sup_{f \in \mathcal{F}}   \sum_{l=1}^{\infty} 2^{l/2}  \max_{i \in N_{n}}   || \Delta_l f ||_{L^{2p}(P_{Z_{i}}) } \lesssim &  \gamma_{2,a}(\mathcal F ,    \max_{i \in N_{n}} ||\cdot  ||_{L^{2p}(P_{Z_{i}})}   ), \\
	\sup_{f \in \mathcal{F}}   \sum_{l=1}^{\infty} 2^{l}    || \Delta_l f ||_{\infty} \lesssim & \gamma_{1,b}(\mathcal F ,   ||\cdot||_{\infty}). 
\end{align*}

Hence,
\begin{align*}
	\left \| \sup_{f\in\mathcal{F}} \mathbb G_{n}(f )   \right\|_{L^{1}(\mathbb P)} 	\lesssim &  \sqrt{K_{r}} \max_{m \in [M]} \Gamma_{p}(B_{m})  \gamma_{2,a} ( \mathcal{F} ,  \max_{i \in N_{n}} ||\cdot  ||_{L^{2p}(P_{Z_{i}})} )   +     \frac{K_{r}  q}{\sqrt{n}}    \gamma_{1,b}(\mathcal F ,   ||\cdot||_{\infty}) \\
	& + \sqrt{n} \bar{\tau}_{ \mathcal F } \left(  \mathcal P_{M} ,  \mathcal C_{r}[ \mathcal P_{M} ]  \right) 
\end{align*}
and Theorem \ref{thm:maximal} follows. 

We conclude by pointing out that by Lemma \ref{lem:tau.properties},
\begin{align*}
	\bar{\tau}_{ \mathcal F } (\mathcal P_{M} , \mathcal C_{r}[\mathcal P_{M}] )  \leq 2 \sup_{ f \in \mathcal F } ||f||_{\infty}  \bar{\tau}_{ \operatorname{cone} \mathcal F } (\mathcal P_{M} , \mathcal C_{r}[\mathcal P_{M}] ). 
\end{align*}
And since our choice of partition,  $|| f||_{\infty} \lesssim  \gamma_{1,b}(\mathcal F ,   ||\cdot||_{\infty})  $ for all $f \in \mathcal{F}$ it follows that
\begin{align*}
	\left \| \sup_{f\in\mathcal{F}} \mathbb G_{n}(f )   \right\|_{L^{1}(\mathbb P)} 	\lesssim &  \sqrt{K_{r}} \max_{m \in [M]} \Gamma_{p}(B_{m})  \gamma_{2,a} ( \mathcal{F} ,  \max_{i \in N_{n}} ||\cdot  ||_{L^{2p}(P_{Z_{i}})} )    \\
	& +    \left(   \frac{K_{r}  q}{\sqrt{n}}  + \sqrt{n} \bar{\tau}_{\operatorname{cone} \mathcal F } \left(  \mathcal P_{M} ,  \mathcal C_{r}[ \mathcal P_{M} ]  \right)   \right)  \gamma_{1,b}(\mathcal F ,   ||\cdot||_{\infty}) .
\end{align*}

\subsection{Proof of Theorem \ref{thm:single-scale-network-gc}}
\label{sec:proofGC}

	Fix \(n\). Let $f \mapsto P^{0}_{n}f : = n^{-1} \sum_{i \in N_{n}} \mathbb E [f(Z_{i})] $. For every \(f\in\mathcal F\), let \(C(f)\in T_{L_n}\) be the
	cell containing \(f\). Then
	\begin{align*}
		(P_n-P_n^0)f
		=
		(P_n-P_n^0)\pi_{L_n}f
		+
		(P_n-P_n^0)(f-\pi_{L_n}f).
	\end{align*}
	Taking the supremum over \(f\in\mathcal F\), we obtain
	\begin{align}
		\sup_{f\in\mathcal F}
		|(P_n-P_n^0)f|
		\le
		\max_{\pi\in\Pi_{L_n}}
		|(P_n-P_n^0)\pi| +
		\sup_{f\in\mathcal F}
		|(P_n-P_n^0)(f-\pi_{L_n}f)|.
		\label{eq:ssgc-basic-decomp}
	\end{align}
	
	We control the first term by Theorem~\ref{thm:maximal}. Applying the theorem with $\mathcal F = \Pi_{L_{n}}$ and then dividing by
	\(\sqrt n\), yields
	\begin{align}
		\mathbb E
		\left[
		\max_{\pi\in\Pi_{L_n}}
		|(P_n-P_n^0)\pi|
		\right]
		&\lesssim
		\frac{
			\sqrt{K_{r_{n}}}\max_m\Gamma_p(B_m)
			\gamma_{2,a}
			\left(
			\Pi_{L_n},
			\max_i\|\cdot\|_{L^{2p}(P_{Z_i})}
			\right)
		}{\sqrt n}
		\nonumber\\
		&\quad+
		\left(
		\frac{K_{r_{n}} q_n}{n}
		\right)
		\gamma_{1,b}
		\left(
		\Pi_{L_n},
		\|\cdot\|_{\infty}
		\right) +	\bar{\tau}_{\Pi[T_{L_{n}}] }
		\left( 	\mathcal P_{n}, 	\mathcal C_{r_{n}}[\mathcal P_{n} ] \right)  .\label{eqn:GC.bound.0}
	\end{align}
	
	We now control the second term in
	\eqref{eq:ssgc-basic-decomp}. If \(f\in C\), then $|f(z)-\pi_{L_n}f(z)|
	\le
	D_C(z)$, $z\in \mathsf Z$. 
	Therefore, $|(P_n-P_n^0)(f-\pi_{L_n}f)|
		\le
		P_n|f-\pi_{L_n}f|
		+
		P_n^0|f-\pi_{L_n}f|
		\le
		P_n D_C + P_n^0 D_C$. 
	Taking suprema over \(f\in\mathcal F\), equivalently over
	\(C\in T_{L_n}\), gives
	\begin{align}
		\mathbb E \left[   	\sup_{f\in\mathcal F}
		|(P_n-P_n^0)(f-\pi_{L_n}f)| \right] 
		&\le \frac{2}{n} \sum_{i \in N_{n}}     \mathbb  E \left[  
		\sup_{C\in T_{L_n}}
		|D_C(Z_{i})| \right] 
		\label{eq:ssgc-diameter-bound}
	\end{align}
	
	If we show that the RHS of \eqref{eqn:GC.bound.0} and  \eqref{eq:ssgc-diameter-bound} vanish as $n$ diverges, the desired result follows by the Markov inequality. 
	
	By assumption, the RHS of  \eqref{eq:ssgc-diameter-bound}  vanishes as $n$ diverges, thus, it only remains to control the RHS of  \eqref{eqn:GC.bound.0}.  By Lemma \ref{lem:finite-class-gamma-bounds} with $T = \Pi_{L_{n}}$,
	{\small{\begin{align*}
		&	\frac{
			\sqrt{K_{r_{n}} }\max_m\Gamma_p(B_m)
			\gamma_{2,a}
			\left(  \Pi_{L_{n}}  ,
			\max_i\|\cdot\|_{L^{2p}(P_{Z_i})}
			\right)
		}{\sqrt n} \\
	& \quad \leq  \frac{
			\sqrt{K_{r_{n}} }\max_m\Gamma_p(B_m) 2^{[a L_{n}]/2} \operatorname{diam}(\Pi_{L_{n}} , \max_{i \in N_{n}} ||.||_{L^{2p}(P_{Z_{i}})})
		}{\sqrt n},  \\
		&	\left(
		\frac{K_{r_{n}} q_n}{n}
		\right)
		\gamma_{1,b}
		\left(  \Pi_{L_{n}} ,
		\|\cdot\|_{\infty}
		\right)  +
		\bar{\tau}_{ \Pi[T_{L_{n}}] }
		\left( 	\mathcal P_{n}, 	\mathcal C_{r_{n}}[\mathcal P_{n} ] \right)  \\
		& \quad	\leq \left(
		\frac{K_{r_{n}} q_n}{n}  
		\right) 2^{[b L_{n}]} \operatorname{diam}(\Pi_{L_{n}} ,  ||.||_{\infty}) + 
		\bar{\tau}_{ \Pi[T_{L_{n}}] }
		\left( 	\mathcal P_{n}, 	\mathcal C_{r_{n}}[\mathcal P_{n} ] \right) .
	\end{align*}}}
	
	Since $\sup_{ f \in \mathcal F } ||f||_{\infty} < \infty$, $\sup_{n \in \mathbb{N}} \operatorname{diam}(\Pi_{L_{n}} , \max_{i \in N_{n}} ||.||_{L^{2p}(P_{Z_{i}})} ) \leq \operatorname{diam}(\Pi_{L_{n}} ,  ||.||_{\infty}) < \infty$. Thus, the RHS in the previous display vanishes as $n$ diverges by condition in \eqref{eq:ssgc-net-gamma1}. 

    \subsection{Proof of Corollary \ref{cor:finite-dependency-graph}}
    \label{sec:dep-graph.proof}

    \begin{proof}[Proof of Corollary \ref{cor:finite-dependency-graph}]
Apply Theorem~\ref{thm:maximal} to the symmetrized class $\mathcal F_{\pm}
    :=
    \mathcal F\cup(-\mathcal F)$, 
and take the singleton partition  $\mathcal P_n
    :=
    \bigl\{\{i\}:i\in N_n\bigr\}$. A proper coloring of \(\mathfrak D_{n} \) partitions \(N_n\) into
\(K_1=\chi(\mathfrak D_{n})\) independent sets. Within each color class, every
variable is independent of the sigma-field generated by its
predecessors. It follows that the coupling error vanishes: $ \overline{\tau}_{\operatorname{cone}\mathcal F_{\pm}}
    \bigl(
        \mathcal P_n,
        \mathcal C_1[\mathcal P_n]
    \bigr)
    =0$. 
Moreover, because every block is a singleton, $\max_{B\in\mathcal P_n}|B|=
    \max_{B\in\mathcal P_n}\Gamma_1(B)=1$. 
Theorem~\ref{thm:maximal} therefore yields
\begin{align*}
    \left\|
        \sup_{f\in\mathcal F_{\pm}}
        \mathbb G_n(f)
    \right\|_{L^1}
    \lesssim
    \sqrt{\chi(\mathfrak D_{n})}\,
    \gamma_{2,a}
    \left(
        \mathcal F_{\pm},
        \max_{i\in N_n}
        \|\cdot\|_{L^2(P_{Z_i})}
    \right)
    +
    \frac{\chi(\mathfrak D_{n})}{\sqrt n}\,
    \gamma_{1,b}
    \left(
        \mathcal F_{\pm},
        \|\cdot\|_\infty
    \right).
\end{align*}
Since \(|\mathcal F_{\pm}|\leq 2J\), the standard finite-class
bounds for the generic chaining functionals imply
\begin{align*}
    \gamma_{2,a}
    \left(
        \mathcal F_{\pm},
        \max_{i\in N_n}\|\cdot\|_{L^2(P_{Z_i})}
    \right)
    &\lesssim
    \sigma\sqrt{\log(2J)}~and~
    \gamma_{1,b}
    \left(
        \mathcal F_{\pm},
        \|\cdot\|_\infty
    \right)
    \lesssim
    b\log(2J).
\end{align*}
Dividing the resulting bound by \(\sqrt n\) gives
\eqref{eq:finite-dependency-graph}. 
\end{proof}

\pagebreak

\setcounter{page}{1}

\begin{center}
	{\Huge{\textbf{Online Appendix}}}
\end{center}

\section{Appendix for Coupling Results}
\label{app:coupling}


The next lemma contains useful properties of the $\tau$-coefficient defined in Section \ref{sec:mixing}. In the lemma, let $\boldsymbol{\tau}_{\mathcal B}$ denote any of the three $\tau$-coefficients defined in \eqref{eq:average-coloring-adapted-tau}, \eqref{eq:max-coloring-adapted-tau}, and \eqref{eq:global-graph-tau}, respectively.

For any class on $\mathsf{Z}$ of real-valued measurable functions, $\mathcal B$, let  $\operatorname{cone} \mathcal B  : = \{ \lambda f \colon ||\lambda f ||_{\infty} \leq 1,~f \in \mathcal B~and~ \lambda > 0 \}$. 

\begin{lemma}\label{lem:tau.properties} 
	For any $\mathcal B \subseteq \mathcal F$, any partition $\mathcal P_{M}$, and any associated proper $r$-coloring $\mathcal C_{r}$ the following statements are true.
	\begin{enumerate}
		\item $\boldsymbol{\tau}_{\mathcal B } \leq \sup_{f \in \mathcal B} ||f||_{\infty} \boldsymbol{\tau}_{\operatorname{cone} \mathcal B }$, for any $r,q \in \mathbb{N}$. 
		\item $q \mapsto \tau_{  \mathcal B }(r,q)$ is non-decreasing.
		\item $\mathcal B \mapsto \boldsymbol{\tau}_{\mathcal B }$ is non-decreasing (under the inclusion ordering).
        \item For any $\mathcal B' \subseteq \mathcal F$, $\boldsymbol{\tau}_{\mathcal B - \mathcal B' } \leq 2 \boldsymbol{\tau}_{ \mathcal F } $.
	\end{enumerate}
\end{lemma}

\begin{proof}[Proof of Lemma \ref{lem:tau.properties}]
	
		\textsc{Part 1.} We first show that for any $B_{m_{k,\ell}} \in \mathcal P_{M}$, any $C_{k} \in \mathcal C_{r}$ and any $m_{k,\ell} \in C_{k}$, \begin{align*}
			\tau_{d_{\mathcal B,|B_{m_{k,\ell}}|}}
			\left(
			\sigma\left(
			Z_j:
			j\in\operatorname{Pred}_{\ell}( C_k)
			\right),
			Z_{B_{m_{k,\ell}}}
			\right) \leq  \sup_{ f \in \mathcal B } ||f||_{\infty}  \tau_{d_{\operatorname{cone} \mathcal B,|B_{m_{k,\ell}}|}}
			\left(
			\sigma\left(
			Z_j:
			j\in\operatorname{Pred}_{\ell}( C_k)
			\right),
			Z_{B_{m_{k,\ell}}}
			\right).
		\end{align*}
		
		Let $B : = \sup_{ f \in \mathcal B }||f||_{\infty}$. If this quantity is infinity there is nothing to prove, so we proceed under the assumption that it is finite. 
		
		Take any $g \in \operatorname{Lip}_{1}(\mathsf{Z}^{m} , d_{\mathcal B , m})$, then for any $z,z' \in \mathsf{Z}^{m}$, $|g(z) - g(z') | \leq  \sup_{ f \in \mathcal B }  \sum_{s=1}^{m} |f(z_{s}) - f(z'_{s})| = B  \sup_{ f \in \mathcal B } \sum_{s=1}^{m} |f(z_{s})/B - f(z'_{s})/B| $. Observe that for any $f \in \mathcal B$, $f/B \in \operatorname{cone} \mathcal B  : = \{ \lambda f \colon ||\lambda f ||_{\infty} \leq 1,~f \in \mathcal B~and~ \lambda > 0 \}$. Thus $\sup_{ f \in \mathcal B } \sum_{s=1}^{m}  |f(z_{s})/B - f(z'_{s})/B| \leq \sup_{ h \in \operatorname{cone}  \mathcal  B } \sum_{s=1}^{m}   |h(z_{s}) - h(z'_{s})|$. This readily implies the desired inequality.

	Since all instances of the $\tau$-coefficient --- $\bar{\tau}_{\mathcal B }
	\left(
	\mathcal P_{M},
	\mathcal C_{r}
	\right)$; $\tau^{\mathrm{max}}_{\mathcal B }
	\left(
	\mathcal P_{M},
	\mathcal C_{r}
	\right)$; and $\tau_{\mathcal B }
	\left( r, q
	\right)$ --- are increasing in $	\tau_{d_{\mathcal B,|B|}}$, the desired result follows.
		
		\bigskip  
	
		\textsc{Part 2.} By \eqref{eq:global-graph-tau} the result immediately follows as a larger $q$ increases the supremum.  
	
\bigskip 	
	
	\textsc{Part 3.} It is clear that $\mathcal B \mapsto d_{\mathcal B,m}(\cdot,\cdot)$ is non-decreasing, which renders $\mathcal B \mapsto \mathrm{Lip}_{1}( \mathsf{Z}^{m}, d_{\mathcal B,m} )$ as non-decreasing. The supremum in the definition of $\tau$ in \eqref{eq:tau-generic} is taken over this class, so the $\tau$-coefficient in that expression in non-decreasing. Thus, $\mathcal B \mapsto \boldsymbol{\tau}_{\mathcal B}$ is also non-decreasing. 

\bigskip 

    \textsc{Part 4.} We show that $\operatorname{Lip}_{1}(\mathsf{Z}^{m} , d_{\mathcal B - \mathcal B' , m}) \subseteq \operatorname{Lip}_{2}(\mathsf{Z}^{m} , d_{\mathcal F , m})$. Take any $g \in \operatorname{Lip}_{1}(\mathsf{Z}^{m} , d_{\mathcal B - \mathcal B' , m})$, then for any $z,z' \in \mathsf{Z}^{m}$, $|g(z) - g(z') | \leq \sup_{ f \in \mathcal B - \mathcal B' } \sum_{s=1}^{m}  |f(z_{s}) - f(z'_{s})| \leq \{ \sup_{ f \in \mathcal B }  \sum_{s=1}^{m} |f(z_{s}) - f(z'_{s})| + \sup_{ f' \in \mathcal B' }  \sum_{s=1}^{m}  |f'(z_{s}) - f'(z'_{s})| \}   $. Since $\mathcal B, \mathcal B' \subseteq \mathcal F$, the previous inequality implies $|g(z) - g(z') | \leq 2  \sup_{ f \in \mathcal F }  \sum_{s=1}^{m} |f(z_{s}) - f(z'_{s})|  $ and the desired result holds. 

     By \eqref{eq:tau-generic}, $\tau_{d_{\mathcal B - \mathcal B' , m}} \leq 2 \tau{d_{\mathcal F , m}}$ and the desired result follows because all instances of the $\tau$-coefficient are linear on $\tau$.

\end{proof}

\section{Appendix for Maximal Inequality Results}
\label{app:maximal}

\begin{proof}[Proof of Lemma \ref{lem:VarBlock.upper}]
	For any $g \in \mathcal F - \mathcal F$,  by Lemma \ref{lem:VarBlock.upper.0}, 
	\begin{align*}
		\sqrt{	\operatorname{Var}\left(  |B|^{-1/2} \sum_{i \in B} g(Z_{i})     \right)  } \leq ||g||_{B},
	\end{align*}
	where 
	\begin{align*}
		||g||_{B} : =  \sqrt{ \sum_{\rho \in R(B)}  	\frac{|S_{B}(\rho)|}{|B|}    \max_{i \in B} \int_{0}^{1}  1\left\{\alpha_{ \mathcal{F} , B }(\rho) \geq u	\right\} |Q_{|g(Z_{i})|}(u) |^{2} du }. 
	\end{align*}

    We now show that \begin{align}\label{eqn:bound.Bnorm}
	||g||_{B}   \leq \Gamma_{p}(B)  \max_{i \in B}   || g ||_{L^{2p}(P_{Z_{i}}) }.
\end{align}
	
	By, H\"{o}lder inequality, for any $p \in [1,\infty]$, 
	\begin{align*}
		\max_{i \in B}   \int_{0}^{1}  1\left\{\alpha_{ \mathcal{F} , B }(\rho) \geq u	\right\}  |Q_{|g(Z_{i})|}(u) |^{2} du \leq  & \left( \int_{0}^{1} 1\{  \alpha_{ \mathcal F,B}(\rho) \geq u  \} du \right)^{\frac{p-1}{p}}   \max_{i \in B}  \left(  \int_{0}^{1} |Q_{|g(Z_{i})|}(u)|^{2p}  \right)^{1/p}  \\
		= &  \left(  \alpha_{ \mathcal F,B}(\rho)  \right)^{\frac{p-1}{p}}   \max_{i \in B} ||g||^{2}_{L^{2p}(P_{Z_{i}})}. 
	\end{align*}
	Hence,
	\begin{align*}
		||g||_{B}  \leq   \sqrt{ \sum_{\rho \in R(B)}  	\frac{|S_{B}(\rho)|}{|B|}   \left(  \alpha_{ \mathcal F,B}(\rho)  \right)^{\frac{p-1}{p}}    }  \max_{i \in B} ||g||_{L^{2p}(P_{Z_{i}})}. 
	\end{align*}
	
	On the other hand, 
	\begin{align*}
		||g||_{B} \leq   \sqrt{   \int_{0}^{1}  \sum_{\rho \in R(B)}  	\frac{|S_{B}(\rho)|}{|B|}    1\left\{\alpha_{ \mathcal{F} , B }(\rho) \geq u	\right\}  \max_{i \in B} |Q_{|g(Z_{i})|}(u) |^{2} du } = : A_B(g).
	\end{align*}
	
	Define $\mu_B(u)	:=
	\sum_{\rho\in R(B)}
	\frac{|S_B(\rho)|}{|B|}
	1\{\alpha_{\mathcal F,B}(\rho)\geq u\}$.  Then, by H\"{o}lder's inequality,
	\begin{align*}
		A_B(g)^2
		& \leq 
		\left(
		\int_0^1 \mu_B(u)^{p'}\,du
		\right)^{1/p'}
		\left(
		\int_0^1
		\max_{i\in B} Q_{|g(Z_i)|}^{2p}(u)
		du
		\right)^{1/p},
	\end{align*}
	with $1/p' + 1/p = 1$.  Since $\max_{i\in B} Q_{|g(Z_i)|}^{2p}(u)
	\leq
	\sum_{i\in B} Q_{|g(Z_i)|}^{2p}(u)$, 
	we have
	\begin{align*}
		\left(
		\int_0^1
		\max_{i\in B} Q_{|g(Z_i)|}^{2p}(u)
		du
		\right)^{1/(2p)}
		&\leq
		\left(
		\sum_{i\in B}
		\int_0^1
		Q_{|g(Z_i)|}^{2p}(u)
		du
		\right)^{1/(2p)}	&\leq
		|B|^{1/(2p)}
		\max_{i\in B}
		\|g(Z_i)\|_{L^{2p}}.
	\end{align*}
	Observe that if $||dP_{Z_{i}}/dP_0||_{\infty} \leq C$, then $\int_{0}^{1} \max_{i\in B} Q_{|g(Z_i)|}^{2p}(u) du
	\leq C \int_{0}^{1} Q_{|g(Z_i)|,P_0}^{2p}(u) du$ and no extra $|B|^{1/2p}$ is needed --- $Q_{\cdot,P_0}$ denotes the quantile associated to $P_0$.

	Therefore,  defining
	\begin{align*}
		\Gamma_p^{\sharp}(B)
		:=
		\left(
		\int_0^1
		\left[
		\sum_{\rho\in R(B)}
		\frac{|S_B(\rho)|}{|B|}
		1\{\alpha_{\mathcal F,B}(\rho)\geq u\}
		\right]^{p'}
		du
		\right)^{1/(2p')},
	\end{align*}
	and $\Gamma_1^{\sharp}(B)
	\leq
	\left(
	\sum_{\rho\in R(B)}
	\frac{|S_B(\rho)|}{|B|}
	\right)^{1/2}$ for $p = 1$ 
	we obtain $A_B(g)
		\leq
		\Gamma_p^{\sharp}(B)
		|B|^{1/(2p)}
		\max_{i\in B}
		\|g(Z_i)\|_{L^{2p}}$.

\end{proof}

\subsection{Proofs of Supplemental Lemmas}

\begin{proof}[Proof of Lemma \ref{lem:VarBlock.upper.0}]	
	Observe that 
	\begin{align*}
		\operatorname{Var} \left( |B|^{-1/2} \sum_{i \in B} g(Z_{i})   \right) = |B|^{-1} \left(  \sum_{i \in B} \operatorname{Var}(g(Z_{i})) + \sum_{\substack{i,j\in B\\ i \ne j}} \operatorname{Cov}(g(Z_i),g(Z_j)) \right).
	\end{align*}
	By Theorem 1.1 in \cite{rio2017asymptotic}, $|\operatorname{Cov}(g(Z_i),g(Z_j))| \leq \int^{\alpha(g(Z_i),g(Z_j))}_{0} Q_{|g(Z_i)|}(u) Q_{|g(Z_j)|}(u) du $ where $\alpha(X,Y) : = 2 \sup_{s,t \in \mathbb{R}} | \mathbb P(Z_{i} \geq t ~and~Z_{j} \geq s) - P_{Z_{i}}(Z_{i} \geq t) P_{Z_{j}}(Z_{j} \geq s)   |$ for any real-valued random variable $X$ and $Y$. Moreover, if $d_{n}(i,j ) \geq r $ for some $r \geq 1$, 
	\begin{align*} 
		\alpha(g(Z_i),g(Z_j)) \leq \sup_{i , j \in B \colon d_{n}(i,j ) \geq r }   \alpha(g(Z_i),g(Z_j))  	\leq &  \sup_{g \in  \mathcal B} \sup_{i , j \in B \colon d_{n}(i,j ) \geq r }   \alpha(g(Z_i),g(Z_j)) 	= :  \alpha_{\mathcal B,B} (r),
	\end{align*}

    Hence,
	\begin{align*}
		\operatorname{Var}\left( |B|^{-1/2}  \sum_{i \in B} g(Z_{i})  \right) \leq & |B|^{-1}  \sum_{i \in B} \left( \operatorname{Var}(g(Z_{i})) +  \sum_{j\in B} \int^{1}_{0} 1\{ \alpha(g(Z_{i}),g(Z_{j})) \geq u \} Q_{|g(Z_{i})|}(u)  Q_{|g(Z_{j})|}(u) du \right) \\
		\leq &  |B|^{-1} \left( \sum_{i,j\in B} \int^{1}_{0} 1\{ \alpha_{\mathcal{B} , B }(d_{n}(i,j)) \geq u \} Q_{|g(Z_{i})|}(u)  Q_{|g(Z_{j})|}(u) du \right) \\
		= &  |B|^{-1}	\sum_{r \in R(B)} \sum_{(i,j) \in S_{B}(r)} 	\int_{0}^{1} 1\left\{\alpha_{\mathcal{B} , B }(r) \geq u	\right\}		Q_{|g(Z_{i})|}(u)Q_{|g(Z_{j})|}(u)du.
	\end{align*}
	where $S_{B}(r) := \{ (i,j) \in B \times B : d_{n}(i,j)=r \}$ and $R(B) = \{ r \in \{ 0,1, \ldots, n-1\} \cup \{\infty\}    \colon 	S_{B}(r) \ne \emptyset  \}$. We set $d_{n}(i,i) = 0$ and $\alpha_{\mathcal{B} , B }(0) = 1$. The case $r=\infty$ captures pairs that are not connected, which by our assumption over the graph-dependent process are independent and thus $\alpha_{\mathcal{B} , B }(\infty) : =  0$.  
	
	For any $r \in R(B)$, let
	\begin{align*}
		\bar{N}_{B}(r)  : = |B|^{-1} |S_{B}(r)|. 
	\end{align*}
	That is, $	\bar{N}_{B}(r) $ is the average number of nodes at distance $r$ from each other within $B$. Then,
	\begin{align*}
		\operatorname{Var}\left( |B|^{-1/2}  \sum_{i \in B} g(Z_{i})  \right)  \leq & 	\sum_{r \in R(B)}  	\bar{N}_{B}(r)  	\int_{0}^{1}  1\left\{\alpha_{\mathcal{B} , B }(r) \geq u	\right\} \sum_{(i,j) \in S_{B}(r)}   \frac{  Q_{|g(Z_{i})|}(u)Q_{|g(Z_{j})|}(u) 	}{	|S_{B}(r) |  }	du \\
		= &  \sum_{r \in R(B)}  	\bar{N}_{B}(r)   \sum_{(i,j) \in S_{B}(r)}   \frac{  	\int_{0}^{1}  1\left\{\alpha_{ \mathcal{B} , B }(r) \geq u	\right\} Q_{|g(Z_{i})|}(u)Q_{|g(Z_{j})|}(u) 	du}{	|S_{B}(r) |  }. 
	\end{align*}
	
	For any $i,j,r$, by Cauchy-Swarchz inequality,
	\begin{align*}
		\int_{0}^{1}  1\left\{\alpha_{\Delta_l \mathcal{F} , B }(r) \geq u	\right\} Q_{|g(Z_{i})|}(u)Q_{|g(Z_{j})|}(u) du \leq & \sqrt{ \int_{0}^{1}  1\left\{\alpha_{ \mathcal{B} , B }(r) \geq u	\right\} |Q_{|g(Z_{i})|}(u)|^{2} du  } \\
		&  \times \sqrt{ \int_{0}^{1}  1\left\{\alpha_{\mathcal{B} , B }(r) \geq u	\right\} |Q_{|g(Z_{j})|}(u) |^{2} du }\\
		\leq & \max_{i \in B} \int_{0}^{1}  1\left\{\alpha_{\mathcal{B} , B }(r) \geq u	\right\} |Q_{|g(Z_{i})|}(u) |^{2} du.
	\end{align*}
	
	Therefore, 
	\begin{align*}
		\operatorname{Var}\left( |B|^{-1/2}  \sum_{i \in B} g(Z_{i})  \right)   \leq & 	\sum_{r \in R(B)}  	\bar{N}_{B}(r)    \max_{i \in B} \int_{0}^{1}  1\left\{\alpha_{\mathcal{B} , B }(r) \geq u	\right\} |Q_{|g(Z_{i})|}(u) |^{2} du.
	\end{align*}
	
	The desired expression follows because $\mathcal B \subseteq \mathcal F - \mathcal F$ and $\mathcal B \mapsto  \alpha_{ \mathcal B , B }$ is non-decreasing. 
\end{proof}

\begin{proof}[Proof of Lemma \ref{lem:bern}]
	We can cast $	\mathbb G_{n}^{(k),\ast}$ as
	\begin{align}
		\mathbb G_{n}^{(k),\ast}(g)
		&= \frac{\sqrt{|B_{m}|}}{\sqrt n}\sum_{m\in\mathcal C_k} \frac{1}{\sqrt{|B_{m}|}} \sum_{i \in B_{m}} \{ g(Z^{\ast}_{i}) - \mathbb{E}[g(Z_{i})]  \}
	\end{align}
	Set, for all $m \in C_{k}$,
	\begin{align}
		\delta^{\ast}_m g := \frac{1}{\sqrt {|B_{m}|} } \sum_{i \in B_{m}} \{ g(Z^{\ast}_{i}) - \mathbb{E}[g(Z_{i})]  \}.
	\end{align}
	By Theorem \ref{thm:coupling}, the variables \((Z^{\ast}_{B_{m}})_{m\in\mathcal C_k}\) are independent, so, \((	\delta^{\ast}_m g)_{m\in\mathcal C_k}\) are too. Moreover, $\mathbb E^{\ast}[	\delta^{\ast}_m g] = 0$, 
	because, by Theorem \ref{thm:coupling} \(Z_i^\ast\stackrel d= Z_{i} \) for each \(i \in B_{m}\). Also, $||\delta^{\ast}_m g||_{\infty} \leq 2 \sqrt{|B_{m}|} ||g||_{\infty} \leq 2 \sqrt{q} ||g||_{\infty}$ 
	where the last inequality follows because $|B_{m}|\leq q$.

	Therefore the Bernstein inequality for independent centered random variables (e.g., see \cite{VdV-W1996}) yields, for every \(t>0\),
	\begin{align*}
		\mathbb P^{\ast} \!\left(
		\left| n^{-1/2}  \sum_{m\in\mathcal C_k} \sqrt{|B_{m}|} \delta^{\ast}_m g  \right|\ge t
		\right)
		\le &
		2\exp\left(
		-\frac{n t^2}{
			2\sum_{m\in\mathcal C_k} |B_{m}|  \operatorname{Var}( \delta^{\ast}_m g )
			+\frac{2}{3}  \sqrt{n} q ||g||_{\infty}   t   
		}
		\right)\\
		= & 
		2\exp\left(
		-\frac{t^2}{
			2\sum_{m\in\mathcal C_k} \frac{|B_{m}|}{n}  \operatorname{Var}( \delta_m g )
			+\frac{2}{3}  \frac{q}{\sqrt{n}} ||g||_{\infty}   t   
		}
		\right),
	\end{align*}
	where $ \operatorname{Var}( \delta_m g ) =  \operatorname{Var}( \delta^{\ast}_m g )$ follows because, from Theorem \ref{thm:coupling}, $Z^{\ast}_{B_{m}} \stackrel{d}{=} Z_{B_{m}}$.
	
	By Lemma \ref{lem:VarBlock.upper.0}, $\operatorname{Var}( \delta_m g )  \leq ||g||^{2}_{B_{m}}$, so 
	\begin{align*}
		\mathbb P^{\ast} \!\left(
		\left| n^{-1/2}  \sum_{m\in\mathcal C_k} \sqrt{|B_{m}|} \delta^{\ast}_m g  \right|\ge t
		\right)
		\le  2\exp\left(
		-\frac{t^2}{
			4 \sum_{m\in\mathcal C_k} \frac{|B_{m}|}{n}  ||g||^{2}_{B_{m}}
			+\frac{2}{3}  \frac{q}{\sqrt{n}} ||g||_{\infty}   t   
		}
		\right),
	\end{align*}
	an thus the desired result follows. 
\end{proof}

\begin{proof}[Proof of Lemma~\ref{lem:bound.pr.talagrand}]
	Fix \(l\geq 1\). For notational simplicity, write $q_l:=\max_{m\in[M_l]}|B_m|$ and, for each color class \(k\in[K_l]\) and each measurable function \(g\), define
	\begin{align*}
		\sigma_{k}^{2}(g)
		&:=
		\sum_{m\in\mathcal C_k}
		\frac{|B_m|}{n}\|g\|_{B_m}^{2}~and~
		b_l(g)
		:=
		\frac{q_l}{\sqrt n}\|g\|_{\infty}.
	\end{align*}
	By Lemma \ref{lem:bern} (observe that $g \in \Delta_{l} \mathcal F \subseteq \mathcal F - \mathcal F$), for every \(x>0\),
	\begin{align}
		\mathbb{P}^{\ast} \left(
		|\mathbb G^{(k),*}_{n}(g)|
		>
		2\sigma_{k}(g)\sqrt{x}
		+
		\frac{2}{3}b_l(g)x
		\right)
		\leq
		2e^{-x}.
		\label{eq:bernstein-color-class}
	\end{align}
	
	Define $A_{l,k}(g)
		:=
		2\sigma_{k}(g)\ell(l)
		+
	\frac{2}{3}b_l(g)\ell(l)^2$ and $A_l(g):=\sum_{k=1}^{K_l}A_{l,k}(g)$. 
	Taking \(x=t\ell(l)^2\) in \eqref{eq:bernstein-color-class}, and using
	\(\sqrt t\leq t\) for \(t\geq1\), gives
	\begin{align}
		\mathbb P^{\ast} \left(
		|\mathbb G^{(k),*}_{n}(g)|
		>
		tA_{l,k}(g)
		\right)
		\leq
		2e^{-t\ell(l)^2},
		\qquad t\geq1.
		\label{eq:normalized-tail}
	\end{align}
	
   Let $Y_{l,k}(g):=\left( \frac{|\mathbb G^{(k),*}_{n}(g)|}{A_{l,k}(g)}  -1\right)_+$. 	For every \(s\geq0\), \eqref{eq:normalized-tail} with $t=1+s$ implies $\mathbb P^{\ast} \left(Y_{l,k}(g)>s\right) \leq 
	2e^{-(1+s)\ell(l)^2}
	\leq
	2e^{-s\ell(l)^2}$. 
	Hence, for any \(e\in(0,1)\),
	\begin{align}
		\mathbb E^{\ast} \left[
		\exp\{e\ell(l)^2Y_{l,k}(g)\}
		\right]
		&\leq 1 + \int_{1}^{\infty} \mathbb{P}^{\ast} \left( 
		\exp\{e\ell(l)^2Y_{l,k}(g)\} \geq u  \right) du \nonumber \\
        & = 
		1+
		e\ell(l)^2
		\int_0^\infty
		e^{e\ell(l)^2s}
		\mathbb P^{\ast} (Y_{l,k}(g)>s)\,ds \nonumber \\
        & \leq
		1+
		2e\ell(l)^2
		\int_0^\infty
		e^{-(1-e)\ell(l)^2s}\,ds
		\nonumber\\
		&=
		1+\frac{2e}{1-e}
		=:C_e,
		\label{eq:shifted-exp-moment}
	\end{align}
    where the second line follows from a change of variables ($u = e^{e \ell(l) s }$).
	
	Let  $w_{l,k}(g):=\frac{A_{l,k}(g)}{A_l(g)}$ for all $ k=1,\ldots,K_l$. Then \(\sum_{k=1}^{K_l}w_{l,k}(g)=1\), and
	\[
	\frac{
		\sum_{k=1}^{K_l}|\mathbb G^{(k),*}_{n}(g)|
	}{
		A_l(g)
	}
	=
	\sum_{k=1}^{K_l}w_{l,k}(g) \frac{|\mathbb G^{(k),*}_{n}(g)|}{A_{l,k}(g)} 
	\leq
	1+
	\sum_{k=1}^{K_l}w_{l,k}(g)Y_{l,k}(g).
	\]
	Therefore, if the left-hand side is larger than \(v\), then $\sum_{k=1}^{K_l}w_{l,k}(g)Y_{l,k}(g)>v-1$.  By Markov's inequality, convexity of the exponential, and
	\eqref{eq:shifted-exp-moment},
	\begin{align}
		 \mathbb P^{\ast} \left(
		\sum_{k=1}^{K_l}
		|\mathbb G^{(k),*}_{n}(g)|
		>
		vA_l(g)
		\right)
		\nonumber & \leq
		\mathbb P^{\ast} \left(
		\sum_{k=1}^{K_l}
		w_{l,k}(g)Y_{l,k}(g)>v-1
		\right)
		\nonumber\\
		& \leq
		e^{-e\ell(l)^2(v-1)}
		\mathbb E^{\ast} \left[
		\exp\left\{
		e\ell(l)^2
		\sum_{k=1}^{K_l}w_{l,k}(g)Y_{l,k}(g)
		\right\}
		\right]
		\nonumber\\
		& \leq
		e^{-e\ell(l)^2(v-1)}
		\sum_{k=1}^{K_l}
		w_{l,k}(g)
		\mathbb E^{\ast} \left[
		\exp\{e\ell(l)^2Y_{l,k}(g)\}
		\right]
		\nonumber\\
		& \leq
		C_e e^{-e\ell(l)^2(v-1)}.
		\label{eq:sum-color-tail}
	\end{align}
	If \(A_l(g)=0\), then \(A_{l,k}(g)=0\) for every \(k\), and the same conclusion holds trivially.
	
	Now take a union bound over the possible increments at scale \(l\),
    	\begin{align}
		\mathbb P^{\ast} \left(
		\exists g\in\Delta_l\mathcal F:
		\sum_{k=1}^{K_l}
		|\mathbb G^{(k),*}_{n}(g)|
		>
		vA_l(g)
		\right) \leq \sum_{g \in \Delta_l \mathcal F} \mathbb P^{\ast} \left(  \sum_{k=1}^{K_l}
		|\mathbb G^{(k),*}_{n}(g)|
		>
		vA_l(g)  \right)
	\end{align}
    
    The number of possible increments \(g=\Delta_l f\) is bounded by $|T_l(c,d)|\,|T_{l-1}(c,d)|
	\leq
	|T_l(c,d)|^2
	\leq
	2^{2d2^{l/c}}$ where the last inequality follows because  $|T_l(c,d)|\leq 2^{d2^{l/c}}$. Therefore,
	\begin{align}
		\mathbb P^{\ast} \left(
		\exists g\in\Delta_l\mathcal F:
		\sum_{k=1}^{K_l}
		|\mathbb G^{(k),*}_{n}(g)|
		>
		vA_l(g)
		\right) \leq
		C_0\,2^{2d2^{l/c}}e^{-e\ell(l)^2(v-1)}.
		\label{eq:union-increments-scale-l}
	\end{align}
	
	Define the event $E:=
	\bigcap_{l\geq1}
	\bigcap_{g\in\Delta_l\mathcal F}
	\left\{
	\sum_{k=1}^{K_l}
	|\mathbb G^{(k),*}_{n}(g)|
	\leq
	vA_l(g)
	\right\}$. 
	By \eqref{eq:union-increments-scale-l},
	\begin{align}
		\mathbb P^{\ast} (E^c)
		\leq
		C_e
		\sum_{l=1}^{\infty}
		2^{2d2^{l/c}}e^{-e\ell(l)^2(v-1)} \iff \mathbb P^{\ast} (E)
		\geq 1- 
		C_e
		\sum_{l=1}^{\infty}
		2^{2d2^{l/c}}e^{-e\ell(l)^2(v-1)}.
		\label{eq:event-complement-bound}
	\end{align}
	
	On the event \(E\), for every \(f\in\mathcal F\),
	\begin{align*}
		\sum_{l=1}^{\infty}
		\sum_{k=1}^{K_l}
		|\mathbb G^{(k),*}_{n}(\Delta_l f)|
		&\leq
		v
		\sum_{l=1}^{\infty}
		A_l(\Delta_l f)
		=
		v
		\sum_{l=1}^{\infty}
		\sum_{k=1}^{K_l}
		\left[
		2
		\sqrt{
			\sum_{m\in\mathcal C_k}
			\frac{|B_m|}{n}
			\|\Delta_l f\|_{B_m}^{2}
		}
		\ell(l)
		+
		\frac{2}{3}
		\frac{q_l}{\sqrt n}
		\|\Delta_l f\|_{\infty}
		\ell(l)^2
		\right].
	\end{align*}
	Taking the supremum over \(f\in\mathcal F\), we obtain
	\begin{align*}
		\sup_{f\in\mathcal F}
		\sum_{l=1}^{\infty}
		\sum_{k=1}^{K_l}
		|\mathbb G^{(k),*}_{n}(\Delta_l f)|
	 \leq
		v
		\sup_{f\in\mathcal F}
		\sum_{l=1}^{\infty}
		\sum_{k=1}^{K_l}
		\left[
		2
		\sqrt{
			\sum_{m\in\mathcal C_k}
			\frac{|B_m|}{n}
			\|\Delta_l f\|_{B_m}^{2}
		}
		\ell(l)
		+
		\frac{2}{3}
		\frac{\max_{m\in[M_l]}|B_m|}{\sqrt n}
		\|\Delta_l f\|_{\infty}
		\ell(l)^2
		\right].
	\end{align*}
	Combining this display with \eqref{eq:event-complement-bound} proves that
	\begin{align*}
		&\mathbb P^{\ast} \left(
		\sup_{f\in\mathcal F}
		\sum_{l=1}^{\infty}
		\sum_{k=1}^{K_l}
		|\mathbb G^{(k),*}_{n}(\Delta_l f)| \leq
		v
		\sup_{f\in\mathcal F}
		\sum_{l=1}^{\infty}
		\sum_{k=1}^{K_l}
		\left[
		2
		\sqrt{
			\sum_{m\in\mathcal C_k}
			\frac{|B_m|}{n}
			\|\Delta_l f\|_{B_m}^{2}
		}
		\ell(l)
		+
		\frac{2}{3}
		\frac{\max_{m\in[M_l]}|B_m|}{\sqrt n}
		\|\Delta_l f\|_{\infty}
		\ell(l)^2
		\right]
		\right)
		\\
		&\qquad\geq
		1-
		C_e
		\sum_{l=1}^{\infty}
		2^{2d2^{l/c}}e^{-e\ell(l)^2(v-1)}.
	\end{align*}
	In particular, if \(\ell(l)\geq1\) for all \(l\) and \(v\geq4\), then, $e \ell(l)^{2}(v-1) = e \ell(l)^{2}(0.5 v + 0.5 v-1)  \geq 2 e \ell(l)^{2} + e (0.5 v - 1) $. Hence the previous display implies
	\begin{align*}
		& \mathbb P ^{\ast} \left(
		\sup_{f\in\mathcal F}
		\sum_{l=1}^{\infty}
		\sum_{k=1}^{K_l}
		|\mathbb G^{(k),*}_{n}(\Delta_l f)|
	 \leq
		v
		\sup_{f\in\mathcal F}
		\sum_{l=1}^{\infty}
		\sum_{k=1}^{K_l}
		\left[
		2
		\sqrt{
			\sum_{m\in\mathcal C_k}
			\frac{|B_m|}{n}
			\|\Delta_l f\|_{B_m}^{2}
		}
		\ell(l)
		+
		\frac{2}{3}
		\frac{\max_{m\in[M_l]}|B_m|}{\sqrt n}
		\|\Delta_l f\|_{\infty}
		\ell(l)^2
		\right]
		\right)
		\\
		&\qquad\geq
		1-
		C
		\left(
		\sum_{l=1}^{\infty}
		2^{2d2^{l/c}}e^{-2e\ell(l)^2}
		\right)e^{-e0.5v},
	\end{align*}
	for some universal constant $C$. 
	
	Set $ e  = 0.5$. We conclude by observing that since $\ell(l) \geq 2 d 2^{l/c}$, $2d 2^{l/c} -  \ell(l)^{2} \leq 2d 2^{l/c} - 4 d^{2} 2^{2l/c}$ which is negative as $d \geq 1$. Hence, $	\sum_{l=1}^{\infty}
	2^{2d2^{l/c}}e^{-2e\ell(l)^2}$ is finite.

\end{proof}

\section{Appendix for Section \ref{sec:geometry}}
\label{app:geometry} 

\subsection{Uniform doubling with polynomial growth graphs}

The following lemma is used to show the results in this section.

\begin{lemma}\label{lem:PolGraph.growth}
	Suppose $(\mathsf X,\mathsf d)$ is uniformly doubling with dimension $d$.
	Then
	for all $x \in \mathsf X$ and all $0<r\le R$,
	\begin{align}
		\sup \Big\{
		|A| : A \subseteq B(x,R), \;
		\mathsf d (a,b) \ge r \;\; \forall a \neq b \in A
		\Big\}
		\;\le\;
		2^{2d} \left(1+\frac{R}{r}\right)^d.
	\end{align}
\end{lemma}

\begin{proof}[Proof of Lemma \ref{lem:PolGraph.growth}]
	Let $N_{\mathrm{cov}}(B(x,R),\varepsilon)$ denote the covering number.
	By the doubling property, $N_{\mathrm{cov}}(B(x,R),R/2) \le C_{\mathrm{dbl}}$, 
	and by iteration, for any $k\in\mathbb{N}$, $N_{\mathrm{cov}}(B(x,R),R/2^k) \le C_{\mathrm{dbl}}^k = 2^{k d}$. 
	Choose $k$ such that $R/2^k \le r < R/2^{k-1}$, which implies
	$k \le \log_2(R/r)+1$. Then
	\begin{align*}
		N_{\mathrm{cov}}(B(x,R),r) \le 	N_{\mathrm{cov}}(B(x,R),R/2^{k}) \le 
		2^{k d}  \leq 	2^{d} \left(\frac{R}{r}\right)^{d}.
	\end{align*}
	
	Now let $A \subseteq B(x,R)$ be $r$-separated. So, no ball of radius $r/2$ can contain two distinct points of $A$. Thus,  $	|A| \le	N_{\mathrm{cov}}(B(x,R+r/2),r/2)$. By applying the previous display with $r/2$,
	\begin{align}
		|A| \leq 	2^{d}\left(\frac{R+r/2}{r/2}\right)^d = 	2^{2d}\left(1 + \frac{R}{r} \right)^d,
	\end{align}
	which proves the claim.
\end{proof}

 \begin{proof}[Proof of Lemma \ref{lem:PolyGraph.color}]
	Fix \(q\ge 1\) and set $\rho := q^{1/d}$. 
	Let \(\{x_1,\ldots,x_M\}\subset N_n\) be a maximal \(\rho\)-separated set, that is, $d_n(x_m,x_\ell)>\rho,~\forall m\neq \ell$, 
	and no further point of \(N_n\) can be added while preserving this property.
	Maximality implies the covering property: for all $i\in N_n$,  $\min_{1\le m\le M} d_n(i,x_m)\le \rho$. 
	
	Define the Voronoi partition \(\mathcal P_{M,q} =\{B_1,\ldots,B_M\}\) by $B_m
	:=
	\Big\{
	i\in N_n:
	d_n(i,x_m)\le d_n(i,x_\ell)
	\ \text{for all } \ell
	\Big\}$, 
	with ties broken arbitrarily. By construction, $B_m \subseteq B(x_m,\rho)$. 
	Therefore, using the polynomial upper-growth condition,
	\begin{align*}
		|B_m|
		\le
		|B(x_m,\rho)|
		\le \sum_{l=0}^{\rho} | \partial B(x_m,l)|  \le
		C
		\sum_{\ell=0}^{\rho}
		(1+\ell)^{d-1}  \lesssim
		(1+\rho)^{d-1+1} \lesssim  q .
	\end{align*}
	Hence, this construction satisfies \eqref{eqn:Double.Poly.1}.  
	
	Moreover, for any $(i,j) \in B_{m}$, by the triangle inequality $d_{n}(i,j) \leq d_{n}(i,x_{m}) +  d_{n}(j,x_{m}) \leq 2 \rho$. 
	So, $\operatorname{diam}(B_{m}) \lesssim \rho$ holds.

	We now show \eqref{eqn:Double.Poly.chrom}. To do this we use the well known greedy bound \eqref{eqn:color.greedy.bound} in Appendix \ref{app:graph}. For this, take the $r$-conflict graph (defined in that Appendix), $H_{r} : = ([M],E_{r})$ and endow it with the following metric $(m,m')\mapsto d_{H}(m,m') : = d_{n}(x_{m},x_{m'})$. The tuple $(H_{r},d_{H})$ is uniformly doubling with dimension at most $d$ because $(H_{r},d_{H})$ is isomorphic to a subset of $(\mathfrak G_n,d_n)$, which is uniformly doubling with dimension $d$. 

    Moreover, for any $m \in [M]$ the set of neighbors according to $E_{r}$ is given by $\mathcal N _{H_{r}}(m) : = \{ \ell \in [M] \setminus \{m\}  \colon d_{n}(B_{m},B_{\ell}) \leq r \} $ and is a set of $\rho$-separated points under $d_{H}$ --- because each $\{x_{1},...,x_{M}\}$ are $\rho$-separated under $d_{n}$. Finally, for any $m' \in \mathcal N _{H_{r}}(m)$, 
    \begin{align*}
        d_{H}(m,m') = d_{n}(x_{m},x_{m'}) \leq d_{n}(x_{m},i) + d_{n}(i,j) + d_{n}(x_{m'},j) \leq 2 \rho + r.
    \end{align*}
    Here $(i,j)$ is the pair such that $i \in B_{m}$ and $j \in B_{m'}$ and $d_n(i,j)\le r$ --- such pair exists because $d_{n}(B_{m},B_{\ell}) \leq r$. Hence $\mathcal N _{H_{r}}(m) \subseteq B_{d_{H}}(m,r+2\rho)$. 

    Therefore, we can apply Lemma \ref{lem:PolGraph.growth} with $\mathsf d = d_{H}$, $\mathsf X = [M]$, $A = \mathcal N _{H_{r}}(m)$, $r= \rho$ and $R=2\rho + r$, and obtain 
    \begin{align*}
		\deg_{H_{r}}(m) = | \mathcal N _{H_{r}}(m) | \leq 2^{3d} \left(1+  \frac{r}{\rho} \right)^{d}.
	\end{align*}
    Since the bound is uniform on $x_{m}$, $\Delta (H_{r}) \lesssim \left( 1 + \frac{r}{\rho}  \right) ^{d}$. 

Thus, \eqref{eqn:Double.Poly.chrom} follows because  every finite graph is colorable with at most one plus its maximum degree, i.e., 
	\begin{align*}
		K_{r} (\mathcal P_M) 
		\le
		1+ \Delta (H_{r})  \leq 
		1+
		2^{3d} \left(1+\frac{r}{\rho}\right)^d .
	\end{align*}
	Using \(\rho=q^{1/d}\) and 	$\left(1+\frac{r}{\rho}\right)^d
	\lesssim
	1+\frac{r^d}{\rho^d}
	=
	1+\frac{r^d}{q}$, the desired result follows. Observe that the constants implicit in $\lesssim$ depend only on the doubling and polynomial upper-growth
	constants, and not on \(n,r,q\).

	We now show \eqref{eqn:Double.Poly.shells}. By definition, $|S_{B_m}(r)|
	=	\sum_{i\in B_m} 	\big|\{j\in B_m:d_n(i,j)=r\}\big|$.  For each \(i\in B_m\), $\big|\{j\in B_m:d_n(i,j)=r\}\big|
	=
	| \partial B(i,r)|
	\le
	Cr^{d-1}$, 	where the last inequality follows from polynomial upper growth. Hence $|S_{B_m}(r)|
	\le
	C|B_m|r^{d-1}$. 
	On the other hand, trivially, $|S_{B_m}(r)|
	\le
	|B_m|^2$. Combining these two bounds gives
	\begin{align*}
		|S_{B_m}(r)|
		\le
		\min\{C|B_m|r^{d-1},|B_m|^2\}
		\lesssim
		|B_m|\min\{r^{d-1},|B_m|\}.
	\end{align*}
	
	Since \(|B_m|\lesssim q\), it follows that $|S_{B_m}(r)|
	\lesssim q\min\{r^{d-1},q\}$. Thus \eqref{eqn:Double.Poly.shells} follows. 
    
\end{proof}

\begin{proof}[Proof of Proposition \ref{pro:maximal.DoublePoly}]
	We take the inequality in Theorem \ref{thm:maximal} and specialize it to the setting of Lemma \ref{lem:PolyGraph.color} with $r_{n} = q^{1/d}_{n}$.

	Recall that $R(B_m)
	:=
	\{r\ge 0:S_{B_m}(r)\neq \varnothing\}$, 
	where $S_{B_m}(r)
	:=
	\{(i,j)\in B_m^2:d_n(i,j)=r\}$. So, if $r \in R(B_{m})$,  there exist \(i,j\in B_m\) such that 	$d_n(i,j)=r$ so that $r	\le
	\operatorname{diam}(B_m)$. Hence $R(B_m)
	\subseteq
	\{0,1,\ldots,\operatorname{diam}(B_m)\}$.  This and the fact that $	\operatorname{diam}(B_m)\le 2 q^{1/d} $ (by Lemma \ref{lem:PolyGraph.color}) imply that $|R(B_m)| \leq \lfloor 2 q^{1/d} \rfloor + 1$.

	This result, the definition of $\Gamma_{p}$, and the fact that $	\frac{|S_{B_m}(r)|}{|B_m|}
	\lesssim 
	\min\{r^{d-1}_{n} , q_{n}  \} =r^{d-1}_{n} = q^{1-1/d}_{n}$ imply that
	\begin{align*}
		\Gamma^{\mathrm{sum}}_{p}(B_m) \lesssim &  q^{0.5(1-1/d)}_{n}  \left\{  
		\begin{array}{ll}
			\sqrt{ \sum_{\rho \in R(B_m)}   \left(  \alpha_{ \mathcal{F} , B _m }(\rho)   \right)^{1-\frac{1}{p} } }  & if~p \in (1,\infty] \\
			\sqrt{ 	q^{1/d}_{n}  } & if~p = 1
		\end{array}
		\right. \\
		\lesssim &
		\left\{  
		\begin{array}{ll}
			\sqrt{  q^{(1-1/d)}_{n}   \sum_{\rho=0}^{	\lfloor 2 q^{1/d}_{n} \rfloor+1}   \left(  \alpha_{ \mathcal{F} , B_m }(\rho)   \right)^{1-\frac{1}{p} } }  & if~p \in (1,\infty] \\
			\sqrt{ 	q_{n}  } & if~p = 1
		\end{array}
		\right.
	\end{align*}
And,
    \begin{align*}
	\Gamma_p^{\sharp}(B_m)
	\leq 
	\left\{  
		\begin{array}{ll}
	\sqrt{ q_{n}^{(1-1/d)} \left(
	\int_0^1
	\left[
	\sum_{\rho = 0}^{ \lfloor 2 q^{1/d}_{n} \rfloor+1 }
	1\{\alpha_{\mathcal F,B_m}(\rho)\geq u\}
	\right]^{\frac{p}{p-1}}
	du
	\right)^{1-\frac{1}{p} } } & if~p \in (1,\infty] \\
	\sqrt{ q_{n} }  & if~p = 1	
	 	\end{array}
 		 \right.
\end{align*}

Thus, for $p > 1$
\begin{align*}
    \Gamma_{p}(B_{m})^{2} \leq \min \left\{ q_{n}^{(1-\frac{1}{d}) + \frac{1+p}{p}} \left(
	\int_0^1
	\left[
	\sum_{\rho = 0}^{ \lfloor 2 q^{1/d}_{n} \rfloor+1 }
	1\{\alpha_{\mathcal F,B_m}(\rho)\geq u\}
	\right]^{\frac{p}{p-1}}
	du
	\right)^{1-\frac{1}{p} } ,   q^{1-\frac{1}{d}}_{n}   \sum_{\rho=0}^{	\lfloor 2 q^{1/d}_{n} \rfloor+1}   \left(  \alpha_{ \mathcal{F} , B_m }(\rho)   \right)^{1-\frac{1}{p} } \right\} 
\end{align*}
	

	Finally, by Lemma \ref{lem:PolyGraph.color}, choosing $r^{d}_{n} = q_{n}$ implies $K_{r_{n}} = O(1)$. 
\end{proof}

\subsection{Exponential growth graphs} 

\begin{proof}[Proof of Lemma \ref{lem:coloring.exponential}]
Choose \(\rho=\rho(q)\) such that $C_B e^{c \rho}
	\lesssim q$. Equivalently, \(\rho\lesssim \log q\).
    
We now construct the partition. Let
	\(\{x_1,\ldots,x_M\}\subseteq N_n\) be a maximal \(\rho(q)\)-separated set, meaning $d_n(x_m,x_\ell)>\rho$ for  $m\neq \ell$, and maximal with respect to inclusion. By maximality, the balls
	\(\{B(x_m,\rho(q))\}_{m=1}^M\) cover \(N_n\). Define a Voronoi-type partition by assigning each node \(i\in N_n\) to one
	nearest center \(x_m\), breaking ties arbitrarily. Denote the resulting cells by
	\(\mathcal P_{M} : = \{ B_1,\ldots,B_M\} \). Such construction implies that $B_m\subseteq B(x_m,\rho(q))$ and thus 
 \begin{align*}
		|B_{m}| \leq |B(x_{m},\rho)|
		=
		1+\sum_{s=1}^{\rho} |\partial B(i,s)|  \le
		1+C\sum_{s=1}^{\rho} e^{c s}
		\lesssim
		e^{c \rho} \lesssim q
	\end{align*}
    where the last inequality follows from our choice of $\rho = \rho(q)$.

    We now bound the $r$-chromatic number. For this we construct the $r$-conflict graph, $H_{r} : = H_{r}[\mathcal P_{M}] : =([M],E_{r})$, and apply the greedy bound. $m$ and $\ell$ are adjacent according to $E_{r}$ if the two blocks \(B_m\) and \(B_\ell\) are such that $d_n(B_m,B_\ell)\le r$. If this happens, then there exist \(i\in B_m\) and \(j\in B_\ell\) such that
	\(d_n(i,j)\le r\). Since \(B_m\subseteq B(x_m,\rho)\) and
	\(B_\ell\subseteq B(x_\ell,\rho)\), we obtain
	\begin{align*}
		d_n(x_m,x_\ell) \le d_n(x_m,i)+d_n(i,j)+d_n(j,x_\ell) \le \rho+r+\rho = r+2\rho.
	\end{align*}
	Thus every neighbor of \(B_m\) in the conflict graph has a center \(x_\ell\)
	inside \(B(x_m,r+2\rho)\). Therefore,
	\begin{align*}
		\deg_{H_{r}}(m) \le
		\left|B(x_m,r+2\rho)\right| \lesssim
		e^{c(r+2\rho)}  =	e^{cr} e^{2c\rho} \lesssim	q^{2} e^{cr}.
	\end{align*}

    Since this bound is uniform on $m$, by the greedy bound, this proves that $K_{r}(\mathcal P_{M}) \lesssim q^{2} e^{cr}$.

    We now bound the within-block $r$-distance. Recall that $S_{B_m}(r)
	=
	\{(i,j)\in B_m\times B_m:d_n(i,j)=r\}$. So the trivial bound $|S_{B_{m}}(r)| \leq |B_{m}|^{2} \lesssim q$ is always feasible. On th other hand, for each fixed \(i\in B_m\), the number of \(j\in B_m\) with \(d_n(i,j)=r\)	is at most \(|\partial B(i,r)|\). Hence
	\begin{align}\label{eqn:exponential.S.bound}
		|S_{B_m}(r)| \le
		\sum_{i\in B_m} |\partial B(i,r)|  \le
		|B_m|\, C e^{c r}.
	\end{align}
	Dividing by \(|B_m|\) gives	$\bar N_{B_m}(r)
	=
	\frac{|S_{B_m}(r)|}{|B_m|}
	\lesssim e^{c r}$.
	This proves the result.
    
\end{proof}

\begin{proof}[Proof of Proposition \ref{pro:maximal.exponential.growth}]
By Theorem \ref{thm:maximal}, for any partition with \(|B_m|\le q_n\)
	and  by the exponential-growth coloring bound, $K_{r_{n}}
	\lesssim
	q_n^2 e^{c r_n}$,
	\begin{align*}
		\left\|
		\sup_{f\in\mathcal F}\mathbb G_n(f)
		\right\|_{L^1(P)}
		\lesssim\;&
		q_n e^{c r_n/2}
		\max_{m\in[M_n]}\Gamma_p(B_m)
		\gamma_{2,a}
		\left(
		\mathcal F,
		\max_{i\in N_n}\|\cdot\|_{L^{2p}(P_{Z_i})}
		\right)
		\\
		&+
		\left(
		\frac{q_n^3 e^{c r_n}}{\sqrt n}
		\right)
		\gamma_{1,b}
		\left(
		\mathcal F,\|\cdot\|_\infty
		\right) + 	\sqrt n
		\bar{\tau}_{\mathcal F}
		\left( 	\mathcal P_{M_{n}}, 	\mathcal C_{r_{n}}[\mathcal P_{M_{n}} ] \right) .
	\end{align*}
	
	If \(q_n = q \asymp 1\), then \(q_n\), \(q_n^2\), and \(q_n^3\) are absorbed
	into the implicit constant. So, the previous display reduces to
{\small{	\begin{align*}
		\left\|
		\sup_{f\in\mathcal F}\mathbb G_n(f)
		\right\|_{L^1(P)}
		\lesssim\;&
		e^{c r_n/2}
		\max_{m\in[M_n]}\Gamma_p(B_m)
		\gamma_{2,a}
		\left(
		\mathcal F,
		\max_{i\in N_n}\|\cdot\|_{L^{2p}(P_{Z_i})}
		\right)+
		\left(
		\frac{e^{c r_n}}{\sqrt n}
		\right)
		\gamma_{1,b}
		\left(
		\mathcal F,\|\cdot\|_\infty
		\right) 	+
		\sqrt n\,
	\bar \tau_{\mathcal F}(r_n),
	\end{align*}}}
	where  $ 	\bar{\tau}_{ \mathcal F}
	\left( 	\mathcal P_{M_{n}}, 	\mathcal C_{r_{n}}[\mathcal P_{M_{n}} ] \right) = : 
\bar \tau_{\mathcal F}(r_n)$. 
	
	It remains to verify the stated bounds on \(\Gamma_p(B_m)\). By the
	exponential shell-growth condition, for every \(i\in N_n\), $|\partial B(i,s)| 	\lesssim	e^{c s}$. Therefore, by \eqref{eqn:exponential.S.bound}, for each block \(B_m\), $|S_{B_m}(s)| \lesssim |B_m|e^{c s}$.  Dividing by \(|B_m|\) gives $\bar N_{B_m}(s)	:=	\frac{|S_{B_m}(s)|}{|B_m|}	\lesssim	e^{c s}$.  On the other hand, the trivial bound $|S_{B_m}(s)|
	\le |B_m|^2$ implies $\bar N_{B_m}(s)
	\le |B_m|
	\lesssim q_n$. 	Combining the two bounds, $\bar N_{B_m}(s)
	\lesssim
	\min\{e^{c s},q_n\}$. 
	
	Also, from the proof of Lemma \ref{lem:coloring.exponential} is easy to see that $\operatorname{diam}(B_m)\lesssim \log q_n$. 	Hence, if \(S_{B_m}(s)\neq\varnothing\), then necessarily $s\le \operatorname{diam}(B_m)	\lesssim \log q_n$.  Thus $R(B_m)
	\subseteq
	\{s:s\le C\log q_n\}$ for some finite universal constant $C$. 
	Using the assumed shell representation of \(\Gamma_p(B_m)\),
	\begin{align*}
		\Gamma_p(B_m) &\lesssim
		1+
		\sum_{s\in R(B_m)}
		\bar N_{B_m}(s)
		\alpha_{\mathcal F,B_m}(s)^{1-1/p}
		\lesssim
		1+
		\sum_{s\le C\log q_n}
		\min\{e^{c s},q_n\}
		\alpha_{\mathcal F,B_m}(s)^{1-1/p} \\
		& \lesssim 
		1+
		\sum_{s\le C\log q_n} 	\min\{e^{c s},q_n\},
	\end{align*}
	where the last line follows from the fact that $\alpha_{ \mathcal{F} , B _m } \leq 1$. 
	
	Finally, if \(q_n\asymp 1\), then \(\log q_n=O(1)\) and
	\(\min\{e^{c s},q_n\}=O(1)\) uniformly over the finite range
	\(s\le C\log q_n\). Since \(\alpha_{\mathcal F,n}(s)\le 1\), it follows that
	\begin{align*}
		\max_{m\in[M_n]}\Gamma_p(B_m)
		\lesssim 1.
	\end{align*}
	This proves the proposition.
\end{proof}

\subsection{Directed dyadic graphs}

\begin{proof}[Proof of Lemma \ref{lem:directed-dyads-coloring-tau}]

	First, since \(N_n\) consists of all ordered pairs \((i,j)\) with \(i\neq j\), its cardinality is $|N_n|=n(n-1)$. 

    Fix any dyad \((i,j)\in N_n\). We now characterize $S((i,j),r)$ for any $r \geq 0$.  The shell of radius zero contains only the dyad itself, so $|S((i,j),0)|=1$. Next consider the dyads adjacent to \((i,j)\). A directed dyad \((k,\ell)\neq (i,j)\) is adjacent to \((i,j)\) if and only if it contains \(i\) or \(j\) as one of its endpoints. The number of directed dyads involving \(i\) is \(2(n-1)\): there are \(n-1\) dyads of the form \((i,k)\) and \(n-1\) dyads of the form \((k,i)\). Similarly, the number of directed dyads involving \(j\) is \(2(n-1)\). The dyads involving both \(i\) and \(j\) are precisely \((i,j)\) and \((j,i)\), and these have been counted twice. Therefore the number of dyads involving either \(i\) or \(j\) is $2(n-1)+2(n-1)-2 = 4n-6$. Thus, $|S((i,j),1)| \asymp n$. 

    We now show that $|S(i,j),r)| = 0$ for any $r > 2$  and $|S((i,j),2)| \asymp n^{2}$. To do this, observe that the graph \(G_n\) has diameter two.  Take any dyad \((k,\ell)\) that does not share an endpoint with \((i,j)\). Then \(k,\ell\notin\{i,j\}\), we can take the dyad \((i,k)\) that belongs to \(N_n\), and ``connects'' $(i,j)$ and $(k,\ell)$, thus  $d_n((i,j),(k,\ell))\leq 2$. Therefore, $|S((i,j),r)| = 0$ for any $r > 2$  and $|S((i,j),2)| = |N_{n}| - 1 - |S((i,j),1)|$. This last equality implies that $|S((i,j),2)| \asymp n^2$.

	It remains to prove the claim about \(\tau\). Let \(A,B\subseteq N_n\) satisfy $d_n(A,B)\geq 2$. 
	Then no dyad in \(A\) shares an endpoint with any dyad in \(B\). 
	
	Therefore $	Z_{A} : = \{Z_{ij}:(i,j)\in A\} $ and  $Z_{B} : = \{Z_{kl}:(k,l)\in B\}$ are two collections of primitive random variables that are disjoint and mutually independent. Hence $\sigma(Z_A)
	\indep
	\sigma(Z_B)$. 	Independence implies that the corresponding \(\tau\)-coefficient is zero. Taking the supremum over all such \(A,B\) with \(|B|\leq q\) yields
	\begin{align*}
		\tau_{\mathcal B}(r,q)=0
		\qquad
		\text{for all } r\geq 2.
	\end{align*}
	
	At distance one, the coefficient need not vanish. For example, \(Z_{ij}\) and \(Z_{ik}\) share the latent variable \(\mu_i\). Unless \(F\) is such that the shared latent variable has no effect, these two random variables are generally dependent. Therefore \(\tau_{\mathcal B}(1,q)\) is generally nonzero.
\end{proof}

The next lemma is used in the proof of Proposition \ref{prop:maximal-directed-dyads}.

\begin{lemma}
	\label{lem:gamma-directed-matching-blocks}
	Let	$N_n:=\{(i,j):1\leq i,j\leq n,\ i\neq j\}$, 
	and suppose $(Z_{ij})_{(i,j) \in N_{n}}$ is given by \eqref{eqn:Z.Fform}. Then, for any directed matching block  \(B\subseteq N_n\),\footnote{A directed matching block implies that for any two distinct dyads
		\((i,j),(k,\ell)\in B\), $\{i,j\}\cap\{k,\ell\}=\varnothing$.}
	\begin{align*}
		\Gamma^{\mathrm{sum}}_p(B) = 	\Gamma^\sharp_p(B) =1,~\forall p \in (1,\infty),
	\end{align*}
	and
	\begin{align*}
		\Gamma^{\mathrm{sum}}_1(B)
		=
		\Gamma^\sharp_1(B)
		=
		\sqrt{|B|},~for~p=1.
	\end{align*}
\end{lemma}

\begin{proof}[Proof of Lemma \ref{lem:gamma-directed-matching-blocks}]
	Fix a directed matching block \(B\), and write \(q_B:=|B|\). By construction, if
	\((i,j),(k,\ell)\in B\) are distinct, then $\{i,j\}\cap\{k,\ell\}=\varnothing$.  Therefore
	\begin{align*}
		Z_{ij}=F(\mu_i,\mu_j,\varepsilon_{ij})~and~
		Z_{k\ell}=F(\mu_k,\mu_\ell,\varepsilon_{k\ell})
	\end{align*}
	are functions of disjoint collections of primitive random variables. Since the primitive variables are mutually independent, the dyadic observations in \(B\) are mutually independent.
	
	Now consider the within-block distance profile inside \(B\).  Since $B$ is a directed matching block, any pairs $a$ and $b$ in $B$ are either the same, so $d_{n}(a,b) = 0$ or different --- so they don't share any endpoints --- and thus $d_{n}(a,b) = 2$. Thus, $R(B) : =  \{0,2\}$. 
	
	The diagonal shell is $S_B(0)=\{(a,a):a\in B\}$, so $|S_B(0)|=q_B$. For $r = 2$, since pairs are ordered, for each $a \in B$, there are $q_B-1$ distinct pairs, so $|S_B(2)|=q_B(q_B-1)$. 
	
	By convention, $\alpha_{\mathcal F,B}(0)=1$, and since for the off-diagonal shell the observations are independent, $\alpha_{\mathcal F,B}(2)=0$. 
	
	Therefore, for \(p\in(1,\infty]\), the definition of \(\Gamma^{\mathrm{sum}}_p(B)\) gives
	\begin{align*}
		\Gamma^{\mathrm{sum}}_p(B)^2
		&=
		\sum_{\rho\in R(B)}
		\frac{|S_B(\rho)|}{|B|}
		\left(
		\alpha_{\mathcal F,B}(\rho)
		\right)^{1-\frac1p} =
		\frac{|S_B(0)|}{q_B}
		\left(
		\alpha_{\mathcal F,B}(0)
		\right)^{1-\frac1p}
		+
		\frac{|S_B(2)|}{q_B}
		\left(
		\alpha_{\mathcal F,B}(2)
		\right)^{1-\frac1p} \\
		&=1.
	\end{align*}
	
	For \(p=1\), the definition of \(\Gamma^{\mathrm{sum}}_1(B)\) does not use the mixing coefficients. Hence
	\begin{align*}
		\Gamma^{\mathrm{sum}}_1(B)^2
		=
		\sum_{\rho\in R(B)}
		\frac{|S_B(\rho)|}{|B|} =
		\frac{|S_B(0)|}{q_B}
		+
		\frac{|S_B(2)|}{q_B} =
		1+(q_B-1) =q_B.
	\end{align*}
	
	We next consider \(\Gamma^\sharp_p(B)\). For \(p\in(1,\infty)\), define
	\begin{align*}
		A_B(u)
		:=
		\sum_{\rho\in R(B)}
		\frac{|S_B(\rho)|}{|B|}
		1\{\alpha_{\mathcal F,B}(\rho)\geq u\}.
	\end{align*}
	Using $\alpha_{\mathcal F,B}(0)=1$ and $\alpha_{\mathcal F,B}(2)=0$, we obtain, for \(u\in(0,1]\), $A_B(u)
	=
	\frac{|S_B(0)|}{q_B}1\{1\geq u\}
	+
	\frac{|S_B(2)|}{q_B}1\{0\geq u\} 
	=
	1$. 
	Therefore $\int_0^1
	\left[
	A_B(u)
	\right]^{\frac{p}{p-1}}
	du
	=
	1$ and	consequently,
	\begin{align*}
		\Gamma^\sharp_p(B)
		=
		\sqrt{
			\left(
			\int_0^1
			\left[
			A_B(u)
			\right]^{\frac{p}{p-1}}
			du
			\right)^{1-\frac1p}
		} =1
	\end{align*}
	for all  $p\in(1,\infty)$.
	
	Finally, for \(p=1\), the displayed definition of \(\Gamma^\sharp_1(B)\) reduces to the same shell-counting expression which implies that $\Gamma^\sharp_1(B)=\sqrt{q_B}$. 
\end{proof}

\begin{proof}[Proof of Proposition \ref{prop:maximal-directed-dyads}]
	
	We apply Theorem \ref{thm:maximal} with an specific partition and coloring. For this construction, we assume that \(n\) is even --- this restriction is not essential, but it simplifies the exposition. 
	
	Let
	\((\mathcal A_n,\mathcal E_n)\) denote the complete graph of units, where
	\begin{align*}
		\mathcal A_n := [n],
		\qquad
		\mathcal E_n := \big\{\{i,j\}:1\leq i<j\leq n\big\}.
	\end{align*}
	Since \(n\) is even, there exist perfect matchings
	\(\mathcal M_1,\ldots,\mathcal M_{n-1}\) such that $	\mathcal E_n
	=
	\bigcup_{s=1}^{n-1}\mathcal M_s$, and $ \mathcal M_s\cap\mathcal M_{s'}=\varnothing$ for any $s\neq s'$, and each \(\mathcal M_s\) contains \(n/2\) edges with no shared endpoints. That is, for every \(s\), any $\{i,j\},\{k,\ell\}\in \mathcal M_s$ such that $\{i,j\}\neq \{k,\ell\}$ then $
	\{i,j\}\cap\{k,\ell\}=\varnothing$. 
	
	For each perfect matching \(\mathcal M_s\), define two directed blocks:
	\begin{align*}
		B_s^{+}
		:=
		\big\{(i,j)\in N_n:\{i,j\}\in\mathcal M_s,\ i<j\big\},~and~B_s^{-}
		:=
		\big\{(j,i)\in N_n:\{i,j\}\in\mathcal M_s,\ i<j\big\}.
	\end{align*}
	These sets partition $N_n$, and each is a directed matching: distinct dyads in the same set share no endpoints.
	
	Take the singleton partition $\mathcal P_{M_n}
	:=
	\bigl\{\{(i,j)\}:(i,j)\in N_n\bigr\}$ with $M_n=n(n-1)$. Indexing the singleton blocks by their dyads, define the color classes
	\[
	C_{2s-1}:=B_s^+,
	\qquad
	C_{2s}:=B_s^-,
	\qquad s=1,\ldots,n-1.
	\]
	Distinct dyads in the same color class share no endpoints and
	therefore have graph distance two. Thus, this construction defines
	a proper $1$-coloring with $K_1\le 2(n-1) $ and $\max_{B\in\mathcal P_{M_n}}|B|=1$. Moreover, the observations within each color class are mutually
	independent, since they depend on disjoint collections of primitive
	random variables. Consequently, $\bar\tau_{\mathcal F}
	\bigl(\mathcal P_{M_n},\mathcal C_1[\mathcal P_{M_n}]\bigr)=0$. 
	Since each block is a singleton, $\max_{B\in\mathcal P_{M_n}}\Gamma_p(B)=1$.

	Therefore, by Theorem \ref{thm:maximal}, observing that $|N_n|=n(n-1)$,
	\begin{align*}
		\left\|
		\sup_{f\in\mathcal F}|G_n(f)|
		\right\|_{L^1(P)}
		&\lesssim
		\sqrt{2(n-1)}\,
		\gamma_{2,a}\left(
		\mathcal F,
		\max_{(i,j)\in N_n}\|\cdot\|_{L^{2p}(P_{Z_{ij}})}
		\right)+
		\frac{2(n-1)}{\sqrt{n(n-1)}}\,
		\gamma_{1,b}(\mathcal F,\|\cdot\|_\infty)
		\\
		&\lesssim
		\sqrt n\,
		\gamma_{2,a}(\mathcal F,\|\cdot\|_{L^{2p}(P)})
		+
		\gamma_{1,b}(\mathcal F,\|\cdot\|_\infty),
	\end{align*}
	where the last inequality uses the common marginal distribution
	$P_{Z_{ij}}=P$.
\end{proof}

\section{Appendix for Glivenko--Cantelli results}
\label{app:GC}

The following lemma is used in the proof of Theorem \ref{thm:single-scale-network-gc}.

\begin{lemma}
	\label{lem:finite-class-gamma-bounds}
	Let \((T,d)\) be a finite pseudo-metric space and suppose that, for some
	\(k\ge 0\), $|T|
	\le
	2^{2^k}$. Then, for every \(p,r>0\),
	\begin{align*}
		\gamma_{p,r}(T,d)
		&\lesssim  2^{\lceil k r \rceil / p }\operatorname{diam}(T,d),
	\end{align*}
	where $\operatorname{diam}(T,d)
	:=
	\sup_{t,t'\in T} d(t,t')$. 
\end{lemma}

\begin{proof}[Proof of Lemma \ref{lem:finite-class-gamma-bounds}]
	Let $D := \operatorname{diam}(T,d)$. By definition \ref{def:talagrand}
	\begin{align*}
		\gamma_{p,r}(T,d)
		:=
		\inf_{T^\infty \in \mathcal T_r(T)}
		\sup_{t \in T}
		\sum_{l=0}^{\infty}
		2^{l/p}\, d\big( D(t,T_l) \big).
	\end{align*}
	
	Let \(\ell_k\) be the smallest integer such that $2^{2^{\ell_k/r}}	\ge	|T|$.  Since \(|T|\le 2^{2^k}\), it is enough that $2^{\ell_{k}/r} \geq 2^{k} \iff \ell_{k} =  \lceil  r k  \rceil $. We now construct an admissible partition sequence:  For
	\(l<\ell_k\), set $\mathcal T_l	:=	\{T\}$,  the trivial partition. For \(l \ge \ell_k\), set \(\mathcal T_l\) equal to the discrete partition of \(T\) into singletons. This sequence is admissible by the definition of \(\ell_k\). Under this partition sequence, for \(l<\ell_k\), $d\big( D(t,T_l) \big) \leq D$, while for \(l \ge\ell_k\), $d\big( D(t,T_l) \big) = 0$. Therefore, 
	\begin{align*}
		\sup_{t \in T}
		\sum_{l=0}^{\infty}
		2^{l/p}\, d\big( D(t,T_l) \big) = 	\sup_{t \in T}
		\sum_{l=0}^{\ell_{k} - 1}
		2^{l/p}\, d\big( D(t,T_l) \big) \leq D 	\sum_{l=0}^{\ell_{k} - 1}
		2^{l/p} \lesssim D 2^{\ell_{k}/p}.
	\end{align*}
	This completes the proof.
\end{proof}
\section{Some useful concepts related to graphs}
\label{app:graph}

We introduce some graph-related concepts that will be used throughout the paper. Let $\mathfrak{G} =(N,V)$ be a graph endowed with the graph distance $d$. For any $M \in \mathbb{N}$, let $\mathcal P_{M} : = \{B_1,\ldots,B_M\}$ be a partition of $N$. For two cells $B_m,B_{m'}\in\mathcal P_{M}$, define their distance by
\begin{align*}
	d(B_m,B_{m'})
	:=
	\min_{i\in B_m;j\in B_{m'}} d(i,j).
\end{align*}

For any graph $\mathfrak{G} =(N,V)$, let $\mathcal N_{\mathfrak{G}}(x)
	:=
	\{ x'\in N \setminus\{x\}: (x,x')\in V \}$ 
denote the neighborhood of $x$ in the graph, and let
\begin{align*}
	\deg_{\mathfrak{G}}(x)
	:=
	|\mathcal N_{\mathfrak{G}}(x)|
\end{align*}
denote its degree. The maximum degree of the graph is
\begin{align*}
	\Delta(\mathfrak{G})
	:=
	\max_{x \in N }
	\deg_{\mathfrak{G}}(x).
\end{align*}

For $r\ge 0$, the $r$-conflict graph associated with $\mathcal P_{M}$ is the graph $H_{r}[\mathcal P_{M}]
:=
([M],E_{r})$, 
where
\begin{align*}
	\{m,m'\}\in E_{r}
	\Longleftrightarrow
	m\neq m'
	\ \text{ and }
	d(B_m,B_{m'})\le r.
\end{align*}
Thus, the vertices of $H_{r}[\mathcal P_{M}]$ correspond to cells of the partition, and two vertices are adjacent whenever the corresponding cells are within graph distance $r$.

A proper $(r,K)$-coloring of $\mathfrak G$ associated to $\mathcal P_{M}$ is a map $c \colon [M] \to [K] $  such that if $c(m)=c(m')$ and $m\neq m'$, then 
$d(B_m,B_{m'})>r$. The sets
\begin{align*}
	C_{k} : = c^{-1}(k)  = \{m \in [M] :c(m)=k \},~\forall k\in [K],
\end{align*}
are referred to as the color classes. In particular, 
\begin{align*}
	\forall m, m' \in C_{k} ~such~that~m \ne m' \Longrightarrow d(B_{m},B_{m'}) > r,
\end{align*}
i.e., any two distinct cells whose indices belong to the same color class are separated by a distance strictly greater than $r$.

The $r$-chromatic number of a graph $\mathfrak G$ associated to $\mathcal P_{M}$, denoted by $\chi_{r}(\mathcal P_{M})$, is the smallest $K$ for which a proper $(r,K)$-coloring exists.

By the standard greedy-coloring bound,
\begin{align}\label{eqn:color.greedy.bound.0}
\chi_{r}\bigl(\mathcal P_{M}\bigr)
\leq
\Delta\bigl(H_{r}[\mathcal P_{M}]\bigr)+1.
\end{align}
Indeed, coloring the vertices sequentially, each vertex has at most $\Delta\bigl(H_r[\mathcal P_M]\bigr)$ neighbors and therefore at most $\Delta\bigl(H_r[\mathcal P_M]\bigr)$ forbidden colors, so one of $\Delta\bigl(H_r[\mathcal P_M]\bigr)+1$ colors is always available.

It is easy to see that an $(r,\chi_{r}(\mathcal P_{M}))$-coloring coincides with an $r$-coloring in Definition \ref{def:color} in the text. Hence, $K_{r} = \chi_{r}(\mathcal P_{M})$ and thus 
\begin{align}\label{eqn:color.greedy.bound}
K_{r} = K_{r}\bigl(\mathcal P_{M}\bigr) 
\leq
\Delta\bigl(H_{r}[\mathcal P_{M}]\bigr)+1.
\end{align}

	\end{document}